\documentclass[11pt]{amsart}
\usepackage[all]{xy}
\usepackage{yfonts}
\usepackage{amsmath}
\usepackage{amsfonts}
\usepackage{amssymb}
\usepackage{verbatim}
\usepackage{amsthm}
\usepackage{amscd}
\usepackage{amsmath}
\usepackage{graphicx,xcolor}
\usepackage{mathrsfs}
\usepackage{upgreek}
\usepackage{tikz}
\tikzset{dynkdot/.style={circle,draw,scale=.38}}
\definecolor{refkey}{gray}{0.5}
\definecolor{labelkey}{gray}{0.5}
\usepackage[
bookmarks=false,
colorlinks=true,
debug=true,
naturalnames=true,
pdfnewwindow=true,
citecolor=blue,
linkcolor=blue,
urlcolor = blue]{hyperref}
\usepackage{todonotes}
\usepackage{mathtools}
\newcommand{\seq}{\coloneqq}

\numberwithin{equation}{section}

\newcommand{\arxiv}[1]{\href{http://arxiv.org/abs/#1}{\texttt{arXiv:#1}}}
\newtheorem{Thm}{Theorem}[section]
\newtheorem{Cor}[Thm]{Corollary}
\newtheorem{Prop}[Thm]{Proposition}
\newtheorem{Lem}[Thm]{Lemma}
\newtheorem{Conj}[Thm]{Conjecture}

\theoremstyle{definition}
\newtheorem{Def}[Thm]{Definition}
\newtheorem{Rem}[Thm]{Remark}
\newtheorem{Ex}[Thm]{Example}
\newcommand{\ul}{\underline}
\newcommand{\ol}{\overline}

\newcommand{\im}{\imath}
\newcommand{\jm}{\jmath}

\newcommand{\Aut}{\mathop{\mathrm{Aut}}\nolimits}

\newcommand{\Irr}{\mathop{\mathrm{Irr}}\nolimits}
\newcommand{\res}{\mathop{\mathrm{res}}\nolimits}
\newcommand{\Rep}{\mathop{\mathrm{Rep}}}

\newcommand{\id}{\mathrm{id}}
\newcommand{\CL}{\mathrm{CL}}
\newcommand{\evt}{\mathrm{ev}_{t=1}}
\newcommand{\evv}{\mathrm{ev}_{v=1}}
\newcommand{\wt}{\mathrm{wt}}
\newcommand{\zero}{\mathrm{zero}}

\newcommand{\Z}{\mathbb{Z}}
\newcommand{\Q}{\mathbb{Q}}
\newcommand{\C}{\mathbb{C}}
\newcommand{\F}{\mathbb{F}}
\newcommand{\kk}{\Bbbk}

\newcommand{\sg}{\mathsf{g}}
\newcommand{\sn}{\mathsf{n}}
\newcommand{\sP}{\mathsf{P}}
\newcommand{\sQ}{\mathsf{Q}}
\newcommand{\sR}{\mathsf{R}}
\newcommand{\sW}{\mathsf{W}}

\newcommand{\cI}{\mathcal{I}}
\newcommand{\cQ}{\mathcal{Q}}
\newcommand{\cD}{\mathcal{D}}
\newcommand{\cM}{\mathcal{M}}
\newcommand{\cY}{\mathcal{Y}}
\newcommand{\cT}{\mathcal{T}}
\newcommand{\cK}{\mathcal{K}}
\newcommand{\cA}{\mathcal{A}}
\newcommand{\cU}{\mathcal{U}}
\newcommand{\cJ}{\mathcal{J}}
\newcommand{\cL}{\mathcal{L}}

\newcommand{\Cc}{\mathscr{C}}
\newcommand{\Nn}{\mathscr{N}}
\newcommand{\Aa}{\mathscr{A}}

\newcommand{\fg}{\mathfrak{g}}
\newcommand{\fD}{\mathfrak{D}}

\newcommand{\fd}{\textfrak{d}}

\newcommand{\bfi}{\boldsymbol{i}}

\newcommand{\bfa}{\boldsymbol{a}}

\newcommand{\bfc}{\boldsymbol{c}}

\newcommand{\tvee}{\widetilde{\vee}}
\newcommand{\tc}{\widetilde{c}}
\newcommand{\tm}{\widetilde{m}}
\newcommand{\tC}{\widetilde{C}}
\newcommand{\tY}{\widetilde{Y}}
\newcommand{\tG}{\widetilde{G}}
\newcommand{\tD}{\widetilde{D}}
\newcommand{\tB}{\widetilde{B}}
\newcommand{\tF}{\widetilde{F}}
\newcommand{\tA}{\widetilde{A}}
\newcommand{\tcA}{\widetilde{\cA}}
\newcommand{\tpsi}{\widetilde{\psi}}

\newcommand{\hDs}{\widehat{\Delta}^{\sigma}}
\newcommand{\hI}{\widehat{I}}

\newcommand{\bphi}{\bar{\phi}}

\newcommand{\tbfB}{\widetilde{\mathbf{B}}^*}
\newcommand{\bfB}{\mathbf{B}^*}

\title[Isomorphisms among quantum Grothendieck rings]
{Isomorphisms among quantum Grothendieck rings and propagation of positivity}

\date{\today}

\author[R.~Fujita]{Ryo Fujita}
\address[R.~Fujita]{Research Institute for Mathematical Sciences, Kyoto University, Oiwake-Kitashirakawa, Sakyo, Kyoto, 606-8502, Japan \& Institut de Math\'{e}matiques de Jussieu-Paris Rive Gauche, Universit\'{e} de Paris, F-75013, Paris, France}
\email{rfujita@kurims.kyoto-u.ac.jp}

\author[D.~Hernandez]{David Hernandez}
\address[D.~Hernandez]{Universit\'{e} de Paris and Sorbonne Universit\'{e}, CNRS, IMJ-PRG, IUF, F-75006, Paris, France}
\email{david.hernandez@u-paris.fr}

\author[S.-j.~Oh]{Se-jin Oh}
\address[S.-j.~Oh]{Ewha Womans University Seoul, 52 Ewhayeodae-gil, Daehyeon-dong, Seodaemun-gu, Seoul, South Korea}
\email{sejin092@gmail.com}

\author[H.~Oya]{Hironori Oya}
\address[H.~Oya]{Department of Mathematical Sciences, Shibaura Institute of Technology, 307 Fukasaku, Minuma-ku, Saitama-shi, Saitama, 337-8570, Japan}
\email{hoya@shibaura-it.ac.jp}


\begin{document} 

\maketitle

\begin{abstract}
Let ($\fg,\sg)$ be a pair of complex finite-dimensional simple Lie algebras whose 
Dynkin diagrams are related by (un)folding, with $\sg$ being of simply-laced type.
We construct a collection of ring isomorphisms 
between the quantum Grothendieck rings 
of monoidal categories $\Cc_{\fg}$ and $\Cc_{\sg}$ of 
finite-dimensional representations over the quantum loop algebras of $\fg$ and  $\sg$ respectively.
As a consequence, we solve long-standing problems :  the positivity of the analogs of Kazhdan-Lusztig polynomials
and the positivity of the structure constants of the quantum Grothendieck rings for any non-simply-laced $\fg$. In addition, comparing our isomorphisms with the categorical relations 
arising from the generalized quantum affine Schur-Weyl dualities, we prove the analog of Kazhdan-Lusztig conjecture (formulated in \cite{Hernandez04}) for simple modules in remarkable monoidal subcategories of $\Cc_{\fg}$ for any non-simply-laced $\fg$, and for 
any simple finite-dimensional modules in $\Cc_{\fg}$ for $\fg$ of type $\mathrm{B}_n$.
In the course of the proof we obtain and combine several new ingredients. In particular we establish a quantum analog of $T$-systems,
and also we generalize the isomorphisms of \cite{HL15, HO19} to all $\fg$ in a unified way, that is isomorphisms between subalgebras of the quantum group of $\sg$ and subalgebras of the quantum Grothendieck ring of $\Cc_\fg$. 
\end{abstract}

\tableofcontents

\section{Introduction}

\subsection{}
For a complex finite-dimensional simple Lie algebra $\fg$,
the associated quantum loop algebra $U_{q}(L\fg)$ 
serves as a natural quantum affinization of $\fg$.
It can be obtained as a subquotient of the quantum affine algebra $U_q(\widehat{\fg})$
corresponding to the level zero representations and hence inherits a structure of Hopf algebra over $\ol{\Q(q)}$. 
In particular, finite-dimensional representations over $U_{q}(L\fg)$
form a rigid monoidal category $\Cc_{\fg}$, whose structure is quite intricate.
For example, in contrast to the classical finite-dimensional representation theory of $\fg$, 
it is neither semisimple as an abelian category nor braided as a monoidal category.
Since the naissance of quantum groups in the middle of the 80s in the context of solvable lattice models in physics,     
the category $\Cc_\fg$ has been intensively studied from various perspectives.

From the representation-theoretic point of view, many fundamental results were obtained in the 90s.
Chari-Pressley~\cite{CP, CP95} gave 
a complete classification of simple representations in $\Cc_\fg$ in terms of the Drinfel'd polynomials. 
After that, motivated by the theory of deformed $\mathcal{W}$-algebras, 
Frenkel-Reshetikhin~\cite{FR98} introduced the notion of $q$-characters of representations in $\Cc_{\fg}$ 
as a natural quantum loop analog of the usual characters for $\fg$.
Their basic properties were studied in a subsequent work by Frenkel-Mukhin~\cite{FM01}.
These results particularly imply that the Grothendieck ring $K(\Cc_\fg)$ is a commutative polynomial ring
generated by the classes of \emph{fundamental modules},
whose $q$-characters are computable by Frenkel-Mukhin's combinatorial algorithm. 
These fundamental modules also have crystal bases \cite{Kas02} (up to a spectral parameter shift). Moreover, every simple module $L(m)$ in $\Cc_\fg$ can be obtained as a 
head of a unique ordered tensor product of fundamental modules, 
called \emph{the standard module} $M(m)$.
Here $m$ belongs to an ordered set $\cM$ of \emph{dominant monomials} which serves as a parameter set for simple modules in $\Cc_{\fg}$
(which is comparable with the set of Drinfel'd polynomials). 

In this paper, we are particularly interested in the next step : the fundamental problem to determine 
the dimensions and $q$-characters of simple modules.
One possible approach is to compute the $q$-characters of simple modules from those of standard modules
with \emph{a Kazhdan-Lusztig type algorithm} explained below.


\subsection{}
When $\fg$ is simply-laced (i.e.,~of type $\mathrm{ADE}$), 
a remarkable Kazhdan-Lusztig type algorithm was established by Nakajima based on the geometry of quiver varieties~\cite{Nakajima01q, Nakajima04}. 
This algorithm computes the Jordan-H\"{o}lder multiplicity 
$P_{m,m'}$ of the simple module $L(m')$ occurring in the standard module $M(m)$.
Since we have
$$
[M(m)] = [L(m)] + \sum_{m' \in \cM\colon m' < m} P_{m,m'} [L(m')] 
$$ 
in the Grothendieck ring $K(\Cc_{\fg})$, this algorithm enables us to compute all the simple $q$-characters in principle.  
It returns the multiplicity $P_{m,m'}$ as the specialization at $t=1$ of a certain positive polynomial $P_{m,m'}(t) \in t\Z_{\ge 0}[t]$,
which can be seen as an analog of Kazhdan-Lusztig polynomial. 
These polynomials $P_{m,m'}(t)$ are constructed from the \emph{quantum Grothendieck ring} $\cK_t(\Cc_{\fg})$ 
which is a $t$-deformation of the Grothendieck ring $K(\Cc_{\fg})$~\cite{Nakajima04, VV03}. 
Here we emphasize that the quiver varieties play an essential role to guarantee the validity of the algorithm $P_{m,m'}(1) = P_{m,m'}$ and the positivity of $P_{m,m'}(t)$.

When $\fg$ is non-simply-laced, the above theory is not applicable
for the absence of a fully developed theory of quiver varieties.  
However, one can still formulate a conjectural Kazhdan-Lusztig type algorithm for general $\fg$ as it is possible to construct the quantum Grothendieck ring $\cK_{t}(\Cc_{\fg})$ in a purely algebraic way. 
This is what the second named author established in \cite{Hthese, Hernandez04}. 
In this algebraic setting, a certain quantum torus $\cY_t$ is obtained from a different construction involving formal power series of elements in the quantum affine algebra which can be seen as vertex operators. Then, mimicking 
\cite{Nakajima04, VV03}, the quantum Grothendieck ring $\cK_{t}(\Cc_{\fg})$ is defined inside $\cY_t$ 
as a $\Z[t^{\pm 1/2}]$-subalgebra whose specialization at $t=1$ is identical to $K(\Cc_\fg)$ via the $q$-character map 
(the purely algebraic proof of the existence of $\cK_{t}(\Cc_{\fg})$ in \cite{Hernandez04} works in non simply-laced types as well). By the same arguments as in \cite{Nakajima04}, $\cK_{t}(\Cc_{\fg})$ has a standard $\Z[t^{\pm 1/2}]$-basis $\{ E_t(m) \}_{m \in \cM}$, whose members are called 
\emph{the $(q,t)$-characters of standard modules}. 
Under the specialization $\cK_t(\Cc_\fg) \to K(\Cc_{\fg})$,
the element $E_t(m)$ corresponds to the class $[M(m)]$ and hence the name. 
As an analog of the Kazhdan-Lusztig basis,
we can construct a canonical $\Z[t^{\pm 1/2}]$-basis $\{ L_t(m) \}_{m \in \cM}$ of $\cK_t(\Cc_\fg)$
characterized by the invariance under a standard bar-involution and the unitriangular property   
$$
E_t(m) = L_t(m) + \sum_{m' \in \cM\colon m' < m} P_{m,m'}(t) L_t(m')
$$ 
for some $P_{m,m'}(t) \in t\Z[t]$. The polynomials $P_{m,m'}(t)$ are nothing but the desired analogs of Kazhdan-Lusztig polynomials.

\subsection{}
With the above construction, the following conjectures naturally arise.

\begin{Conj}[Analog of Kazhdan-Lusztig conjecture, {\cite[Conjecture 7.3]{Hernandez04}}] \label{Conj0:KL}
Under the specialization $\cK_t(\Cc_\fg) \to K(\Cc_{\fg})$, the element $L_t(m)$ corresponds to 
the class $[L(m)]$ for any $m \in \cM$. In particular, we have $P_{m,m'}(1) = P_{m,m'}$ for any $m,m' \in \cM$.   
\end{Conj}

\begin{Conj}[Positivities in the quantum Grothendieck ring] \label{Conj0:p}\hfill
\begin{enumerate}
\renewcommand{\theenumi}{\rm P\arabic{enumi}}
\item \label{p:KL} The polynomial $P_{m,m'}(t)$ has non-negative coefficients, i.e., $P_{m,m'}(t) \in t \Z_{\ge 0}[t]$. 
\item \label{p:sc} The structure constants of $\cK_{t}(\Cc_{\fg})$ with respect to the canonical basis $\{ L_t(m) \}_{m \in \cM}$ have non-negative coefficients.
\item \label{p:ch} In the quantum torus $\cY_t$, the elements $L_t(m)$ have non-negative coefficients. 
\end{enumerate}
\end{Conj}

Motivated by Conjecture~\ref{Conj0:KL}, we refer to the element $L_t(m)$ as \emph{the $(q,t)$-character of the simple module} $L(m)$. 
Since they can be computed algorithmically as in the usual Kazhdan-Lusztig theory,
Conjecture~\ref{Conj0:KL} enables us to compute all the simple $q$-characters algorithmically once it is verified. 
{\rm Conjecture~\ref{Conj0:p}~(\ref{p:KL})} was checked in several small examples and expected generally in  \cite{Hthese}.
 As mentioned before,
when $\fg$ is simply-laced, all of these conjectures are known to be true by \cite{Nakajima04, VV03} using the geometry of quiver varieties.
On the other hand,  they are still unsolved for non-simply-laced $\fg$.

In this paper, using the new ingredients that we explain below, we improve this situation with the following results.

\begin{Thm}[= Corollary~\ref{Cor:pos_st} + Corollary~\ref{Cor:PsiAB} + Theorem~\ref{Thm:posqtch}]\hfill \label{Thm0:main}
\begin{enumerate}
\item For general $\fg$, {\rm Conjecture~\ref{Conj0:p}~(\ref{p:KL})} and {\rm (\ref{p:sc})} hold.
\item If $\fg$ is of type $\mathrm{B}$, {\rm Conjecture~\ref{Conj0:KL}} and {\rm Conjecture~\ref{Conj0:p}~(\ref{p:ch})} hold as well.
\end{enumerate}
\end{Thm}


Our proof of Theorem~\ref{Thm0:main} relies in part on the isomorphisms established in this paper between quantum Grothendieck rings for non-simply-laced quantum loop algebras and their unfolded simply-laced ones. 

Note that in their recent study of the non-symmetric quantized Coulomb branches  \cite{NW19, NakajimaQCB}, Nakajima and Weekes established 
a parametrization of simple 
modules in non-simply-laced types in terms of simple modules in simply-laced types and proved that $q$-characters of simple modules of non-symmetric Kac-Moody algebras are obtained from those of unfolded symmetric ones by imposing a certain modulus condition. 
In particular, they solved the character problem for simple modules by a different approach.
Although they also used a connection between non-symmetric and symmetric quantum affine algebras, it seems to be different from the one we used in this paper.  
      

\subsection{}
In what follows, we outline our proof of the above main results. 
First of all, we notice that the description of the (quantum) Grothendieck ring of the whole category $\Cc_\fg$ can be reduced
to that of a certain rigid monoidal subcategory $\Cc_{\Z, \fg} \subset \Cc_{\fg}$ introduced in \cite{HL10}
(see Section~\ref{ssec:CZ}). 
Roughly speaking, $\Cc_{\Z, \fg}$ is a Serre subcategory of $\Cc_{\fg}$ generated by the simple modules 
whose Drinfel'd polynomials have their roots at some integral powers of $q$ only.
From now on, we focus on this skeleton subcategory $\Cc_{\Z, \fg}$.\footnote{In the main body of this paper, 
we only consider the quantum Grothendieck ring of the subcategory $\Cc_{\Z, \fg} \subset \Cc_{\fg}$ from the beginning.}

Now we assume that $\fg$ is of non-simply-laced type.
Then we choose another simple Lie algebra $\sg$ of simply-laced type whose Dynkin diagram is obtained by \emph{unfolding} the Dynkin diagram of $\fg$.  
In other words, we fix a pair $(\fg, \sg)$ of simple Lie algebras whose Dynkin types are either $(\mathrm{B}_n, \mathrm{A}_{2n-1})$, 
$(\mathrm{C}_n, \mathrm{D}_{n+1})$, $(\mathrm{F}_{4}, \mathrm{E_6})$ or $(\mathrm{G}_2, \mathrm{D}_4)$.
To prove the positivities (\ref{p:KL}) and (\ref{p:sc}) for non-simply-laced $\fg$, we establish the following.
\begin{Thm}[$\doteq$ Theorem~\ref{Thm:Psi}] \label{Thm0:CZisom}
There exists an isomorphism of $\Z[t^{\pm 1/2}]$-algebras
\begin{equation} \label{eq:CZisom}
\cK_{t}(\Cc_{\Z, \sg}) \simeq \cK_{t}(\Cc_{\Z, \fg})
\end{equation}
which extends an isomorphism \eqref{eq:CQisom2} explained below and which induces a bijection between the $(q,t)$-characters of simple modules.
\end{Thm}

Once we establish an isomorphism (\ref{eq:CZisom}), 
it propagates the known positivity in $\cK_{t}(\Cc_{\Z, \sg})$ to the desired positivity in $\cK_{t}(\Cc_{\Z, \fg})$.
We note that such an isomorphism is not unique. 
In fact, we construct a collection of such isomorphisms (\ref{eq:CZisom}) labeled by certain combinatorial objects called \emph{Q-data}, 
which we explain later in \ref{Intro:CQ} (they lead to the same positivity results).   
The isomorphisms (\ref{eq:CQisom2}) concern subcategories and already depend on the choice of Q-data (see Theorem \ref{Thm0:HLO} below). 
Also we note that our isomorphisms are not deduced from the results in \cite{NW19, NakajimaQCB}.

The specialization at $t=1$ of the isomorphism (\ref{eq:CZisom}) yields an isomorphism between the usual Grothendieck rings
$$
K(\Cc_{\Z, \sg}) \simeq K(\Cc_{\Z, \fg}).
$$
We note that it is still non-trivial. For example, it does not respect the classes of fundamental modules nor of standard modules.
We conjecture that it induces a bijection between the classes of simple modules.
In this paper, we prove this conjecture when $\fg$ is of type $\mathrm{B}$ with a specific choice of Q-data.
Then Conjecture~\ref{Conj0:KL} for type $\mathrm{B}$ is obtained as a corollary. See~\ref{Intro:B}.
  
\subsection{}
Our main step to construct an isomorphism (\ref{eq:CZisom}) is to compare a presentation of the \emph{localized} quantum Grothendieck ring\footnote{
In the main body of the paper, it is denoted by $\cK_{\Q(t^{1/2})}$.}
$$
\cK_{t}(\Cc_{\Z, \fg})_{loc} \seq \cK_{t}(\Cc_{\Z, \fg}) \otimes_{\Z[t^{\pm 1/2}]} \Q(t^{1/2})
$$
of $\Cc_{\Z, \fg}$ 
with that of $\Cc_{\Z, \sg}$.

A presentation of $\cK_{t}(\Cc_{\Z, \sg})_{loc}$ for simply-laced $\sg$
was previously obtained by B.~Leclerc and the second named author in \cite{HL15},
which revealed that $\cK_{t}(\Cc_{\Z, \sg})_{loc}$ is isomorphic to the derived Hall algebra of the category $\Rep Q$ of representations of a Dynkin quiver $Q$ of the same type as $\sg$.  
Under this isomorphism, the cohomological degree shift $[1]$ in the derived category $D^b(\Rep Q)$ corresponds to 
the left dual functor $\cD$ in the rigid monoidal category $\Cc_{\Z, \sg}$. 
If we restrict it to the subalgebras corresponding to the heart $\Rep Q$ of the standard $t$-structure in $D^b(\Rep Q)$, 
we get an isomorphism 
\begin{equation} \label{eq:CQisom}
\cK_t(\Cc_Q) \simeq \cA_{t}[N_-],
\end{equation}
where $\cA_{t}[N_-]$ is the quantum unipotent coordinate algebra associated with $\sg$, and
$\Cc_{Q}$ is a certain monoidal subcategory of $ \Cc_{\Z, \sg}$ defined by using
the Auslander-Reiten (AR) quiver $\Gamma_Q$ of $\Rep Q$.
Moreover, the isomorphism (\ref{eq:CQisom}) sends the basis of $\cK_t(\Cc_Q)$ formed by the simple $(q,t)$-characters 
to the dual canonical basis of $\cA_{t}[N_-]$. 

Roughly speaking, the proof in~\cite{HL15} 
consists of the following two steps: 
\begin{itemize}
\item[(i)] establish the isomorphism (\ref{eq:CQisom}),
\item[(ii)] investigate the relations among the subalgebras 
$\cK_{t}(\cD^{k}\Cc_{Q}) \subset \cK_{t}(\Cc_{\Z, \sg})$ for $k \in \Z$.
\end{itemize}
Note that the subcategory $\cD^{k}\Cc_{Q} \subset \Cc_{\Z, \sg}$ corresponds to 
the shifted heart $(\Rep Q)[k]$ in $D^b(\Rep Q)$
and in particular we have $\cK_t(\cD^k\Cc_Q) \simeq \cK_t(\Cc_Q)$ as $\Z[t^{\pm 1/2}]$-algebras. 
As a result, we find two kinds of relations: \emph{the quantum Serre relations} from (i) and \emph{the quantum Boson relations} from (ii),
which yields the desired presentation of $\cK_{t}(\Cc_{\Z, \sg})_{loc}$.   

In this paper, we generalize the above story in a unified way including non-simply-laced cases.
As a result, we find that the localized Grothendieck rings $\cK_{t}(\Cc_{\Z, \fg})_{loc}$ and $\cK_{t}(\Cc_{\Z, \sg})_{loc}$
share the same presentation, 
which is finally enhanced to the desired isomorphism (\ref{eq:CZisom}). 
Our proof also consists of two steps analogous to (i) and (ii) above.  

\subsection{} \label{Intro:CQ}
To generalize the step (i), we need to extend the definition of the subcategory $\Cc_Q$ to the general case.  
Such a generalization was considered by the third named author and his collaborators in~\cite{KO17, OS19b, OhS19}
using the \emph{twisted AR quiver} $\Gamma_{\cQ}$ associated with a \emph{Q-datum} $\cQ$ for $\fg$,
instead of the AR quiver $\Gamma_{Q}$ associated with a Dynkin quiver $Q$ in the last paragraph. 
The twisted AR quiver is a quiver encoding some combinatorial information 
of a certain commutation class in the Weyl group of $\sg$ (not $\fg$) 
called a \emph{twisted adapted class}. It was 
originally introduced by U.~R.~Suh and the third named author in \cite{OS19b}.
The notion of a Q-datum is a combinatorial generalization of a Dynkin quiver (with a height function) 
introduced by the first and the third named authors in \cite{FO21}
to give a unified description of these twisted AR quivers.  
With these notions, we can define the monoidal subcategory 
$\Cc_{\cQ} \subset \Cc_{\Z, \fg}$ associated with each Q-datum $\cQ$ for $\fg$  in a suitable way.

As a generalization of the isomorphism (\ref{eq:CQisom}), 
we prove the following.
\begin{Thm}[$\doteq$ Corollary~\ref{Cor:HLO}] \label{Thm0:HLO}
For each Q-datum $\cQ$ for $\fg$, there exists an isomorphism of $\Z[t^{\pm 1/2}]$-algebras
\begin{equation} \label{eq:CQisom2}
\cK_t(\Cc_\cQ) \simeq \cA_{t}[N_-],
\end{equation}
which sends the basis of $\cK_t(\Cc_\cQ)$ formed by the simple $(q,t)$-characters 
to the dual canonical basis of $\cA_{t}[N_-]$, and the basis formed by the standard $(q,t)$-characters to the normalized dual PBW-type basis associated with a certain reduced word corresponding to $\cQ$.  
\end{Thm}
It is notable that the quantum unipotent coordinate algebra $\cA_{t}[N_-]$ here
is associated with the simply-laced $\sg$ (not $\fg$).   
When $\fg$ is of type $\mathrm{B}$, our isomorphism (\ref{eq:CQisom2}) coincides with the isomorphism 
established in the previous work \cite{HO19} of the second and the fourth named authors. 
Hence we refer to the isomorphism (\ref{eq:CQisom2}) as \emph{the HLO isomorphism} in this paper. 
Our proof of Theorem~\ref{Thm0:HLO} generalizes the arguments in~\cite{HO19} with new ingredients, in particular the combinatorial theory of Q-data.
The essence of the proof is a comparison 
between the determinantal identities satisfied by the normalized quantum unipotent minors in $\cA_{t}[N_-]$ 
(which can be understood as a part of its quantum cluster algebra structure)
and \emph{the quantum $T$-systems} in $\cK_t(\Cc_{\cQ})$,
where the latter is a newly obtained result in this paper except for types $\mathrm{ABDE}$ as explained in the next paragraph.

\subsection{}
Recall that the $T$-systems are certain functional relations appearing  in solvable lattice models (see~\cite{KNS11} for instance).
They were proved to be satisfied by the $q$-characters of \emph{Kirillov-Reshetikhin modules} (KR modules for short) by Nakajima~\cite{Nakajima03II}
for simply-laced $\fg$ and by the second named author~\cite{Hernandez06} for general $\fg$.
For types $\mathrm{ADE}$ and for type $\mathrm{B}$, 
their natural $t$-analogs, which we call the quantum $T$-systems, were established in \cite{HL15} and \cite{HO19} respectively,
as the relations satisfied in the quantum Grothendieck ring $\cK_t(\Cc_{\fg})$.
In this paper, we prove the quantum $T$-systems  for all $\fg$ in a unified way.
Strictly speaking, 
we establish them as a system of relations satisfied by another kind of elements $F_t(m)$ of $\cK_t(\Cc_{\fg})$
which contain $m$ corresponding to KR modules as their unique dominant monomials (see Section~\ref{ssec:qTsys}).   
Although it is not clear in our formulation (as well as that of \cite{HO19} for type $\mathrm{B}$) 
that the quantum $T$-systems are satisfied by the simple $(q,t)$-characters corresponding to KR modules,
we conjecture that it is true (indeed Conjectures~\ref{Conj0:KL} and \ref{Conj0:p} imply $L_t(m) = F_t(m)$ for any $m$ corresponding to KR modules).

\subsection{}
The HLO isomorphism (\ref{eq:CQisom2}) itself has important corollaries.
Since it translates 
the known positivities of the dual canonical basis of $\cA_t[N_-]$ to the desired positivities in $\cK_t(\Cc_{\cQ})$, 
we find that (\ref{p:KL}) and (\ref{p:sc}) in Conjecture~\ref{Conj0:p} hold when we restrict ourselves to the category $\Cc_{\cQ}$
for some Q-datum $\cQ$.  
Moreover, we can check that the specialization at $t=1$ of the HLO isomorphism
coincides with the isomorphism $K(\Cc_\cQ) \simeq \cA_{t=1}[N_-]$
induced from \emph{the generalized quantum affine Schur-Weyl duality functors} 
studied in \cite{KKK15, KKKOIV, KO17, OhS19, Fujita20, Naoi21}. 
It is known that the latter isomorphism sends the basis of $K(\Cc_\cQ)$ 
formed by the classes of simple modules to the specialization of the dual canonical basis of $\cA_{t=1}[N_-]$.
Thus it implies the following.

\begin{Thm} [$\doteq$ Corollary \ref{Cor:SW}]
The analog of Kazhdan-Lusztig conjecture, i.e., Conjecture~\ref{Conj0:KL} holds for the category $\Cc_{\cQ}$ for any $\cQ$.  
\end{Thm}

\subsection{}
Next we consider the analog of step (ii) for general $\fg$. 
To prove that the desired quantum Boson relations are satisfied in $\cK_t(\Cc_{\Z, \fg})$, 
we have to check several commutation relations between the $(q,t)$-characters of fundamental modules.
For this purpose, 
we need to know the singularities of \emph{the normalized $R$-matrices} between the fundamental modules.
Fortunately, the denominators of these $R$-matrices are already known due to many works~\cite{AK97, DO94, Fujita19, KKK15, Oh15R, OhS19}. 
Moreover, their relation to the structure constants of the quantum torus $\cY_t$ was observed by the first and the third named authors in~\cite{FO21}.  
Combining these ingredients with the positivities of $\cK_t(\Cc_\cQ)$ (discussed in the last paragraph),
we can obtain the conclusion.
  
\subsection{} \label{Intro:B}
Finally, we specialize the situation to the case when $(\fg, \sg)$ is of type $(\mathrm{B}_n, \mathrm{A}_{2n-1})$. 
In this case, we can check that the specialization at $t=1$ of
our isomorphism (\ref{eq:CZisom}) 
with a specific choice of Q-data
coincides with 
another isomorphism $K(\Cc_{\Z, \fg}) \simeq K(\Cc_{\Z, \sg})$
obtained by M.~Kashiwara, M.~Kim and the third named author in \cite{KKO19}
through a categorical relation arising from another example of the generalized quantum affine Schur-Weyl duality functors. 
The computation we need in the proof has already been done in the previous work~\cite[Section 12]{HO19} by the second and the fourth named authors.
Since the latter isomorphism $K(\Cc_{\Z, \fg}) \simeq K(\Cc_{\Z, \sg})$ respects the classes of simple modules, 
we immediately obtain the proof of Conjecture~\ref{Conj0:KL} for this case.

\subsection{} As mentioned above, in opposition to the case of simply-laced quantum affine algebras,
as far as the authors know, no geometric construction of (co)standard modules involving
non-simply-laced analog of quiver varieties is known or fully developed.
Note however that, in addition to the recent results of Nakajima-Weekes mentioned above, geometric $q$-characters formulas of certain modules, including KR modules, have been obtained in \cite{HL16} for
non-simply-laced types. Also an approach to certain analogs of non simply-laced quiver varieties have been developed in \cite{GLS17}. 
This raises the question of a geometric interpretation of the positivity results we establish in the present paper, in terms of quiver varieties associated to non-simply-laced quantum affine algebras.

\subsection{} The authors plan to come back in the near future to the following related 
questions. The approach in \cite{HO19} was based on the study of quantum 
cluster algebra structures. 
We will study the compatibility between the isomorphisms of quantum Grothendieck rings that we establish and 
the various involved quantum cluster algebras structures (see \cite{Qin17, Bit19}).
We expect this will also provide new interesting information on related 
$q$-characters and their $t$-analogs, in the spirit of \cite{HL16, Bit19}. 
We will also investigate the relation between on one hand the specialization of our isomorphisms to classical Grothendieck ring isomorphisms, and on the other hand the various 
isomorphisms obtained via the framework of the generalized quantum affine Schur-Weyl duality \cite{KKK18}. Eventually, we expect our approach and results to hold for integrable representations of general quantum affinizations of Kac-Moody algebras and 
for representations in the category $\mathcal{O}$ of a Borel subalgebra of a quantum affine algebra \cite{HJ12}.

\subsection*{Organization} 
This paper is organized as follows.
In Section~\ref{sec:qloop}, we recall the definition of quantum loop algebras and 
some important properties of their finite-dimensional representations.
In Section~\ref{sec:Grot}, we review the definition and some properties of 
the quantum Grothendieck ring of the category $\Cc_{\Z, \fg}$. 
Section~\ref{sec:Qdata} is a recollection on the combinatorial theory related to Q-data and their twisted AR quivers.
In Section~\ref{sec:C_Q}, we see several applications of the theory of Q-data to the representation theory of the quantum loop algebras.
In particular, the definition of the monoidal subcategory $\Cc_{\cQ}$ is recalled.
In Section~\ref{sec:qTsys}, after reviewing the classical $T$-systems in term of Q-data, 
we establish the quantum $T$-systems for all $\fg$.  
Section~\ref{sec:coord} is a brief review of the quantum unipotent coordinate algebra $\cA_{t}[N_-]$
and its quantum cluster algebra structure. 
In Section~\ref{sec:HLO}, we prove the HLO isomorphism (\ref{eq:CQisom2}) for each Q-datum
and obtain its corollaries.     
In Section~\ref{sec:Rmat}, we deduce several commutation relations in the quantum Grothendieck ring
which play an important role in the proof of our main theorem.
In Section~\ref{sec:isom}, we prove the desired presentation of the localized Grothendieck ring 
and then construct a collection of isomorphisms among the quantum Grothendieck rings.
In particular, we obtain an isomorphism (\ref{eq:CZisom}) and hence the desired positivities (\ref{p:KL}) and (\ref{p:sc}) for all $\fg$.    
In the last Section~\ref{sec:typeB}, we focus on the case when $\fg$ is of type $\mathrm{B}$ to prove 
Conjecture~\ref{Conj0:KL} and also the remaining positivity (\ref{p:ch}).   

\subsection*{Acknowledgments}
The authors are grateful to Hiraku Nakajima for helpful discussions, comments and correspondence. The authors would like also to thank Bernard Leclerc and Alex Weekes for their comments. The authors are also thankful to the anonymous referee whose suggestions improve our present paper. 
A part of this work was done when the first named author enjoyed short-term postdoctoral positions in Kyoto University and in Kavli IPMU during the spring-summer semester in 2020. 
He thanks both institutions and colleagues for their supports in the complicated situation.   

R.~F.~was supported in part by JSPS Grant-in-Aid for Scientific Research (A) JP17H01086, (B) JP19H01782 and by JSPS Overseas Research Fellowships.  
D.~H.~was supported by the European Research Council
under the European Union’s Framework Programme H2020 with ERC Grant Agreement
number 647353 Qaffine.
S.-j.\ Oh was supported by
the Ministry of Education of the Republic of Korea and the National Research Foundation of Korea (NRF-2019R1A2C4069647). 
H. O. is supported by Grant-in-Aid for Young Scientists No.19K14515.

\subsection*{Convention} Throughout this paper, we use the following general convention.
\begin{itemize}
\item For a statement $\mathtt{P}$, we set $\delta(\mathtt{P})$ to be $1$ or $0$
according that $\mathtt{P}$ is true or not.
As a special case, we use the notation $\delta_{i,j} \seq \delta(i=j)$ (Kronecker's delta).
\item For a totally ordered set $J=\{\cdots < j_{-1} < j_{0} < j_{1} < j_{2} < \cdots \}$, write
$$
\prod_{j \in J}^{\to} A_{j} \seq \cdots A_{j_{-1}}A_{j_0}A_{j_{1}}A_{j_{2}} \cdots, \quad
\bigotimes_{j \in J}^{\to} A_{j} \seq \cdots A_{j_{-1}}\otimes A_{j_0}\otimes A_{j_{1}}\otimes A_{j_{2}}\cdots.
$$
\item For an abelian category $\Cc$, we denote its Grothendieck group by $K(\Cc)$.
The class of an object $X \in \Cc$ is denoted by $[X] \in K(\Cc)$. 
Write the subset of $K(\Cc)$ formed by 
the classes of simple objects as $\Irr \Cc$. 
If moreover the category $\Cc$ is monoidal with a bi-exact tensor product $\otimes$, 
$K(\Cc)$ is endowed with a ring structure by
$[X] \cdot [Y] = [X \otimes Y]$, which we call the Grothendieck ring of $\Cc$. 

\item Let $\kk$ be a field. For a polynomial $f(z) \in \kk[z]$,
we denote its degree by $\deg(f)$. For a rational function $f(z) \in \kk(z)$, we denote by 
$[f(z)]^{\pm}$ its formal Laurent expansion at $z^{\pm 1} = 0$,
which are elements in $\kk(\!( z )\!)$ and $\kk(\!( z^{-1} )\!)$, respectively.  
\end{itemize}

\section{Quantum loop algebras and their finite-dimensional representations}
\label{sec:qloop}

In this section, we recall the definition of quantum loop algebras and important properties of their finite-dimensional representations.

\subsection{Notation}

Let $\fg$ be a complex finite-dimensional simple Lie algebra of rank $n$.
As is well-known, it is classified by the Dynkin diagrams of types $\mathrm{A}_{n}, \mathrm{B}_{n}, \cdots, \mathrm{G}_2$.
Denote by $r$ \emph{the lacing number of $\fg$}, that is
$$
r \seq \begin{cases}
1 & \text{if $\fg$ is of type $\mathrm{A}_{n},\mathrm{D}_{n}$, or $\mathrm{E}_{6,7,8}$}, \\
2 & \text{if $\fg$ is of type $\mathrm{B}_{n}, \mathrm{C}_{n}$, or $\mathrm{F}_{4}$}, \\
3 & \text{if $\fg$ is of type $\mathrm{G}_{2}$}.
\end{cases}
$$
We say that $\fg$ is simply-laced (resp.~non-simply-laced) if $r=1$ (resp.~$r>1$). 
Let $I = \{ 1,2,\ldots,n\}$ be the set of Dynkin indices and  
$C=(c_{ij})_{i,j \in I}$ the Cartan matrix of $\fg$.
We write $i \sim j$ for $i,j \in I$ if $c_{ij} < 0$.
Let $D = \mathrm{diag}(d_i \mid i \in I)$ denote the unique diagonal matrix 
such that $d_{i} \in \{1,r\}$ for each $i \in I$ 
and $DC = (d_i c_{ij})_{i,j \in I}$ is symmetric.  

\subsection{Quantum loop algebras} 
\label{ssec:qloop}

In what follows, for an invertible element $x \neq \pm 1$ in an integral domain and an integer $k$, we write the standard quantum number
\begin{equation} \label{eq:qinteger}
[k]_{x} \seq \frac{x^{k}-x^{-k}}{x-x^{-1}}.
\end{equation}
For integers $k \ge l \ge 0$, we set
$$
[k]_{x}! \seq \prod_{u=1}^{k} [u]_{x}, \quad
\left[ \begin{matrix} k \\ l \end{matrix} \right]_{x} \seq \frac{[k]_{x}!}{[k-l]_{x}! [l]_{x}!}.
$$

Let $q$ be an indeterminate (quantum parameter) and $\kk$
the algebraic closure of the rational function field $\Q(q)$.
For each $i \in I$, we define $q_{i} \seq q^{d_{i}} \in \kk$.

\begin{Def}
\emph{The quantum loop algebra} $U_{q}(L\fg)$
is the $\kk$-algebra given by the set of generators
$$\{ k_{i}^{\pm 1} \mid i \in I\} \cup \{ x_{i,k}^{\pm} \mid i \in I, k \in \Z\} \cup \{ h_{i, l} \mid i \in I, l \in \Z \setminus\{0\}\}$$
satisfying the following relations:
\begin{itemize}
\item $k_{i} k_{i}^{-1} = k_{i}^{-1} k_{i} = 1, k_{i} k_{j} = k_{j} k_{i}$ for $i,j \in I$,
\item $k_{i} x_{j, k}^{\pm} k_{i}^{-1} = q_{i}^{\pm c_{ij}} x_{j,k}^{\pm}$ for $i,j \in I$ and $k \in \Z$,
\item $[k_{i}, h_{j,l}] = [h_{i,l}, h_{j,m}] = 0$ for $i,j \in I$ and $l,m \in \Z \setminus\{0\}$,
\item $[x_{i}^{+}(z), x_{j}^{-}(w)]
= \displaystyle \frac{\delta_{ij}}{q_{i} - q_{i}^{-1}}
\left( \updelta(z/w) \phi_{i}^{+}(w) - \updelta(w/z) \phi_{i}^{-}(z) \right)$ for $i, j \in I$,
\item $(q^{\pm c_{ij}}_{i}z-w) x_{i}^{\pm}(z) x_{j}^{\pm}(w) = (z-q^{\pm c_{ij}}_{i}w) x_{j}^{\pm}(w) x_{i}^{\pm}(z)$ for $i,j \in I$,
\item $(q^{\pm c_{ij}}_{i}z-w) \phi_{i}^{\varepsilon}(z) x_{j}^{\pm}(w) = (z-q^{\pm c_{ij}}_{i}w) x_{j}^{\pm}(w) \phi_{i}^{\varepsilon}(z)$
for $i,j \in I$ and $\varepsilon \in \{ +, -\}$,
\item $\displaystyle \sum_{g \in \mathfrak{S}_{b}} \sum_{k=0}^{b}
(-1)^{k} \left[\begin{matrix} b \\ k\end{matrix} \right]_{q_i}
x_{i}^{\pm}(z_{g(1)}) \cdots x_{i}^{\pm}(z_{g(k)}) x_{j}^{\pm}(w) x_{i}^{\pm}(z_{g(k+1)}) \cdots x_{i}^{\pm}(z_{g(b)}) = 0
$, where $b \seq 1-c_{ij}$, for $i, j \in I$ with $i \sim j$, 
\end{itemize}
where $\updelta(z), x_{i}^{\pm}(z), \phi_{i}^{\pm}(z)$ are the formal power series defined by
$$
\updelta(z) \seq \sum_{l=-\infty}^{\infty}z^{l},
\quad
x_{i}^{\pm}(z) \seq \sum_{l=-\infty}^{\infty}x_{i,k}^{\pm}z^{k},
\quad
\phi_{i}^{\pm}(z) \seq k_{i}^{\pm 1}
\exp \left( \pm (q_{i}-q_{i}^{-1})
\sum_{l=1}^{\infty}h_{i, \pm l} z^{\pm l} \right).
$$
\end{Def}

The quantum loop algebra $U_{q}(L\fg)$ is a sub-quotient of the corresponding 
(untwisted) quantum affine algebra and hence has a Hopf algebra structure,
which is defined via another realization of $U_{q}(L\fg)$ 
(called the Drinfeld-Jimbo presentation).

\subsection{Finite-dimensional representations} 
\label{ssec:fdrep}
A finite-dimensional $U_{q}(L\fg)$-module is said to be of type $\mathbf{1}$
if the element $k_{i}$ acts as a diagonal $\kk$-linear operator whose eigenvalues 
belong to the set $\{ q^{k} \mid k \in \Z\}$ for each $i \in I$. It is well-known that 
the study of finite-dimensional representations of $U_q(L\fg)$ reduces essentially to the study of the category $\Cc$ of type $\mathbf{1}$ finite-dimensional $U_{q}(L\fg)$-modules.
This is a $\kk$-linear rigid monoidal tensor category. 
As a monoidal category, the category $\Cc$ is not braided. 
Indeed, $V \otimes W$ is not isomorphic to $W \otimes V$ for general $V, W \in \Cc$.
We say that $V$ and $W$ \emph{commute} if $V \otimes W \cong W \otimes V$
as $U_{q}(L \fg)$-modules. 

Let $V \in \Cc$. Since the elements $\{ k_{i}^{\pm 1}, h_{i, l} \mid i \in I, l \in \Z \setminus\{0\}\}$
mutually commute, we have a decomposition 
$$
V = \bigoplus_{\gamma \in (\kk [\![ z ]\!] \times \kk[\![ z^{-1} ]\!])^{I}} V_{\gamma} 
$$  
where, for each $\gamma = (\gamma_{i}^{+}(z), \gamma_{i}^{-}(z))_{i \in I}$, we define
$V_{\gamma}$ as the subspace of $V$ on which each coefficient of the series 
$\phi_{i}^{\pm}(z) - \gamma_{i}^{\pm}(z)$ acts nilpotently for any $i \in I$.
If $V_{\gamma} \neq 0$, it is called \emph{an $\ell$-weight space} and $\gamma$ is called 
the corresponding \emph{$\ell$-weight}.

For a simple module $L$ in  $\Cc$, there uniquely exists an $\ell$-weight space
$L_{\gamma_0}$ such that $x^{+}_{i,k} L_{\gamma_0}=0$ for all
$i \in I$ and $k \in \Z$. 
The isomorphism class of a simple $U_{q}(L\fg)$-module in $\Cc$ is determined by such $\gamma_0$
(called \emph{the $\ell$-highest weight} of $L$).
Moreover we have the following.

\begin{Thm}[{\cite[Theorem 3.3]{CP95}, \cite[Theorem 12.2.6]{CP}}] \label{Thm:CP}
If $\gamma= (\gamma_{i}^{+}(z), \gamma_{i}^{-}(z))_{i \in I}$ is an $\ell$-highest weight 
of a simple $U_{q}(L\fg)$-module in $\Cc$, for each $i\in I$ there exists a unique polynomial $P_{i}(z) \in \kk[z]$ with $P_{i}(0) = 1$ for $i \in I$ such that
\begin{equation} \label{eq:ellhwt}
\gamma_{i}^{\pm}(z) = q_{i}^{\deg(P_i)} \left[\frac{P_{i}(zq_{i}^{-1})}{P_{i}(zq_{i})}\right]^{\pm}.
\end{equation}
Conversely, for any $(P_{i}(z))_{i \in I} \in (1+z\kk[z])^{I}$, we have a simple $U_{q}(L\fg)$-module 
in $\Cc$ whose $\ell$-highest weight $\gamma$ is given by {\rm (\ref{eq:ellhwt})}. 
\end{Thm} 

This $I$-tuple of polynomials $(P_{i}(z))_{i \in I}$ corresponding to a simple module $L$ 
is called \emph{the Drinfeld polynomials of $L$}.
A non-zero vector in the $\ell$-highest weight space $L_{\gamma}$ 
is called \emph{an $\ell$-highest weight vector of $L$}. Since $\dim_\kk L_\gamma=1$,
it is unique up to a scalar.  

\subsection{$q$-characters} 
\label{ssec:qch}
Frenkel and Reshetikhin~\cite[Proposition 1]{FR98} proved that 
any $\ell$-weight $\gamma$ of a $U_{q}(L\fg)$-module in $\Cc$ 
is of the form
$$
\gamma_{i}^{\pm}(z) = q_{i}^{\deg(Q_i) - \deg(R_i)} \left[ \frac{Q_{i}(zq_{i}^{-1})R_{i}(zq_{i})}{Q_{i}(zq_{i})R_{i}(zq_{i}^{-1})} \right]^{\pm}
$$ 
for some $(Q_{i}(z))_{i \in I}, (R_{i}(z))_{i \in I} \in (1+z\kk[z])^{I}$.
For this $\gamma$, define a monomial $m_{\gamma}$
in the Laurent polynomial ring
$\Z[Y_{i,a}^{\pm 1} \mid i \in I, a \in \kk^\times]$ by
$$
m_{\gamma} \seq \prod_{i \in I, a \in \kk^\times} Y_{i,a}^{q_{i,a}-r_{i,a}},
$$ 
where 
$
Q_{i}(z) = \prod_{a \in \kk^{\times}} (1-az)^{q_{i,a}},
R_{i}(z) = \prod_{a \in \kk^{\times}}(1-az)^{r_{i,a}}.$

\begin{Thm}[{\cite[Theorem 3]{FR98}}] \label{Thm:FR}
The assignment 
$$
[V] \mapsto \sum_{\gamma} \dim_{\kk}(V_{\gamma}) m_{\gamma}
$$
defines an injective algebra homomorphism $\chi_{q} \colon K(\Cc) \to \Z[Y_{i,a}^{\pm 1} \mid i \in I, a \in \kk^\times ]$.
\end{Thm}

This implies that the Grothendieck ring $K(\mathcal{C})$ is commutative, although the category $\mathcal{C}$ is not braided. The map $\chi_{q}$ is called \emph{the $q$-character homomorphism.}
A monomial which is a product of only positive powers of $Y_{i,a}$'s is said to be \emph{dominant}.
Write the simple $U_{q}(L\fg)$-module in $\Cc$
with $\ell$-highest weight $\gamma$ as $L(m_{\gamma})$. 
Then, by Theorem~\ref{Thm:CP}, $m_{\gamma}$
is dominant, and the set of dominant monomials are in bijection with $\Irr \Cc$
via $m \mapsto [L(m)]$.

Following \cite{FR98}, for $i \in I, a \in \kk^\times$, we set
$$
A_{i,a} \seq Y_{i,aq_{i}^{-1}} Y_{i,aq_{i}} 
\left( \prod_{j \colon c_{ji}=-1} Y_{j,a}^{-1}\right)
\left( \prod_{j \colon c_{ji}=-2} Y_{j,aq^{-1}}^{-1} Y_{j, aq}^{-1} \right)
\left( \prod_{j \colon c_{ji}=-3} Y_{j,aq^{-2}}^{-1} Y_{j,a}^{-1} Y_{j,aq^{2}}^{-1} \right).
$$
As these monomials are algebraically independent, one can define the Nakajima partial ordering $\le$ on the set of monomials in $\Z[Y_{i,a}^{\pm 1} \mid i \in I, a \in \kk^\times]$
as follows:
$$
\text{$m \le m^{\prime}$ if and only if
$m^{-1}m^{\prime}$ is a product of elements of $\{A_{i,a} \mid i \in I, a \in \kk^\times \}$.}
$$

\begin{Thm}[{\cite{FR98, FM01}}] \label{Thm:FM}
Let $m$ be a dominant monomial. 
Then all the monomials occurring in $\chi_{q}(L(m)) - m$
are strictly less than $m$ with respect to the ordering $\le$. 
\end{Thm}
\subsection{Twists and duality}
\label{ssec:duality}

For each $a \in \kk^\times$, we consider the $\kk$-algebra automorphism 
$\tau_{a}$ of the quantum loop algebra $U_q(L\fg)$ given by 
$$
\tau_a(k_{i}) = k_i, \
\tau_a(x^{\pm}_{i, k}) = a^k x^{\pm}_{i,k}, \
\tau_a(h_{i,l}) = a^l h_{i,l}, 
$$
for $i \in I, k \in \Z, l \in \Z \setminus \{ 0 \}$. 
This automorphism $\tau_a$ is a $q$-analog of the loop rotation $z \mapsto az$ for
the loop algebra $L\fg = \fg \otimes_\C \C[z^{\pm}]$.
For a finite-dimensional representation $V$ in $\Cc$, 
the twisted representation $\tau_a^*V$ by $\tau_a$ again belongs to $\Cc$.
Thus the assignment $V \mapsto \tau_a^*V$ defines a monoidal auto-equivalence $T_{a}$ of $\Cc$. 
The functors $\{ T_{a} \}_{a \in \kk^\times}$ define an action of the group $\kk^\times$
on the monoidal category $\Cc$ as they satisfy $T_a \circ T_b \simeq T_{ab}$ and $T_1 = \id_\Cc$.
If we denote by $y \mapsto y_a$
the $\Z$-algebra automorphism of $\Z[Y_{i, b}^{\pm 1} \mid i  \in I, b \in \kk^\times]$ given by $Y_{i,b} \mapsto Y_{i, ab}$,
we have
$T_a \left( L(m) \right) \cong L(m_a)$ for any simple module $L(m) \in \Cc$. 

Recall that the monoidal category $\Cc$ is rigid, i.e.~every object $V \in \Cc$ has its left dual $\cD(V)$ and right dual $\cD^{-1}(V)$ in $\Cc$.\footnote{By convention, we choose the coproduct of $U_q(L\fg)$ opposite to that of 
\cite{AK97, Kas02}, that is why left and right duals are interchanged in comparison to these papers.}
Since $\cD \circ \cD^{-1} \simeq \cD^{-1} \circ \cD \simeq \id_{\Cc}$, 
we can naturally extend $\cD^{\pm 1}$ to a collection of functors $\{\cD^{k}\}_{k \in \Z}$
so that we have $\cD^{k} \circ \cD^{l} \simeq \cD^{k+l}$ for $k,l \in \Z$. 
Then we have $\cD^2 \simeq T_{q^{2rh^{\vee}}}$ as endo-functors on $\Cc$,
where $h^\vee$ is the dual Coxeter number of $\fg$. 
See Table~\ref{table:cl} in Section~\ref{sec:Qdata} below for its explicit value.

Now we consider the $\Z$-algebra automorphism $\fD$  of $\Z[Y_{i,a}^{\pm 1} \mid i \in I, a \in \kk^\times]$ given by
\begin{equation} \label{eq:deffD}
\fD(Y_{i,a}) \seq Y_{i^{*}, aq^{rh^{\vee}}} \quad \text{for each $i \in I, a \in \kk^{\times}$},
\end{equation}
where $i \mapsto i^*$ denotes the involution of the set $I$ induced from the longest element of the Weyl group of $\fg$.
See Section~\ref{ssec:unf} for its precise definition when $\fg$ is simply-laced, 
while it is trivial when $\fg$ is non-simply-laced.  

\begin{Thm}[cf.~{\cite[Proposition 5.1(b)]{CP96a}, \cite[Appendix A]{AK97}, \cite[Corollary 6.10]{FM01}}] \label{Thm:fD}
For any simple object $L(m) \in \Cc$ and $k \in \Z$, there exists an isomorphism
$$
\cD^k (L(m)) \cong L(\fD^k(m))
$$
of $U_{q}(L\fg)$-modules. In particular, we have
$$
\chi_q(\cD (L(m))) = \fD(\chi_q(L(m))) = \chi_{q}(L(\fD(m))).
$$
\end{Thm}

Finally, we recall the involution $\omega$ of the quantum loop algebra $U_q(L\fg)$ as in \cite{Chari95}: 
it is defined by $$\omega(x_{i,k}^\pm) = -x_{i,-k}^\mp, \ \omega(h_{i,l}) =- h_{i,-l}, \ \omega(k_i) = k_i^{-1}$$
for $i \in I, k \in \Z, l \in \Z \setminus \{ 0 \}$. 
Then, for a representation $V \in \Cc$, the $q$-character of the twisted representation 
$\omega^*V$ by $\omega$ is obtained from $\chi_q(V)$ by replacing each variable $Y_{i,a}$ by 
$\omega(Y_{i,a}) \seq Y_{i,a^{-1}}^{-1}$. In particular, the assignment $\chi_{q}(V) \mapsto \chi_{q}(\omega^*V) = \omega(\chi_{q}(V))$
defines a ring automorphism of $\chi_{q}(K(\Cc))$. 
Moreover, if $V$ is a simple module $L(m)$, 
then we have 
\begin{equation} \label{eq:omg}
\omega^*L(m) \cong L(\omega^{*}m),
\end{equation}
 where $\omega^{*}m$ is the dominant monomial obtained from 
$m$ by replacing each variables $Y_{i,a}$ with $Y_{i^{*},a^{-1}q^{-rh^\vee}}$ (see \cite{Hernandez07}). 

\subsection{The category $\Cc_{\Z}$}
\label{ssec:CZ}

In this subsection, 
 we introduce a certain monoidal subcategory $\Cc_{\Z}$ of $\Cc$, which is of our main interest in this paper. 
For $i,j \in I$, write 
$$
d_{ij}  \seq \min (d_i, d_j).
$$
A function $\epsilon \colon I \to \{0,1\}$ is  called  \emph{a parity function for $\fg$} if it satisfies the condition
\begin{equation} \label{cond:parity}
\epsilon_{i} \equiv \epsilon_{j} + d_{ij} \pmod{2} \quad \text{for any $i,j \in I$ with $i \sim j$}.
\end{equation}
In what follows, we choose and fix a parity function $\epsilon$ for $\fg$. 
Then we define
\begin{equation} \label{eq:hI}
\hI \seq \{ (i,p) \in I \times \Z \mid p - \epsilon_{i} \in 2\Z \}.
\end{equation} 
Let $\cY$ be the subalgebra of $\Z[Y_{i,a}^{\pm 1} \mid (i, a) \in I \times \kk^\times]$
consisting of Laurent polynomials in the variables $Y_{i,q^{p}}$
for $(i,p) \in \hI$. 
Note that we have $A_{i,q^{p+d_{i}}} \in \cY$ 
if and only if $(i,p) \in \hI$.
We denote by $\cM$ the set of all the dominant monomials in $\cY$. Following \cite[Section 3.7]{HL10}, we consider the following category.

\begin{Def}
We define the category $\Cc_{\Z}$ as the Serre subcategory of 
$\Cc$ which satisfies $\Irr \Cc_{\Z} = \{ [L(m)] \mid m \in \cM \}$.
\end{Def}

\begin{Prop}
The category $\Cc_{\Z}$ is closed under taking tensor products and applying the duality functors $\cD^{\pm1}$.
Moreover, the $q$-character homomorphism restricts to the injective algebra homomorphism 
$\chi_{q} \colon K(\Cc_{\Z}) \to \cY$. 
\end{Prop}

The subcategory $\Cc_{\Z}$ can be regarded as a skeleton of the entire rigid monoidal category $\Cc$.
Indeed, the description of simple modules in $\Cc$ and hence that of the Grothendieck $K(\Cc)$ are
essentially reduced to the description of the simple modules in $\Cc_{\Z}$ and that of the Grothendieck ring $K(\Cc_\Z)$ 
thanks to the following fact.

\begin{Prop}
Any simple representation $V$ in $\Cc$ can be factorized into a commutative tensor product of the form
$$
V \cong T_{a_1} (V_1)\otimes T_{a_2}(V_2) \otimes \cdots \otimes T_{a_d}(V_d)
$$ 
with $V_k \in \Cc_\Z$ and $a_k \in \kk^\times$ for $1 \le k \le d$
such that $a_k/a_l \not \in q^{2\Z}$ for all $k \neq l$.
\end{Prop}

In this paper, we always work in the subcategory $\Cc_{\Z} \subset \Cc$.
From now on, we put
$Y_{i,p} \seq Y_{i,q^{p}}$ and $A_{i,p} \seq A_{i,q^{p}}$ for $(i,p) \in I \times \Z$ 
for the sake of simplicity.
For example, the algebra automorphism $\fD \colon \cY \to \cY$ defined in (\ref{eq:deffD})
is given by $\fD(Y_{i,p}) = Y_{i^*, p+rh^\vee }$ in this notation.
For a monomial $m$ in $\cY = \Z[Y_{i,p}^{\pm 1} \mid (i,p) \in \hI \ ]$, we write
\begin{equation} \label{eq:um}
m = \prod_{(i,p) \in \hI}Y_{i,p}^{u_{i,p}(m)}
\end{equation}
with $u_{i,p}(m) \in \Z$. 

The Grothendieck ring $K(\Cc)$ is a polynomial ring in the classes of the \emph{fundamental modules} $L(Y_{i,a})$, see \cite{FR98}. This implies the following.

\begin{Prop} \label{Prop:KCZ}
The $\Z$-algebra $K(\Cc_{\Z})$ is isomorphic to the polynomial ring 
with $\hI$-many variables
$\Z[x_{i,p} \mid (i,p) \in \hI \ ]$ via $[L(Y_{i,p})] \mapsto x_{i,p}$.
\end{Prop}

The $U_{q}(L\fg)$-modules of the form $L(Y_{i,p})$ are exactly the fundamental modules which are in $\Cc_\Z$.

\section{Quantum Grothendieck rings}
\label{sec:Grot}

In this section, we recall the definition and the first properties of the quantum Grothendieck ring 
of the monoidal category $\Cc_{\Z}$. We also discuss some additional symmetry properties of quantum Grothendieck rings with respect to the dualities $\mathcal{D}$ and $\omega$. We keep the notations introduced in the previous section.

\subsection{Inverse of quantum Cartan matrix}
\label{ssec:tc}

Let $z$ be an indeterminate.
The quantum Cartan matrix of $\fg$
is the $\Z[z^{\pm 1}]$-valued $I\times I$-matrix
$C(z) = (C_{ij}(z))_{i,j \in I}$ whose $(i,j)$-entry is given by
$$
C_{ij}(z) \seq \begin{cases}
z^{d_{i}} + z^{-d_{i}} & \text{if $i=j$}, \\
[c_{ij}]_{z} & \text{if $i \neq j$},
\end{cases}
$$
where we used the notation (\ref{eq:qinteger}).
As the Cartan matrix $C = C(1)$ is invertible, so is $C(z)$. Denote its inverse by $\tC(z) = (\tC_{ij}(z))_{i,j \in I}$.
Note that the entries of $\tC(z)$ are rational functions in $z$. 
Then we define a collection of integers 
$\{ \tc_{ij}(u) \mid i,j \in I, u \in \Z \}$ as the coefficients of the Laurent expansions
$$
\sum_{u \in \Z} \tc_{ij}(u) z^{u} \seq \left[ \tC_{ij}(z) \right]^{+}.
$$

The $2h^\vee$-periodicity of these integers for simply-laced types was established in \cite{HL15}. 
Here we list several important properties obtained in \cite{FO21} for general types.

\begin{Lem}[{\cite[Section~4]{FO21}}] \label{Lem:tc}
The integers $\{ \tc_{ij}(u) \mid i,j \in I, u \in \Z\}$ satisfy the following properties$\colon$
\begin{enumerate}
\item \label{tc:d} $\tc_{ij}(u) = 0$ if $u < d_{i}$, and $\tc_{ij}(d_{i}) = \delta_{i,j}$,
\item $\tc_{ij}(u+d_{i}) - \tc_{ij}(u-d_{i}) = \tc_{ji}(u+d_{j}) - \tc_{ji}(u-d_{j})$ for all $u \in \Z$,
\item \label{tc:tp} for any $u \in \Z$, 
$$
\tc_{ij}(u) = \begin{cases}
\tc_{ji}(u) & \text{if $d_{i}=d_{j}$}, \\
\tc_{ji}(u+1) + \tc_{ji}(u-1) & \text{if $(d_{i}, d_{j}) = (1,2)$}, \\
\tc_{ji}(u+2) + \tc_{ji}(u) + \tc_{ji}(u-2) & \text{if $(d_{i}, d_{j}) = (1,3)$},
\end{cases}
$$
\item \label{tc:pe} $\tc_{ij}(u + rh^{\vee}) = -\tc_{ij^{*}}(u)$ for $u \ge 0$,
\item \label{tc:sym} $\tc_{ij}(rh^{\vee}-u) = \tc_{ij^{*}}(u)$ for all $0 \le u \le rh^{\vee}$,
\item \label{tc:po} $\tc_{ij}(u) \ge 0$ for all $0 \le u \le rh^{\vee}$.
\end{enumerate}
\end{Lem}

\subsection{The quantum torus $\cY_{t}$} 
\label{ssec:Yt}
Let $t$ be an indeterminate. 
We consider the Laurent polynomial ring $\Z[t^{\pm 1/2}]$,
introducing a square root $t^{1/2}$ of $t$.

\begin{Def}[\cite{Hernandez04}] \label{Def:Yt}
We define the quantum torus $\cY_{t}$ associated with $\fg$ 
as the $\Z[t^{\pm1/2}]$-algebra presented by 
the set of generators $\{\tY_{i, p}^{\pm 1} \mid (i,p) \in \hI \ \}$ and the following relations:
\begin{itemize}
\item 
$\tY_{i,p}\tY_{i,p}^{-1}= \tY_{i,p}^{-1} \tY_{i,p}=1$ for each $(i,p) \in \hI$,
\item 
$\tY_{i, p}\tY_{j, s} = t^{\Nn(i,p;j,s)}\tY_{j,s}\tY_{i,p}$ for each $(i,p), (j,s) \in \hI$,
\end{itemize}
where $\Nn\colon (I \times \Z)^{2}  \to \Z$ is defined by
\begin{equation} \label{eq:defNn}
\Nn(i,p;j,s) \seq \tc_{ij}(p-s-d_{i})- \tc_{ij}(p-s+d_{i}) - \tc_{ij}(s-p-d_{i})+ \tc_{ij}(s-p+d_{i}). 
\end{equation}
\end{Def}

\begin{Rem}\label{remqt} (i)
This is a remark on the convention.
Our parameter $t$ here is same as $t$ in~\cite{HL15}, which coincides with $t^{-1}$ in~\cite{HO19}. 

(ii) The definition of the quantum torus in \cite{Hernandez04} is based on the analysis of vertex operators occurring in the theory of Frenkel-Reshetikhin deformed $\mathcal{W}$-algebras which are related to transfer-matrices of corresponding quantum integrable systems. The construction of $t$-deformed Grothendieck rings for simply-laced quantum loop algebras in \cite{VV03} and \cite{Nakajima04} is based on a slightly different quantum torus obtained from convolution products on Nakajima quiver varieties. \end{Rem}

\begin{Rem} \label{Rem:Nn}
By Lemma~\ref{Lem:tc}, we have
\begin{equation}
\label{eq:Nnskew}
\Nn(i, p; j, s) = \Nn(j,p;i,s) = - \Nn(i,s;j,p) = -\Nn(j,s;i,p)
\end{equation}
for any $(i,p), (j,s) \in I \times \Z$, and
\begin{equation}
\label{eq:Nnhalf}
\Nn(i,p;j,s) = \tc_{ij}(p-s-d_{i})- \tc_{ij}(p-s+d_{i}) \qquad \text{if $p-s\ge \delta_{i,j}$}.
\end{equation}
Moreover, we have 
\begin{equation}
\label{eq:tYcomm}
\tY_{i,p} \tY_{j,p} = \tY_{j,p} \tY_{i,p}
\end{equation} 
for $(i,p), (j,p) \in \hI$ as $\Nn(i,p;j,p) = 0$.
\end{Rem}

There exists a $\Z$-algebra homomorphism $\evt \colon \cY_{t} \to \cY$
given by
$$
t^{1/2} \mapsto 1, \qquad \tY_{i,p} \mapsto Y_{i,p}.
$$
This map is called {\em the specialization at} $t=1$.

An element $\tm \in \cY_{t}$
is called a {\em monomial} if it is a product of the generators
$\tY_{i,p}$ for $(i,p) \in \hI$ and $t^{\pm 1/2}$. 
For a monomial $\tm \in \cY_{t}$, we set $u_{i,p}(\tm) \seq u_{i,p}(\evt(\tm))$
(recall the notation~(\ref{eq:um})).
A monomial $\tm$ in $\cY_{t}$ is said to be {\em dominant} if $\evt(\tm)$
is dominant, i.e., $u_{i,p}(\tm) \ge 0$ for all $(i,p) \in \hI$.
Moreover, for monomials $\tm, \tm^{\prime}$ in $\cY_{t}$, set
$$
\text{$\tm \le \tm^{\prime}$ if and only if $\evt(\tm) \le \evt(\tm^{\prime})$}.
$$
Following~\cite[Section~6.3]{Hernandez04}, we define the $\Z$-algebra anti-involution $\ol{(\cdot)}$ on $\cY_{t}$ by
$$
t^{1/2} \mapsto t^{-1/2}, \qquad \tY_{i,p} \mapsto t\tY_{i,p}.
$$
It is easy to show that, for any monomial $\tm$ in $\cY_{t}$,
there uniquely exists $a \in \Z$ such that $t^{a/2}\tm$ is $\ol{(\cdot)}$-invariant, 
and this element is denoted by $\ul{\tm}$.
Note that $\ul{t^{k/2}\tm} = \ul{\tm}$ for any $k \in \Z$.
Hence $\ul{\tm}$ depends only on $\evt(\tm)$.
Therefore, for every monomial $m$ in $\cY$, the element $\ul{m}$
is well-defined as an element of $\cY_{t}$.
The elements of this form are called {\em commutative monomials}.
For example, $\ul{Y_{i,p}} = t^{1/2}\tY_{i,p}$.
Note that, for every monomial $m$ in $\cY$, 
we have $\ul{(m^{-1})} = (\ul{m})^{-1} (=: \ul{m}^{-1})$.
The commutative monomials forms a free basis 
of the $\Z[t^{\pm 1/2}]$-module $\cY_{t}$.
For any two monomials $m,m^\prime$ in $\cY$, we have
$$\ul{m\cdot m^\prime} = t^{-\Nn(m, m^\prime)/2}\ul{m} \cdot \ul{m^\prime} = t^{\Nn(m,m^\prime)/2} \ul{m^\prime} \cdot \ul{m}$$
where $\Nn(m,m^\prime) \in \Z$ is defined as
\begin{equation} \label{eq:Nnmm}
\Nn(m,m^\prime) \seq \sum_{(i,p), (j,s) \in \hI} u_{i,p}(m) u_{j,s}(m^\prime) \Nn(i,p;j,s).
\end{equation}

\begin{Rem}
In~\cite{Hernandez04}, the quantum torus $\cY_{t}$ is a $\Z[t^{\pm 1}]$-algebra.
However, to guarantee the existence of commutative monomials, we need the square root of $t^{\pm 1}$.
This is the reason why we add $t^{\pm 1/2}$.  
\end{Rem} 

For $(i,p) \in \hI$, we set
$$
\tA_{i,p + d_{i}} \seq \ul{A_{i, p+d_{i}}} \quad \in \cY_{t}.
$$

\subsection{Quantum Grothendieck ring of $\Cc_{\Z}$} 
\label{ssec:KtCZ}

For each $i \in I$, denote by $\cK_{i,t}$ the $\Z[t^{\pm 1/2}]$-subalgebra
of the quantum torus $\cY_{t}$ generated by 
$$
\{ \tY_{i,p} (1+t^{-1}\tA_{i, p+ d_{i}}^{-1}) \mid p - \epsilon_{i} \in 2 \Z \} 
\cup \{ \tY_{j,s}^{\pm 1} \mid (j,s) \in \hI, j \neq i \}.
$$
Following \cite{Nakajima04, VV03, Hernandez04}, {\em the quantum Grothendieck ring of} $\Cc_{\Z}$ is defined as
$$
\cK_{t} \equiv \cK_{t}(\Cc_{\Z}) \seq \bigcap_{i \in I} \cK_{i,t}. 
$$
By construction, the quantum Grothendieck is stable by the  $\ol{(\cdot)}$-involution.

\begin{Thm}[{\cite[Theorem 5.11]{Hernandez04}}] \label{Thm:Ft}
For every dominant monomial $\tm$ in $\cY_{t}$, 
there uniquely exists an element $F_{t}(\tm)$ of $\cK_{t}$ such that $\tm$ is
the unique dominant monomial occurring in $F_{t}(\tm)$. The monomials $\tm^{\prime}$ occurring in $F_{t}(\tm) - \tm$
satisfy $\tm^{\prime} < \tm$.
In particular, the set $\{ F_{t}(\ul{m}) \mid m \in \cM \}$
forms a $\Z[t^{\pm 1/2}]$-basis of $\cK_{t}$.  
\end{Thm}

Note that $F_{t}(\ul{m})$ is $\ol{(\cdot)}$-invariant for any $m \in \cM$. It is constructed by an algorithm which is a $t$-deformation of the Frenkel-Mukhin algorithm \cite{FM01}.

For a dominant monomial $\tm$ in $\cY_{t}$, we set
$$
E_{t}(\tm) \seq \tm \left( \prod_{p \in \Z}^{\to}
\left( \prod_{i \in I: (i,p) \in \hI} \tY_{i,p}^{u_{i,p}(\tm)} \right)
\right)^{-1} 
\prod_{p \in \Z}^{\to}
\left( \prod_{i \in I: (i,p) \in \hI} F_{t}(\tY_{i,p})^{u_{i,p}(\tm)} \right),
$$
here $\prod_{i} \tY_{i,p}^{u_{i,p}(\tm)}$ is well-defined by (\ref{eq:tYcomm}),
and $\prod_{i} F_{t}(\tY_{i,p})^{u_{i,p}(\tm)}$ is well-defined by~\cite[Lemma 5.12]{Hernandez04}.

Note that $E_{t}(\tm)$ contains $\tm$ as its maximal monomial.
In particular, by Theorem~\ref{Thm:Ft},
\begin{equation} \label{eq:EtFt}
E_{t}(\ul{m}) = F_{t}(\ul{m}) + \sum_{m^{\prime} < m} C_{m,m^{\prime}}(t) F_{t}(\ul{m^{\prime}})
\end{equation} 
with $C_{m, m^{\prime}}(t) \in \Z[t^{\pm 1/2}]$ (in fact, $\Z[t^{\pm 1}]$).
Note that the set $\{E_{t}(\ul{m}) \mid m \in \cM\}$ also forms a $\Z[t^{\pm 1/2}]$-basis of $\cK_{t}$
since 
$
\# \{ m^{\prime} \in \cM \mid m^{\prime} < m \} < \infty
$
for any $m \in \cM$ (cf.~\cite[Section~3.4]{Hernandez04}).
For a dominant monomial $\tm$ in $\cY_{t}$ with $\evt(\tm) = m$, we have
$$
\evt(E_{t}(\tm)) = \chi_{q}(M(m)), \qquad \evt(\cK_{t}) = \chi_{q}(K(\Cc_{\Z}))
$$ 
\cite[Theorem 6.2]{Hernandez04}, here we set
$$
M(m) \seq \bigotimes^{\to}_{p \in \Z}\left( 
\bigotimes_{i \in I: (i,p) \in \hI} L(Y_{i,p})^{\otimes u_{i,p}(m)}
\right), 
$$
called a \emph{standard module}.
Note that, for any fixed $p \in \Z$, the isomorphism class
of the tensor product $\bigotimes_{i \in I: (i,p) \in \hI} L(Y_{i,p})^{\otimes u_{i,p}(m)}$
does not depend on the ordering of the factors, and it is in fact simple~\cite[Proposition 6.15]{FM01}.
Moreover, the module $M(m)$ has a simple head isomorphic to $L(m)$.\footnote{By convention, we choose the coproduct of $U_q(L\fg)$ opposite to that of 
\cite{AK97, Kas02} so that the module $M(m)$ is cyclic from the tensor product of $\ell$-highest weight vectors.}
The element $E_{t}(\ul{m})$ is called \emph{the $(q,t)$-character of the standard module} $M(m)$.

We consider another kind of elements $L_{t}(\ul{m})$ in $\cK_{t}$ which is conjecturally a $t$-quantum version of the 
$q$-character of simple modules.

\begin{Thm}[{\cite[Theorem 8.1]{Nakajima04}}, {\cite[Theorem 6.9]{Hernandez04}}]
\label{Thm:qtch}
For a dominant monomial $m \in \cM$, there exists a unique element $L_{t}(\ul{m})$ in $\cK_{t}$
such that 
\begin{itemize}
\item[(S1)] $\ol{L_{t}(\ul{m})} = L_{t}(\ul{m})$, and
\item[(S2)] $L_{t}(\ul{m}) = E_{t}(\ul{m}) + \sum_{m^{\prime} \in \cM} Q_{m, m^{\prime}}(t) E_{t}(\ul{m^{\prime}})$
with $Q_{m,m^{\prime}}(t) \in t\Z[t]$.
\end{itemize}
Moreover, we have $Q_{m, m^{\prime}}(t) = 0$ unless $m^{\prime} < m$. 
In particular, the set $\{ L_{t}(m) \mid m \in \cM\}$ forms a $\Z[t^{\pm 1/2}]$-basis
of $\cK_{t}$.
\end{Thm} 

The element $L_{t}(\ul{m})$ is called \emph{the $(q,t)$-character of the simple module $L(m)$}.
For example, if $m = Y_{i,p}$, we have 
$$E_t(\ul{Y_{i,p}}) = F_t(\ul{Y_{i,p}}) = L_t(\ul{Y_{i,p}}).$$
But these elements differ for a general $m$.

In what follows, for a dominant monomial $m \in \cM$, we will write 
$$
F_{t}(m) \seq F_{t}(\ul{m}), \qquad 
E_{t}(m) \seq E_{t}(\ul{m}), \qquad
L_{t}(m) \seq L_{t}(\ul{m}) 
$$
for simplicity.

\begin{Conj}[{\cite[Conjecture 7.3]{Hernandez04}}] \label{Conj:KL}
For all $m \in \cM$, we have
$$
\evt(L_{t}(m)) = \chi_{q}(L(m)).
$$
\end{Conj}
In fact, a fundamental theorem of Nakajima~\cite[Theorem 8.1]{Nakajima04} states that this holds true when $\fg$ is simply-laced. The proof used the geometry of quiver varieties (whose theory is fully developed only for simply-laced types at the moment). This was the main motivation for this conjecture.

Thanks to the unitriangular property (S2), we can write
\begin{equation} \label{eq:KL}
E_{t}(m) = L_{t}(m) + \sum_{m^{\prime} \in \cM \colon m' < m } P_{m, m^{\prime}}(t) L_{t}(m^{\prime})
\end{equation}
with some $P_{m,m'}(t) \in t\Z[t]$ for each $m \in \cM$.
The polynomials $P_{m,m'}(t)$ are analogs of Kazhdan-Lusztig polynomials for finite-dimensional representations of quantum loop algebras. 
It is also expected in \cite{Hernandez04} that these polynomials have positive coefficients. By the results of Nakajima \cite{Nakajima04}, it was already known for simply-laced types.

As for the fundamental modules, we have a nice result for general $\fg$ as follows.
Here we say that a monomial $\tm \in \cY_{t}$ is \emph{anti-dominant} if $\tm^{-1}$ is dominant. 

\begin{Thm}[\cite{FM01, Hernandez05}] \label{Thm:qtfund}
For each $(i,p) \in \hI$, the followings hold.
\begin{enumerate}
\item Every monomial in $L_{t}(Y_{i,p})$ has a positive coefficient.
\item \label{qtfund:evt} We have $\evt(L_{t}(Y_{i,p})) = \chi_{q}(L(Y_{i,p}))$.
\item \label{qtfund:monomials} $\tY_{j,s}$ appearing in $L_{t}(Y_{i,p})$ satisfies $p \le s \le p+rh^\vee$. 
Moreover, when $s=p$ $($resp.~$s=p+rh^\vee)$, the monomials containing $\tY_{j,s}$ is
$\ul{Y_{i,p}}$ $($resp.~$\ul{Y_{i^*,p+rh^\vee}^{-1}})$,
which is the unique dominant $($resp.~anti-dominant$)$ monomial in $L_{t}(Y_{i,p})$.
\end{enumerate}
\end{Thm}

\begin{Rem} \label{Rem:triE}
The unitriangular property of the $(q,t)$-characters of standard modules with respect to $\ol{(\cdot)}$
is included in Theorem~\ref{Thm:qtch}. That is, we have
$$
E_{t}(m) - \ol{E_{t}(m)} \in \sum_{m^{\prime} < m} \Z[t^{\pm 1}] E_{t}(m^{\prime}).
$$  
\end{Rem}
\subsection{The automorphisms $\fD_t$ and $\omega_t$}
\label{ssec:fDt}
Recall the automorphism $\fD$ of $\cY$ defined in Section~\ref{ssec:duality} (see (\ref{eq:deffD})).
As its $t$-analog,
we consider the $\Z[t^{\pm 1/2}]$-algebra automorphism $\fD_t$ of $\cY_{t}$ given by
\begin{equation} \label{eq:deffDt}
\fD_t(\tY_{i,p}) \seq \tY_{i^{*}, p+ rh^{\vee}} \quad \text{for $(i,p) \in \hI$}.
\end{equation}
Clearly, we have $\evt \circ \fD_t = \fD \circ \evt$. 
Note also that $\fD_t \circ \ol{(\cdot)} = \ol{(\cdot)} \circ \fD_t$ holds, and hence 
we have $\fD_t(\ul{m}) = \ul{\fD(m)}$ for any $m \in \cM$.
In particular, the set of the commutative monomials is stable under $\fD_t$.  
As a $t$-analog of  Theorem~\ref{Thm:fD}, 
we have the following.

\begin{Lem} \label{Lem:fD}
The automorphism $\fD_t$ of $\cY_{t}$ restricts to an automorphism on $\cK_{t}$,
for which we will use the same symbol. 
Moreover, we have
\begin{equation} \label{eq:fDt}
\fD_{t}(L_{t}(m)) = L_{t}(\fD(m))
\end{equation}
for any $m \in \cM$.
\end{Lem}
\begin{proof}

Obviously we have $\fD_{t}(\tA_{i,p+d_{i}}) = \tA_{i^{*}, p +d_{i}+rh^{\vee}}$ for each $(i,p) \in \hI$.
Thus the automorphism $\fD_{t}$ respects the subring $\cK_{t} \subset \cY_{t}$. 
Since $\fD_t( F_{t}(m))$ is an element in $\cK_t$ which contains the unique dominant monomial 
$\fD_t(\ul{m}) =  \ul{\fD(m)}$, we have
$\fD_t( F_{t}(m)) = F_{t}(\fD(m))$, 
and hence $\fD_{t}(E_{t}(m) )= E_{t}(\fD(m))$ for all $m \in \cM$.
Taking the fact $\fD_t \circ \ol{(\cdot)} = \ol{(\cdot)} \circ \fD_t$ into account, 
we obtain $\fD_t(L_{t}(m)) = L_{t}(\fD(m))$ for each $m \in \cM$ from
the characterization in Theorem~\ref{Thm:qtch}.
\end{proof}

Next we define the
$\Z$-algebra involution $\omega_{t}$ on the quantum torus $\cY_{t}$ given by 
$$
\omega_t(t^{\pm 1/2}) \seq t^{\mp 1/2}, \quad \omega_t(\tY_{i,p}) \seq \tY_{i, -p}^{-1}
$$    
for $(i,p) \in \hI$.
This is a $t$-analog of the involution $\omega$ considered in Section~\ref{ssec:duality}.
Note that it is well-defined because $\Nn(i,-p;j,-s) = - \Nn(i,p;j,s)$ holds by the definition~(\ref{eq:defNn}).
Now we prove the following assertion as a $t$-analog of (\ref{eq:omg}). 
\begin{Lem} \label{Lem:omgt}
The involution $\omega_t$ on $\cY_{t}$ restricts to an involution on $\cK_{t}$,
for which we will use the same symbol. 
Moreover, we have
\begin{equation} \label{eq:omgt}
\omega_{t}(L_{t}(m)) = L_{t}(\omega^*m)
\end{equation}
for any $m \in \cM$. Here $\omega^*m$ is the dominant monomial defined as in {\rm Section~\ref{ssec:duality}}.
\end{Lem}
\begin{proof}
Note that we have $\omega_t  \circ \ol{(\cdot)} = \ol{(\cdot)} \circ \omega_t$
and hence $\omega_t$ stabilizes the set of commutative monomials. 
For each $i \in I$, the subalgebra $\cK_{t,i} \subset \cY_t$ is stable under $\omega_t$ because we have 
$$
\omega_t(\tY_{i,p}(1+t^{-1}\tA_{i, p+d_i}^{-1})) = 
\tY_{i, -p}^{-1}(1+t\tA_{i, -p-d_i}) = M \cdot \tY_{i, -p-2d_i}(1+t^{-1}\tA^{-1}_{i, -p-d_i}),
$$      
where $M = t\tA_{i, -p-d_i} \tY_{i,-p}^{-1} \tY_{i, -p-2d_i}^{-1}$ is a product of $\tY_{j,s}$ with $j \sim i$ and a power of $t$.
Therefore, the quantum Grothendieck ring $\cK_{t}$ is also stable under $\omega_t$.

To prove the equality~(\ref{eq:omgt}), let us consider the \emph{anti}-algebra automorphism
$\bar{\omega}_t \seq \omega_t \circ \ol{(\cdot )} = \ol{(\cdot )} \circ \omega_t$ of $\cY_t$.
Note that it is $\Z[t^{\pm 1/2}]$-linear and satisfies $\bar{\omega}_t  \circ \ol{(\cdot)} = \ol{(\cdot)} \circ \bar{\omega}_t$.
Therefore, from Theorem~\ref{Thm:Ft} and Theorem \ref{Thm:qtfund}~(\ref{qtfund:monomials}), 
it follows that
$$\omega_t(F_{t}(Y_{i,p})) = \bar{\omega}_t(F_{t}(Y_{i,p})) = F_{t}(Y_{i^*, -p-rh^\vee})$$ 
for each $(i,p) \in \hI$. 
Since $\Nn(i^*, -p - rh^\vee;j^*,-s-rh^\vee) = \Nn(j,s;i,p)$ holds by Lemma~\ref{Lem:tc}, 
we obtain
$
\bar{\omega}_t(E_t(m)) = E_{t}(\omega^*m) 
$
for any $m \in \cM$. 
Then, by Theorem~\ref{Thm:qtch}, 
we obtain 
$\bar{\omega}_t(L_t(m)) = L_t(\omega^*m)$ and hence the equality~(\ref{eq:omgt}) for any $m \in \cM$.
\end{proof}

\begin{Rem}
In contrast to \eqref{eq:omgt}, the equality 
\begin{equation} \label{eq:omgtFt} 
\omega_{t}(F_{t}(m)) = F_{t}(\omega^*m)
\end{equation}
does not hold in general. 
In fact, the classical analog of \eqref{eq:omgtFt} already fails. Consider $F(m) = \text{ev}_{t = 1}(F_t(m))$, the element in the image of the $q$-character homomorphism with $m$ as a unique dominant monomial.
For $\mathfrak{g} = sl_2$, we have for $p\in\mathbb{Z}$ :
$$F(Y_p) = Y_p + Y_{p+2}^{-1} = \chi_q(L(Y_p))$$
and
$$F(Y_pY_{p+2})
= Y_pY_{p+2} + Y_p Y_{p+4}^{-1} + Y_{p+2}^{-1}Y_{p+4}^{-1} = \chi_q(L(Y_pY_{p+2})).$$
This implies
$$F(Y_0Y_2^2) = F(Y_0Y_2)F(Y_2)
= Y_0Y_2^2 + 2 Y_0Y_2Y_4^{-1} + Y_4^{-1} + Y_0Y_4^{-2}+ Y_2^{-1}Y_4^{-2}.$$
This element has a unique dominant monomial (in fact, it is equal to $\chi_q(L(Y_0Y_2^2))$).
But its image by $\omega$ has 2 dominant monomials, $Y_{-4}^2Y_{-2}$ and $Y_{-4}$.
So, it is not equal to
$$F(Y_{-4}^2Y_{-2}) = F(Y_{-4}Y_{-2})F(Y_{-4})-F(Y_{-4})$$
which has a monomial $Y_{-2}^{-1}$ with coefficient $-1$.

Later in the proof of Proposition~\ref{Prop:qTsys}, we will see that \eqref{eq:omgtFt} holds when $m$  corresponds to a Kirillov-Reshetikhin module.
\end{Rem}


\section{Combinatorics of Q-data}
\label{sec:Qdata}

In this section, we introduce some combinatorial gadgets arising from what we call \emph{Q-data} following \cite{FO21}.
Our exposition is slightly different from the original one in \cite{FO21}: we start with the simple Lie algebra $\fg$
and then introduce its unfolding $(\Delta, \sigma)$ as follows.


\subsection{Unfoldings}
\label{ssec:unf}

For each complex finite-dimensional simple Lie algebra $\fg$,
we associate a unique pair $(\Delta, \sigma)$, which we call \emph{the unfolding of $\fg$}, consisting of 
a simply-laced Dynkin diagram $\Delta$ and a graph automorphism $\sigma$ of $\Delta$
as given in the Table~\ref{table:cl},
where $\id \colon \Delta_0 \to \Delta_0$ is the identity map and
the automorphisms $\vee$ and $\tvee$ are given by the blue arrows in Figure~\ref{Fig:unf} below.  

\begin{table}[h]
\centering
 { \arraycolsep=1.6pt\def\arraystretch{1.5}
\begin{tabular}{|c|c|c|c|c|c|}
\hline
$r$ & $\fg$ & $\Delta$ (or $\sg$) & $\sigma$ & $h^{\vee}$ & $\ell_0 $ \\
\hline
\hline
& $\mathrm{A}_{n}$ &$\mathrm{A}_{n}$ & $\id$ & $n+1$ & $n(n+1)/2$ \\
 $1$ & $\mathrm{D}_{n}$ & $\mathrm{D}_{n}$ & $\id$ &$2n-2$ & $n(n-1)$ \\
  & $\mathrm{E}_{6,7,8}$ & $\mathrm{E}_{6,7,8}$ & $\id$ & $12, 18, 30$& $36, 63, 120$ \\
\hline
& $\mathrm{B}_{n}$ & $\mathrm{A}_{2n-1}$ & $\vee$ & $2n-1$ & $n(2n-1)$\\
$2$ & $\mathrm{C}_{n}$ & $\mathrm{D}_{n+1}$ & $\vee$ &  $n+1$ & $n(n+1)$ \\
 & $\mathrm{F}_{4}$ & $\mathrm{E}_{6}$ & $\vee$ & $9$ & $36$ \\
\hline
$3$ & $\mathrm{G}_{2}$ & $\mathrm{D}_{4}$ & $\widetilde{\vee}$ &  $4$ & $12$ \\
\hline
\end{tabular}
  }\\[1.5ex]
\caption{Unfoldings and associated numerical data} \label{table:cl} 
\end{table}

\begin{figure}[ht]
\begin{center}
\begin{tikzpicture}[xscale=1.25,yscale=.7]
\node (A2n1) at (-0.2,4.5) {$(\mathrm{A}_{2n-1}, \vee)$};
\node[dynkdot,label={below:\footnotesize$n+1$}] (A6) at (4,4) {};
\node[dynkdot,label={below:\footnotesize$n+2$}] (A7) at (3,4) {};
\node[dynkdot,label={below:\footnotesize$2n-2$}] (A8) at (2,4) {};
\node[dynkdot,label={below:\footnotesize$2n-1$}] (A9) at (1,4) {};
\node[dynkdot,label={above:\footnotesize$n-1$}] (A4) at (4,5) {};
\node[dynkdot,label={above:\footnotesize$n-2$}] (A3) at (3,5) {};
\node (Au) at (2.5, 5) {$\cdots$};
\node (Al) at (2.5, 4) {$\cdots$};
\node[dynkdot,label={above:\footnotesize$2$}] (A2) at (2,5) {};
\node[dynkdot,label={above:\footnotesize$1$}] (A1) at (1,5) {};
\node[dynkdot,label={above:\footnotesize$n$}] (A5) at (5,4.5) {};
\path[-]
 (A1) edge (A2)
 (A3) edge (A4)
 (A4) edge (A5)
 (A5) edge (A6)
 (A6) edge (A7)
 (A8) edge (A9);
\path[-] (A2) edge (Au) (Au) edge (A3) (A7) edge (Al) (Al) edge (A8);
\path[<->,thick,blue] (A1) edge (A9) (A2) edge (A8) (A3) edge (A7) (A4) edge (A6);
\path[->, thick, blue] (A5) edge [loop below] (A5);
\def\Foffset{6.5}
\node (Bn) at (-0.2,\Foffset) {$\mathrm{B}_n$};
\foreach \x in {1,2}
{\node[dynkdot,label={above:\footnotesize$\x$}] (B\x) at (\x,\Foffset) {};}
\node[dynkdot,label={above:\footnotesize$n-2$}] (B3) at (3,\Foffset) {};
\node[dynkdot,label={above:\footnotesize$n-1$}] (B4) at (4,\Foffset) {};
\node[dynkdot,label={above:\footnotesize$n$}] (B5) at (5,\Foffset) {};
\node (Bm) at (2.5,\Foffset) {$\cdots$};
\path[-] (B1) edge (B2) (B2) edge (Bm) (Bm) edge (B3) (B3) edge (B4);
\draw[-] (B4.30) -- (B5.150);
\draw[-] (B4.330) -- (B5.210);
\draw[-] (4.55,\Foffset) -- (4.45,\Foffset+.2);
\draw[-] (4.55,\Foffset) -- (4.45,\Foffset-.2);
\draw[-,dotted] (A1) -- (B1);
\draw[-,dotted] (A2) -- (B2);
\draw[-,dotted] (A3) -- (B3);
\draw[-,dotted] (A4) -- (B4);
\draw[-,dotted] (A5) -- (B5);
\draw[|->] (Bn) -- (A2n1);
\node (Dn1) at (-0.2,0) {$(\mathrm{D}_{n+1}, \vee)$};
\node[dynkdot,label={above:\footnotesize$1$}] (D1) at (1,0){};
\node[dynkdot,label={above:\footnotesize$2$}] (D2) at (2,0) {};
\node (Dm) at (2.5,0) {$\cdots$};
\node[dynkdot,label={above:\footnotesize$n-2$}] (D3) at (3,0) {};
\node[dynkdot,label={above:\footnotesize$n-1$}] (D4) at (4,0) {};
\node[dynkdot,label={above:\footnotesize$n$}] (D6) at (5,.5) {};
\node[dynkdot,label={below:\footnotesize$n+1$}] (D5) at (5,-.5) {};
\path[-] (D1) edge (D2)
  (D2) edge (Dm)
  (Dm) edge (D3)
  (D3) edge (D4)
  (D4) edge (D5)
  (D4) edge (D6);
\path[<->,thick,blue] (D6) edge (D5);
\path[->,thick,blue] (D1) edge [loop below] (D1)
(D2) edge [loop below] (D2)
(D3) edge [loop below] (D3)
(D4) edge [loop below] (D4);
\def\Coffset{1.8}
\node (Cn) at (-0.2,\Coffset) {$\mathrm{C}_n$};
\foreach \x in {1,2}
{\node[dynkdot,label={above:\footnotesize$\x$}] (C\x) at (\x,\Coffset) {};}
\node (Cm) at (2.5, \Coffset) {$\cdots$};
\node[dynkdot,label={above:\footnotesize$n-2$}] (C3) at (3,\Coffset) {};
\node[dynkdot,label={above:\footnotesize$n-1$}] (C4) at (4,\Coffset) {};
\node[dynkdot,label={above:\footnotesize$n$}] (C5) at (5,\Coffset) {};
\draw[-] (C1) -- (C2);
\draw[-] (C2) -- (Cm);
\draw[-] (Cm) -- (C3);
\draw[-] (C3) -- (C4);
\draw[-] (C4.30) -- (C5.150);
\draw[-] (C4.330) -- (C5.210);
\draw[-] (4.55,\Coffset+.2) -- (4.45,\Coffset) -- (4.55,\Coffset-.2);
\draw[-,dotted] (C1) -- (D1);
\draw[-,dotted] (C2) -- (D2);
\draw[-,dotted] (C3) -- (D3);
\draw[-,dotted] (C4) -- (D4);
\draw[-,dotted] (C5) -- (D6);
\draw[|->] (Cn) -- (Dn1);
\node (E6desc) at (6.8,4.5) {$(\mathrm{E}_6, \vee)$};
\node[dynkdot,label={above:\footnotesize$2$}] (E2) at (10.8,4.5) {};
\node[dynkdot,label={above:\footnotesize$4$}] (E4) at (9.8,4.5) {};
\node[dynkdot,label={above:\footnotesize$5$}] (E5) at (8.8,5) {};
\node[dynkdot,label={above:\footnotesize$6$}] (E6) at (7.8,5) {};
\node[dynkdot,label={below:\footnotesize$3$}] (E3) at (8.8,4) {};
\node[dynkdot,label={below:\footnotesize$1$}] (E1) at (7.8,4) {};
\path[-]
 (E2) edge (E4)
 (E4) edge (E5)
 (E4) edge (E3)
 (E5) edge (E6)
 (E3) edge (E1);
\path[<->,thick,blue] (E3) edge (E5) (E1) edge (E6);
\path[->, thick, blue] (E4) edge[loop below] (E4); 
\path[->, thick, blue] (E2) edge[loop below] (E2); 
\def\Foffset{6.5}
\node (F4desc) at (6.8,\Foffset) {$\mathrm{F}_4$};
\foreach \x in {1,2,3,4}
{\node[dynkdot,label={above:\footnotesize$\x$}] (F\x) at (\x+6.8,\Foffset) {};}
\draw[-] (F1.east) -- (F2.west);
\draw[-] (F3) -- (F4);
\draw[-] (F2.30) -- (F3.150);
\draw[-] (F2.330) -- (F3.210);
\draw[-] (9.35,\Foffset) -- (9.25,\Foffset+.2);
\draw[-] (9.35,\Foffset) -- (9.25,\Foffset-.2);
\draw[|->] (F4desc) -- (E6desc);
\path[-, dotted] (F1) edge (E6)
(F2) edge (E5) (F3) edge (E4) (F4) edge (E2);

\node (D4desc) at (6.8,0) {$(\mathrm{D}_{4}, \tvee)$};
\node[dynkdot,label={above:\footnotesize$1$}] (D1) at (7.8,.6){};
\node[dynkdot,label={above:\footnotesize$2$}] (D2) at (8.8,0) {};
\node[dynkdot,label={left:\footnotesize$3$}] (D3) at (7.8,0) {};
\node[dynkdot,label={below:\footnotesize$4$}] (D4) at (7.8,-.6) {};
\draw[-] (D1) -- (D2);
\draw[-] (D3) -- (D2);
\draw[-] (D4) -- (D2);
\path[->,blue,thick]
(D1) edge [bend left=0] (D3)
(D3) edge [bend left=0](D4)
(D4) edge[bend left=90] (D1);
\path[->, thick, blue] (D2) edge[loop below] (D2); 
\def\Goffset{1.8}
\node (G2desc) at (6.8,\Goffset) {$\mathrm{G}_2$};
\node[dynkdot,label={above:\footnotesize$1$}] (G1) at (7.8,\Goffset){};
\node[dynkdot,label={above:\footnotesize$2$}] (G2) at (8.8,\Goffset) {};
\draw[-] (G1) -- (G2);
\draw[-] (G1.40) -- (G2.140);
\draw[-] (G1.320) -- (G2.220);
\draw[-] (8.25,\Goffset+.2) -- (8.35,\Goffset) -- (8.25,\Goffset-.2);
\draw[|->] (G2desc) -- (D4desc);
\path[-, dotted] (D1) edge (G1) (D2) edge (G2);
\end{tikzpicture}
\end{center}
\caption{Unfoldings for non-simply-laced $\fg$} \label{Fig:unf}
\end{figure}
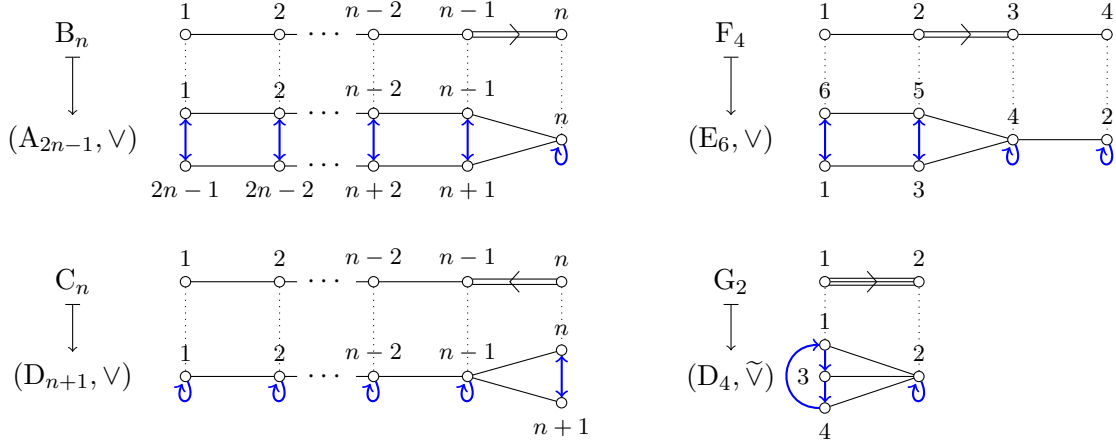

Let $\Delta_{0}$ denote the set of vertices of the graph $\Delta$. 
Under the above assignment $\fg \mapsto (\Delta, \sigma)$, we identify the lacing number $r$ of $\fg$ with the order of the automorphism $\sigma$, 
and also identify the set $I$ of Dynkin indices of $\fg$ with the set $\Delta_{0}/\langle \sigma \rangle$ of $\sigma$-orbits in $\Delta_0$ as 
suggested by the dotted lines in Figure~\ref{Fig:unf}.   
Then the positive integer $d_i \in \{ 1,r \}$ coincides with the cardinality of the $\sigma$-orbit corresponding to $i \in I$.     

In what follows, we use the symbols $\im, \jm, \ldots$ to denote the elements of $\Delta_0$.
We write $\im \sim \jm$ if they are adjacent in the graph $\Delta$.
We denote the natural quotient map $\Delta_0 \to I = \Delta_{0}/\langle \sigma \rangle$
by $\im \mapsto \bar{\im}$. 
By definition, for $\im \in \Delta_0$ and $i \in I$, we have $i = \bar{\im}$ if and only if $\im \in i$.
Note that we have $i \sim j$ for $i,j \in I$ (i.e.~$c_{ij} < 0$)
if and only if there exist $\im \in i$ and $\jm \in j$ satisfying $\im \sim \jm$.

\subsection{Notation for simply-laced root systems}
\label{ssec:notr}

Let $(\Delta, \sigma)$ be the unfolding of $\fg$.
Hereafter, we always denote by $\sg$ the simply-laced Lie algebra 
associated with the Dynkin diagram $\Delta$.
In the following paragraphs, we prepare the notation around the root system of $\sg$ (not $\fg$).
   
Let $\sP=\bigoplus_{\im \in \Delta_{0}} \Z \varpi_{\im}$ denote
the weight lattice of $\sg$,
where $\{\varpi_{\im}\}_{\im \in \Delta_{0}}$ are the fundamental weights.
For each $\im \in \Delta_0$, let $\alpha_{\im} \seq 2 \varpi_{\im} - \sum_{\jm \sim \im} \varpi_{\jm} \in \sP$
denote the corresponding simple root.
We set
$\sQ \seq \bigoplus_{\im \in \Delta_{0}} \Z \alpha_{\im}$ and
$\sQ^{+} \seq \sum_{\im \in \Delta_{0}} \Z_{\ge 0} \alpha_{\im}$.
Let $( \cdot , \cdot ) \colon \sP \times \sP \to \Q$
denote the symmetric bilinear form defined by
$(\varpi_{\im}, \alpha_{\jm}) = \delta_{\im, \jm}$ for $\im, \jm \in \Delta_{0}$.
The Weyl group $\sW$ of $\sg$ is defined as the subgroup of $\Aut(\sP)$
generated by the simple reflections $\{ s_{\im} \}_{\im \in \Delta_{0}}$, which are given by
$s_{\im}(\lambda) = \lambda - (\lambda, \alpha_{\im})\alpha_{\im}$ for $\lambda \in \sP$.
The set of roots is defined by $\sR \seq \sW\cdot\{\alpha_{\im}\}_{\im \in \Delta_{0}}$.
We have the decomposition $\sR = \sR^{+} \sqcup (-\sR^{+})$, where
$\sR^{+} \seq \sR \cap \sQ^{+}$ is the set of positive roots.
In the sequel, we often identify $\sigma$ with the automorphism $\sigma \in \Aut(\sP)$
given by $\sigma \cdot \varpi_{\im} = \varpi_{\sigma(\im)}$ for $\im \in \Delta_0$.
Then the pairing $(\cdot, \cdot)$ on $\sP$ is invariant under 
the action of the extended group $\sW \rtimes \langle \sigma \rangle \subset \Aut(\sP)$. 

For an element $w \in \sW$,
a sequence $\bfi = (\im_{1}, \ldots, \im_{l})$ of elements of $\Delta_{0}$
is called a reduced word for $w$ if
it satisfies $w=s_{\im_{1}}\cdots s_{\im_{l}}$ and
$l$ is the smallest among all the sequences with this property.
The length of a reduced word of $w$ is called the length of $w$, denoted  by $\ell(w) \in \Z_{\ge 0}$.
It is well-known that there exists a unique element $w_{0} \in \sW$ with the largest length
$\ell(w_{0}) = \ell_0 \seq \#\sR^{+}$.
From Table~\ref{table:cl}, we can easily check the equality
$n r h^{\vee} = 2 \ell_{0}$
holds for each $\fg$.
Here we recall that $n$ and $h^\vee$ denote the rank and the dual Coxeter number of $\fg$ (not $\sg$) respectively.  

We define the involution $\im \mapsto \im^{*}$
on $\Delta_{0}$ by the relation $w_{0}(\alpha_{\im}) = -\alpha_{\im^{*}}$.
Under the assumption $\sigma \neq \id$,
we have $* = \id$ if $h^{\vee}$ is even, and $* = \sigma$ if $h^{\vee}$ is odd.
In particular, the involution $*$ on $\Delta_{0}$ induces an involution
on the set $I$, which is trivial for non-simply-laced $\fg$.
This latter involution on $I$ coincides with the involution $i \mapsto i^*$ that appeared
in the definition of $\fD$ (see (\ref{eq:deffD})). 
In other words, we have $(\bar{\im})^* = \ol{\im^*}$ for $\im \in \Delta_0$.


\subsection{Q-data and generalized Coxeter elements}
\label{ssec:Qdata}

Let $\fg$ be a finite-dimensional complex simple Lie algebra of rank $n$ and
$(\Delta, \sigma)$ the unfolding  of $\fg$.

\begin{Def} \label{def:Qdata}
\emph{A height function} on $(\Delta, \sigma)$ is a 
function $\xi \colon \Delta_{0} \to \Z$ satisfying the following conditions (here we put $\xi_{\im} \seq \xi(\im)$):
\begin{enumerate}
\item Let $\im, \jm \in \Delta_{0}$ with $\im \sim \jm$ and $d_{\bar{\im}}=d_{\bar{\jm}}$.
Then we have $|\xi_{\im} - \xi_{\jm}| = d_{\bar{\im}}=d_{\bar{\jm}}$.
\item Let $i, j \in I$ with $i \sim j$ and $d_{i} = 1 < d_{j}=r$.
Then there is a unique $\jm \in j$ such that $|\xi_{\im} - \xi_{\jm}| = 1$ and
$\xi_{\sigma^{k}(\jm)} = \xi_{\jm} - 2k$ for any $1 \le k < r$, where $i=\{\im\}$. 
\end{enumerate} 
Such a triple $\cQ = (\Delta, \sigma, \xi)$ is called \emph{a Q-datum} for $\fg$.  
\end{Def}

\begin{Rem} \label{Rem:quiver}
When $\fg$ is simply-laced, $\sigma$ is the identity and
a height function $\xi$ on $(\Delta, \id)$ is simply a function $\xi \colon \Delta_0 \to \Z$
satisfying $|\xi_{\im} - \xi_{\jm}|=1$ for each adjacent pair $(\im, \jm)$, 
namely a height function on a Dynkin quiver. 
In this sense, the notion of Q-datum is a generalization of 
the notion of a Dynkin quiver with a height function. 
\end{Rem}

\begin{Def} \label{Def:adapt}
Let $\cQ = (\Delta, \sigma, \xi)$ be a Q-datum for $\fg$. 
A vertex $\im \in \Delta_{0}$ is called \emph{a source of $\cQ$} if we have $\xi_{\im} > \xi_{\jm}$
for any $\jm \in \Delta_{0}$ with $\jm \sim \im$.
When $\im$ is a source of $\cQ$, we define a new Q-datum $s_{\im}\cQ = (\Delta, \sigma, s_{\im}\xi)$
by setting 
$$
(s_{\im} \xi)_{\jm} \seq \xi_{\jm} - 2 d_{\bar{\im}}\delta_{\im,\jm} \quad \text{for $\jm \in \Delta_{0}$}. 
$$
We say that a sequence $\bfi = (\im_1, \ldots, \im_l)$ of elements of $\Delta_0$ 
is \emph{adapted to $\cQ$} if $\im_k$ is a source of 
the Q-datum $s_{\im_{k-1}} \cdots s_{\im_2} s_{\im_1} \cQ$ 
for all $k = 1, 2, \ldots, l$.  
\end{Def}

Given a Q-datum $(\Delta, \sigma, \xi)$,
for each $i \in I$, denote by $i^{\circ}$ the unique vertex in the $\sigma$-orbit $i$
satisfying $\xi_{i^{\circ}} = \max \{ \xi_{\im} \mid \im \in i\}$.
We consider the following condition on $\cQ$:
\begin{equation} \label{cond:Qtau}
\text{for each $i \in I$ and $1 \le k < d_{i}$, we have 
$\xi_{\sigma^{k}(i^{\circ})} = \xi_{i^{\circ}} - 2k$}.
\end{equation}

\begin{Prop}[{\cite[Section~3.6]{FO21}}] \label{Prop:Cox}
There is a unique collection $\{ \tau_{\cQ}\}_{\cQ} \subset \sW \rtimes \langle \sigma \rangle$ 
labelled by Q-data for $\fg$ and characterized by the following conditions$\colon$
\begin{enumerate}
\item If $\cQ$ satisfies {\rm (\ref{cond:Qtau})}, we define
$\tau_{\cQ} \seq s_{i_{1}^{\circ}} \cdots s_{i_{n}^{\circ}}\sigma$, where 
$(i_{1}, \ldots, i_{n})$ is any total ordering of $I$ satisfying $\xi_{i_{1}^{\circ}} \ge \cdots \ge \xi_{i_{n}^{\circ}}$.
$($In this case, the sequence $(i^{\circ}_{1}, \ldots, i^{\circ}_{n})$ is a reduced word adapted to $\cQ$.$)$
\item If $\im \in \Delta_{0}$ is a source of $\cQ$, we have
$\tau_{s_{\im}\cQ} = s_{\im} \tau_{\cQ} s_{\im}$.
\end{enumerate}  
\end{Prop}

\begin{Rem}
When $\fg$ is simply-laced, 
a Q-datum $\cQ$ for $\fg$ can be seen as a Dynkin quiver $Q$ with a height function
(see~Remark~\ref{Rem:quiver}) and the condition (\ref{cond:Qtau}) is trivially satisfied in this case.
Then the element $\tau_{\cQ}$ is nothing but the unique Coxeter element adapted to the quiver $Q$
(in the usual sense).
\end{Rem}

\begin{Def} 
We refer to the unique element $\tau_{\cQ}$ in Proposition~\ref{Prop:Cox} as \emph{the generalized Coxeter element}
associated with $\cQ$.
\end{Def}

\begin{Prop}[{\cite[Proposition 3.34]{FO21}}] \label{Prop:Coxord}
Let $\cQ$ be a Q-datum for $\fg$. Then the order of $\tau_{\cQ}$ is $rh^{\vee}$.
If $\fg$ is not simply-laced, we have $\tau_{\cQ}^{rh^{\vee}/2} = -1$.
\end{Prop}
\subsection{Repetition quivers and their $\cQ$-coordinates}
\label{ssec:repet}
Fix a simple Lie algebra $\fg$ and its unfolding $(\Delta, \sigma)$ as above.
Recall the notation $d_{ij} \seq \min(d_{i}, d_{j})$ for $i,j \in I$.

\begin{Def} \label{Def:repQ}
Let $\xi$ be a height function on $(\Delta, \sigma)$.  
\emph{The repetition quiver associated with $(\Delta, \sigma)$} is the quiver $\hDs$ whose
vertex set $\hDs_{0}$ and arrow set $\hDs_{1}$ are given by
\begin{align*}
\hDs_{0} &\seq \{ (\im, p) \in \Delta_{0} \times \Z \mid p-\xi_{\im} \in 2d_{\bar{\im}} \Z \}, \\
\hDs_{1} &\seq \{ (\im,p) \to (\jm, s) \mid (\im, p), (\jm, s) \in \hDs_{0}, \jm \sim \im, s-p = d_{\bar{\im}\bar{\jm}} \}.
\end{align*}
\end{Def}
\begin{Rem}
Although the repetition quiver $\hDs$ depends on the choice of height function $\xi$,
it is not essential because any two repetition quivers defined with different choices of height functions   
are identical to each other after shifting uniformly by an integer. Thus we suppress the dependence on $\xi$ in the notation $\hDs$. 
\end{Rem}

\begin{Ex}
We exhibit some examples of the repetition quivers. Here the labeling of $\Delta_0$ is as in Figure~\ref{Fig:unf}. 
\begin{enumerate}
\item 
When $\fg$ is of type $\mathrm{A}_{5}$, its unfolding is $(\Delta, \id)$ with $\Delta$ of type $\mathrm{A}_5$  and
the repetition quiver $\hDs$ is depicted as:
$$
\raisebox{3mm}{
\scalebox{0.60}{\xymatrix@!C=0.5mm@R=2mm{
(\im\setminus p) & -8 & -7 & -6 &-5&-4 &-3& -2 &-1& 0 & 1& 2 & 3& 4&  5
& 6 & 7 & 8 & 9 & 10 & 11 & 12 & 13 & 14 & 15 & 16  \\
1&\bullet \ar@{->}[dr]&& \bullet \ar@{->}[dr] &&\bullet\ar@{->}[dr]
&&\bullet \ar@{->}[dr] && \bullet \ar@{->}[dr] &&\bullet \ar@{->}[dr] &&  \bullet \ar@{->}[dr] 
&&\bullet \ar@{->}[dr] && \bullet \ar@{->}[dr] &&\bullet \ar@{->}[dr]  && \bullet\ar@{->}[dr] && \bullet \ar@{->}[dr] && \bullet \\
2&&\bullet \ar@{->}[dr]\ar@{->}[ur]&& \bullet \ar@{->}[dr]\ar@{->}[ur] &&\bullet \ar@{->}[dr]\ar@{->}[ur]
&& \bullet \ar@{->}[dr]\ar@{->}[ur]&& \bullet\ar@{->}[dr] \ar@{->}[ur]&& \bullet \ar@{->}[dr]\ar@{->}[ur]&&\bullet \ar@{->}[dr]\ar@{->}[ur]&
&\bullet \ar@{->}[dr]\ar@{->}[ur]&&\bullet\ar@{->}[dr] \ar@{->}[ur]&& \bullet \ar@{->}[dr]\ar@{->}[ur]
&&\bullet \ar@{->}[dr]\ar@{->}[ur]&& \bullet \ar@{->}[dr] \ar@{->}[ur]& \\ 
3&\bullet \ar@{->}[dr] \ar@{->}[ur]&& \bullet \ar@{->}[dr] \ar@{->}[ur] &&\bullet\ar@{->}[dr] \ar@{->}[ur]
&&\bullet \ar@{->}[dr] \ar@{->}[ur] && \bullet \ar@{->}[dr]\ar@{->}[ur] &&\bullet \ar@{->}[dr] \ar@{->}[ur]&&  \bullet \ar@{->}[dr] \ar@{->}[ur]
&&\bullet \ar@{->}[dr] \ar@{->}[ur] && \bullet \ar@{->}[dr] \ar@{->}[ur]&&\bullet \ar@{->}[dr] \ar@{->}[ur] && \bullet\ar@{->}[dr] \ar@{->}[ur]&&
\bullet\ar@{->}[dr] \ar@{->}[ur] && \bullet \\
4&& \bullet \ar@{->}[ur]\ar@{->}[dr]&&\bullet \ar@{->}[ur]\ar@{->}[dr]&&\bullet \ar@{->}[ur]\ar@{->}[dr] &&\bullet \ar@{->}[ur]\ar@{->}[dr]&& \bullet \ar@{->}[ur]\ar@{->}[dr]
&&\bullet \ar@{->}[ur]\ar@{->}[dr]&& \bullet \ar@{->}[ur]\ar@{->}[dr] &&\bullet \ar@{->}[ur]\ar@{->}[dr]&&\bullet \ar@{->}[ur]\ar@{->}[dr]&&
\bullet \ar@{->}[ur]\ar@{->}[dr]&&\bullet \ar@{->}[ur]\ar@{->}[dr]&&\bullet \ar@{->}[ur]\ar@{->}[dr]& \\
5&\bullet  \ar@{->}[ur]&& \bullet  \ar@{->}[ur] &&\bullet \ar@{->}[ur]
&&\bullet  \ar@{->}[ur] && \bullet \ar@{->}[ur] &&\bullet  \ar@{->}[ur]&&  \bullet  \ar@{->}[ur]
&&\bullet  \ar@{->}[ur] && \bullet  \ar@{->}[ur]&&\bullet  \ar@{->}[ur] && \bullet \ar@{->}[ur]&&
\bullet \ar@{->}[ur] && \bullet  }}}
$$
\item 
When $\fg$ is of type $\mathrm{D}_{5}$, its unfolding is $(\Delta, \id)$ with $\Delta$ of type $\mathrm{D}_5$  and
the repetition quiver $\hDs$ is depicted as:
$$
\raisebox{3mm}{
\scalebox{0.60}{\xymatrix@!C=0.5mm@R=2mm{
(\im\setminus p) & -8 & -7 & -6 &-5&-4 &-3& -2 &-1& 0 & 1& 2 & 3& 4&  5
& 6 & 7 & 8 & 9 & 10 & 11 & 12 & 13 & 14 & 15 & 16 \\
1&\bullet \ar@{->}[dr]&& \bullet \ar@{->}[dr] &&\bullet\ar@{->}[dr]
&&\bullet \ar@{->}[dr] && \bullet \ar@{->}[dr] &&\bullet \ar@{->}[dr] &&  \bullet \ar@{->}[dr] 
&&\bullet \ar@{->}[dr] && \bullet \ar@{->}[dr] &&\bullet \ar@{->}[dr]  && \bullet\ar@{->}[dr] &&
\bullet\ar@{->}[dr] && \bullet \\
2&&\bullet \ar@{->}[dr]\ar@{->}[ur]&& \bullet \ar@{->}[dr]\ar@{->}[ur] &&\bullet \ar@{->}[dr]\ar@{->}[ur]
&& \bullet \ar@{->}[dr]\ar@{->}[ur]&& \bullet\ar@{->}[dr] \ar@{->}[ur]&& \bullet \ar@{->}[dr]\ar@{->}[ur]&&\bullet \ar@{->}[dr]\ar@{->}[ur]&
&\bullet \ar@{->}[dr]\ar@{->}[ur]&&\bullet\ar@{->}[dr] \ar@{->}[ur]&& \bullet \ar@{->}[dr]\ar@{->}[ur]
&&\bullet \ar@{->}[dr]\ar@{->}[ur]&& \bullet \ar@{->}[dr] \ar@{->}[ur]& \\ 
3&\bullet \ar@{->}[dr] \ar@{->}[ur]&& \bullet \ar@{->}[dr] \ar@{->}[ur] &&\bullet\ar@{->}[dr] \ar@{->}[ur]
&&\bullet \ar@{->}[dr] \ar@{->}[ur] && \bullet \ar@{->}[dr]\ar@{->}[ur] &&\bullet \ar@{->}[dr] \ar@{->}[ur]&&  \bullet \ar@{->}[dr] \ar@{->}[ur]
&&\bullet \ar@{->}[dr] \ar@{->}[ur] && \bullet \ar@{->}[dr] \ar@{->}[ur]&&\bullet \ar@{->}[dr] \ar@{->}[ur] && \bullet\ar@{->}[dr] \ar@{->}[ur]&&
\bullet\ar@{->}[dr] \ar@{->}[ur] && \bullet \\
4&& \bullet \ar@{->}[ur]&&\bullet \ar@{->}[ur]&&\bullet \ar@{->}[ur] &&\bullet \ar@{->}[ur]&& \bullet \ar@{->}[ur]
&&\bullet \ar@{->}[ur]&& \bullet \ar@{->}[ur] &&\bullet \ar@{->}[ur]&&\bullet \ar@{->}[ur]&&
\bullet \ar@{->}[ur]&&\bullet \ar@{->}[ur]&&\bullet \ar@{->}[ur]& \\
5&& \bullet \ar@{<-}[uul]\ar@{->}[uur]&&\bullet \ar@{<-}[uul]\ar@{->}[uur]&&\bullet \ar@{<-}[uul]\ar@{->}[uur] &&\bullet \ar@{<-}[uul]\ar@{->}[uur]&& \bullet \ar@{<-}[uul]\ar@{->}[uur]
&&\bullet \ar@{<-}[uul]\ar@{->}[uur]&& \bullet \ar@{<-}[uul]\ar@{->}[uur] &&\bullet \ar@{<-}[uul]\ar@{->}[uur]&&\bullet \ar@{<-}[uul]\ar@{->}[uur]&&
\bullet \ar@{<-}[uul]\ar@{->}[uur]&&\bullet \ar@{<-}[uul]\ar@{->}[uur]&&\bullet \ar@{<-}[uul]\ar@{->}[uur]&
}}}
$$
\item
When $\fg$ is of type $\mathrm{B}_{3}$, its unfolding is $(\Delta, \vee)$ with $\Delta$ of type $\mathrm{A}_5$  and
the repetition quiver $\hDs$ is depicted as:
$$
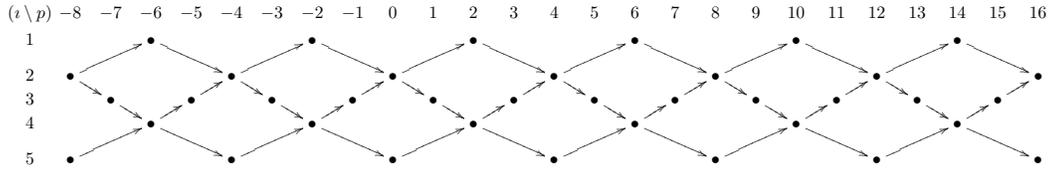
\raisebox{3mm}{
\scalebox{0.60}{\xymatrix@!C=0.5mm@R=0.5mm{
(\im\setminus p) & -8 & -7 & -6 &-5&-4 &-3& -2 &-1& 0 & 1& 2 & 3& 4&  5
& 6 & 7 & 8 & 9 & 10 & 11 & 12 & 13 & 14 & 15 & 16 \\
1&&&\bullet \ar@{->}[ddrr]&&&& \bullet \ar@{->}[ddrr]&&&&  \bullet \ar@{->}[ddrr] &&&& \bullet\ar@{->}[ddrr]
&&&& \bullet \ar@{->}[ddrr]&&&& \bullet \ar@{->}[ddrr] && \\ \\
2&\bullet \ar@{->}[dr] \ar@{->}[uurr]&&&&\bullet\ar@{->}[dr] \ar@{->}[uurr]
&&&& \bullet \ar@{->}[dr]\ar@{->}[uurr] &&&&  \bullet \ar@{->}[dr] \ar@{->}[uurr]
&&&& \bullet \ar@{->}[dr]\ar@{->}[uurr]&&&& \bullet \ar@{->}[dr]\ar@{->}[uurr]&&&& \bullet \\
3&& \bullet \ar@{->}[dr]&& \bullet \ar@{->}[ur] &&\bullet \ar@{->}[dr] && \bullet \ar@{->}[ur] && \bullet \ar@{->}[dr]
&& \bullet \ar@{->}[ur] && \bullet \ar@{->}[dr] &&
\bullet \ar@{->}[ur]&&\bullet \ar@{->}[dr] &&\bullet \ar@{->}[ur]&&\bullet \ar@{->}[dr] &&\bullet \ar@{->}[ur]& \\
4&&&\bullet\ar@{->}[ddrr]\ar@{->}[ur]&&&&\bullet\ar@{->}[ddrr]\ar@{->}[ur] &&&&  \bullet \ar@{->}[ddrr]\ar@{->}[ur] &&&&
\bullet \ar@{->}[ddrr] \ar@{->}[ur]
&&&& \bullet \ar@{->}[ddrr]\ar@{->}[ur]&&&& \bullet\ar@{->}[ddrr]\ar@{->}[ur] && \\ \\
5& \bullet  \ar@{->}[uurr]&&&&\bullet \ar@{->}[uurr] &&&& \bullet \ar@{->}[uurr]
&&&& \bullet   \ar@{->}[uurr] &&&& \bullet \ar@{->}[uurr]
&&&&\bullet \ar@{->}[uurr] &&&& \bullet  }}}
$$
\item 
When $\fg$ is of type $\mathrm{C}_{4}$, its unfolding is $(\Delta, \vee)$ with $\Delta$ of type $\mathrm{D}_5$  and
the repetition quiver $\hDs$ is depicted as:
$$
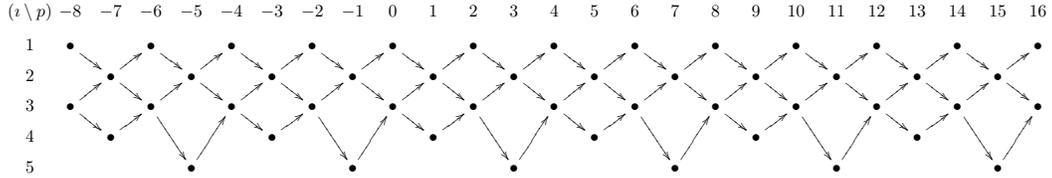
\raisebox{3mm}{
\scalebox{0.60}{\xymatrix@!C=0.5mm@R=2mm{
(\im\setminus p) & -8 & -7 & -6 &-5&-4 &-3& -2 &-1& 0 & 1& 2 & 3& 4&  5
& 6 & 7 & 8 & 9 & 10 & 11 & 12 & 13 & 14 & 15 & 16 \\
1&\bullet \ar@{->}[dr]&& \bullet \ar@{->}[dr] &&\bullet\ar@{->}[dr]
&&\bullet \ar@{->}[dr] && \bullet \ar@{->}[dr] &&\bullet \ar@{->}[dr] &&  \bullet \ar@{->}[dr] 
&&\bullet \ar@{->}[dr] && \bullet \ar@{->}[dr] &&\bullet \ar@{->}[dr]  && \bullet\ar@{->}[dr] &&
\bullet\ar@{->}[dr] && \bullet \\
2&&\bullet \ar@{->}[dr]\ar@{->}[ur]&& \bullet \ar@{->}[dr]\ar@{->}[ur] &&\bullet \ar@{->}[dr]\ar@{->}[ur]
&& \bullet \ar@{->}[dr]\ar@{->}[ur]&& \bullet\ar@{->}[dr] \ar@{->}[ur]&& \bullet \ar@{->}[dr]\ar@{->}[ur]&&\bullet \ar@{->}[dr]\ar@{->}[ur]&
&\bullet \ar@{->}[dr]\ar@{->}[ur]&&\bullet\ar@{->}[dr] \ar@{->}[ur]&& \bullet \ar@{->}[dr]\ar@{->}[ur]
&&\bullet \ar@{->}[dr]\ar@{->}[ur]&& \bullet \ar@{->}[dr] \ar@{->}[ur]& \\ 
3&\bullet \ar@{->}[dr] \ar@{->}[ur]&& \bullet \ar@{->}[ddr] \ar@{->}[ur] &&\bullet\ar@{->}[dr] \ar@{->}[ur]
&&\bullet \ar@{->}[ddr] \ar@{->}[ur] && \bullet \ar@{->}[dr]\ar@{->}[ur] &&\bullet \ar@{->}[ddr] \ar@{->}[ur]&&  \bullet \ar@{->}[dr] \ar@{->}[ur]
&&\bullet \ar@{->}[ddr] \ar@{->}[ur] && \bullet \ar@{->}[dr] \ar@{->}[ur]&&\bullet \ar@{->}[ddr] \ar@{->}[ur] && \bullet\ar@{->}[dr] \ar@{->}[ur]&&
\bullet\ar@{->}[ddr] \ar@{->}[ur] && \bullet \\
4&& \bullet \ar@{->}[ur]&&&&\bullet \ar@{->}[ur] &&&& \bullet \ar@{->}[ur]
&&&& \bullet \ar@{->}[ur] &&&&\bullet \ar@{->}[ur]&&&&\bullet \ar@{->}[ur]&&& \\
5&&&&\bullet \ar@{->}[uur]&&&&\bullet \ar@{->}[uur]&&&&  \bullet \ar@{->}[uur]&&&&
\bullet \ar@{->}[uur]&&&& \bullet \ar@{->}[uur]&&&& \bullet \ar@{->}[uur]& }}}
$$
\item
When $\fg$ is of type $\mathrm{G}_{2}$, its unfolding is $(\Delta, \tvee)$ with $\Delta$ of type $\mathrm{D}_4$  and
the repetition quiver $\hDs$ is depicted as:
$$
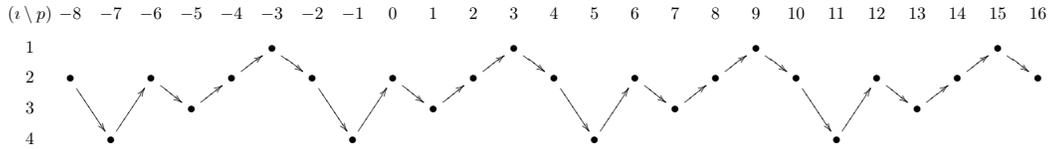
\raisebox{3mm}{
\scalebox{0.60}{\xymatrix@!C=0.5mm@R=2mm{
(\im\setminus p) & -8 & -7 & -6 &-5&-4 &-3& -2 &-1& 0 & 1& 2 & 3& 4&  5
& 6 & 7 & 8 & 9 & 10 & 11 & 12 & 13 & 14 & 15 & 16 \\
1&&&&&&\bullet \ar@{->}[dr]&&&&&&\bullet \ar@{->}[dr]&&& 
&&&\bullet \ar@{->}[dr]&&&&&&\bullet \ar@{->}[dr]&\\
2&\bullet \ar@{->}[ddr]&&\bullet \ar@{->}[dr]&&\bullet\ar@{->}[ur] &&
\bullet \ar@{->}[ddr]&&\bullet \ar@{->}[dr]&&\bullet\ar@{->}[ur] &&
\bullet \ar@{->}[ddr]&&\bullet \ar@{->}[dr]&&\bullet\ar@{->}[ur] &&
\bullet \ar@{->}[ddr]&&\bullet \ar@{->}[dr]&&\bullet\ar@{->}[ur] &&
\bullet \\
3&&&&\bullet \ar@{->}[ur] &&
&&&&\bullet \ar@{->}[ur] &&
&&&&\bullet \ar@{->}[ur] &&
&&&&\bullet \ar@{->}[ur] &&&\\
4&& \bullet \ar@{->}[uur]&&&&&& 
\bullet \ar@{->}[uur]&&&&&& 
\bullet \ar@{->}[uur]&&&&&& 
\bullet \ar@{->}[uur]&&&&& \\
}}}
$$
\end{enumerate}
\end{Ex}

Recall we have fixed a parity function $\epsilon \colon I \to \{0,1\}$ for $\fg$
and defined the set $\hI$ in~(\ref{eq:hI}). 
In what follows, we always assume
that a height function $\xi$ on $(\Delta, \sigma)$ always satisfies 
the condition 
\begin{equation} \label{cond:xiparity}
\xi_{\im} \equiv \epsilon_{\bar{\im}} \pmod{2} \quad \text{for any $\im \in \Delta_{0}$}.
\end{equation}
Note that this assumption is not a loss of generality.
Then we can define the bijection
$$f \colon \hDs_{0} \to \hI \quad \text{given by $(\im, p) \mapsto (\bar{\im}, p)$},$$
which we call \emph{the folding map}.  

Next we  shall attach a bijection  $\phi_{\cQ} \colon \hDs_{0} \to \sR^{+} \times \Z$  
to each Q-datum $\cQ = (\Delta, \sigma, \xi)$ for $\fg$,
which we call \emph{the $\cQ$-coordinate of $\hDs$}.
Using the generalized Coxeter element $\tau_{\cQ}$ associated with $\cQ$,
 we define a positive root $\gamma^{\cQ}_{\im} \in \sR^{+}$ for each $\im \in \Delta_{0}$ by
\begin{equation} \label{eq:gamma}
\gamma_{\im}^{\cQ} \seq (1-\tau_{\cQ}^{d_{\bar{\im}}})\varpi_{\im}.
\end{equation}
Then the bijection $\phi_{\cQ} \colon \hDs_{0} \to \sR^{+} \times \Z$ is defined recursively by
the following rules:
\begin{enumerate}
\item $\phi_{\cQ}(\im, \xi_{\im})=(\gamma^{\cQ}_{\im},0)$ for each $\im \in \Delta_{0}$.
\item If $\phi_\cQ(\im,p)=(\alpha, k)$, we have
$$
\phi_{\cQ}(\im, p\pm2d_{\bar{\im}}) =
\begin{cases}
(\tau_{\cQ}^{\mp d_{\bar{\im}}}\alpha,k) & \text{if $\tau_{\cQ}^{\mp d_{\bar{\im}}}\alpha \in \sR^+$}, \\
(-\tau_{\cQ}^{\mp d_{\bar{\im}}}\alpha, k\pm1) & \text{if $\tau_{\cQ}^{\mp d_{\bar{\im}}}\alpha \in -\sR^{+}$}.
\end{cases}
$$
\end{enumerate}
Composing with the bijection $f^{-1} \colon \hI \to \hDs_0$, we get another bijection
$$\bphi_{\cQ} \seq \phi_{\cQ} \circ (f^{-1}) \colon \hI \to \sR^{+} \times \Z,$$
which we call \emph{the $\cQ$-coordinate of $\hI$}.

\begin{Rem}
When $\fg$ is simply-laced (or equivalently, $\sigma = \id$),
we regard a Q-datum as a Dynkin quiver $Q$ with a height function as in Remark~\ref{Rem:quiver}.
Then the corresponding coordinate map $\phi_{Q}$ can be interpreted 
as an isomorphism between the repetitive quiver $\widehat{\Delta}^{\id}$ and 
the Auslander-Reiten quiver of the bounded derived category of representations of the quiver $Q$
due to Happel~\cite{Happel87}. See~\cite[Section~2]{HL15} or \cite[Section~3]{FO21} for more details. 
\end{Rem}

Finally, we recall a connection between the coordinate map $\bphi_{\cQ}$
and the inverse of quantum Cartan matrices.
Let us consider the map
$$\pi \colon \sR^{+} \times \Z \to \sQ
\quad
\text{ given by $ 
\pi(\alpha, k) \seq (-1)^{k} \alpha$}.
$$

\begin{Lem}[{\cite[(3.18)]{FO21}}] \label{Lem:bphi}
For a Q-datum $\cQ$ and $(i,p) \in \hI$, we have
$$ \pi( \bphi_{\cQ}(i,p)) =
\tau_{\cQ}^{(\xi_{\im}-p)/2}\gamma^{\cQ}_{\im}
\quad \text{for any $\im \in i$.}
$$
\end{Lem}

\begin{Thm}[cf.~{\cite[Theorem 4.8]{FO21}}] \label{Thm:tc} 
Let $\cQ = (\Delta, \sigma, \xi)$ be a Q-datum for $\fg$.
For $i,j \in I$ and any $u \in \Z$, we have
\begin{equation} \label{eq:tc}
\tc_{ij}(u) - \tc_{ij}(-u) = \begin{cases}
(\varpi_{\im}, \tau_{\cQ}^{(u + \xi_{\jm} - \xi_{\im} - d_{i})/2}\gamma^{\cQ}_{\jm}) 
& \text{if $u + \epsilon_{i} + \epsilon_{j} + d_{i} \in 2\Z$}, \\
0 & \text{otherwise},
\end{cases}
\end{equation}
where $\im, \jm \in \Delta_{0}$ are any vertices such that $\im \in i, \jm \in j$.
\end{Thm}
\begin{proof}
Note that the RHS of (\ref{eq:tc}) does not depend on the choice of vertices $\im \in i, \jm \in j$
(cf.~\cite[Lemma 4.6]{FO21}). 
When $u \ge 0$, Lemma~\ref{Lem:tc}~(\ref{tc:d}) tells that
the LHS of (\ref{eq:tc}) is just $\tc_{ij}(u)$ and then the assertion has been proved in~\cite[Theorem 4.8]{FO21}.
Letting $c(u)$ denote the LHS of (\ref{eq:tc}), we see that $c(u)$ is $2rh^{\vee}$-periodic 
as a function in $u \in \Z$, namely we have $c(u+2rh^\vee) = c(u)$ for all $u \in \Z$.
Indeed, it is clear unless $-2rh^\vee < u < 0$ from Lemma~\ref{Lem:tc}~(\ref{tc:pe}).
Even if $-2rh^\vee < u < 0$, we have 
\[ c(u+2rh^\vee) = \tc_{ij}(u+2rh^\vee) = - \tc_{ij}(-u) = c(u)\]
thanks to Lemma~\ref{Lem:tc}~\eqref{tc:pe} \& \eqref{tc:sym}.
On the other hand, the RHS of \eqref{eq:tc} is also $2rh^{\vee}$-periodic as $\tau_{\cQ}^{rh^{\vee}} = 1$ by Proposition~\ref{Prop:Coxord}.
Therefore the desired equality \eqref{eq:tc} holds for all $u \in \Z$.
\end{proof}

\subsection{Twisted Auslander-Reiten quivers} 
\label{ssec:tAR}

Keep the setting from the last subsection.

\begin{Def}
Let $\cQ = (\Delta, \sigma, \xi)$ be a Q-datum for $\fg$.
\emph{The twisted Auslander-Reiten $($AR$)$ quiver of $\cQ$} is the full subquiver 
$\Gamma_{\cQ} \subset \hDs$ whose vertex set $(\Gamma_{\cQ})_{0}$ is given by
$$ 
(\Gamma_{\cQ})_{0} \seq \phi_{\cQ}^{-1}(\sR^{+} \times \{ 0 \}).
$$
\end{Def}

\begin{Prop}[{\cite[Section~3.7]{FO21}}] \label{Prop:tAR}
Let $\cQ = (\Delta, \sigma, \xi)$ be a Q-datum.
\begin{enumerate}
\item \label{tAR:conv} We have $(\Gamma_{\cQ})_{0} = \{ (\im, p) \in \hDs_{0} \mid \xi_{\im^{*}} - rh^{\vee} < p \le \xi_{\im} \}$.
\item For $(\im, p) \in \hDs_{0}$ with $\phi_{\cQ}(\im, p) = (\alpha, k)$, we have
$\phi_{\cQ}(\im^{*}, p \pm rh^{\vee}) = (\alpha, k \mp 1)$.
\end{enumerate}
\end{Prop}

The twisted Auslander-Reiten quiver $\Gamma_{\cQ}$ has a remarkable combinatorial meaning
in the theory of commutation classes in the Weyl group $\sW$, which will be explained later in subsection~\ref{ssec:cc}.

Now we put 
\begin{equation} \label{eq:hIQ}
\hI_{\cQ} \seq f\left((\Gamma_{\cQ})_{0}\right) = \{ (\bar{\im}, p) \in \hI \mid (\im, p) \in (\Gamma_{\cQ})_{0} \}
\end{equation}
The following is an immediate consequence of Proposition~\ref{Prop:tAR}.

\begin{Cor} \label{Cor:fAR} 
Let $\cQ=(\Delta, \sigma, \xi)$ be a Q-datum for $\fg$.
\begin{enumerate}
\item If the condition~{\rm (\ref{cond:Qtau})} is satisfied, 
we have 
$$\hI_{\cQ} = \{ (i, p) \in \hI \mid \xi_{(i^{*})^{\circ}} - rh^{\vee} < p \le \xi_{i^{\circ}} \}.$$
\item \label{fAR:shift}
For $(i, p) \in \hI$ with $\bphi_{\cQ}(i, p) = (\alpha, k)$, we have
$\bphi_{\cQ}(i^{*}, p \pm rh^{\vee}) = (\alpha, k \mp 1)$.
\end{enumerate}
\end{Cor}

The following observation will be useful in the sequel. 

\begin{Lem} \label{Lem:si-so}
Given an integer $b \in \Z$, there is a Q-datum $\cQ$ for $\fg$ 
such that
$$
\hI_{\cQ} = \{ (i,p) \in \hI \mid b-rh^{\vee}< p \le b \}.
$$ 
\end{Lem}
\begin{proof}
In view of Corollary~\ref{Cor:fAR}, it suffices to show that there is 
a Q-datum $\cQ=(\Delta, \sigma, \xi)$ for $\fg$ satisfying the condition~(\ref{cond:Qtau})
and that $\xi_{i^{\circ}} \in \{ b, b-1 \}$ for all $i \in I$. 
We can choose $\xi$ explicitly case-by-case as follows. Here 
$\epsilon$ denotes a fixed parity function for $\fg$ and
we will use the labeling in Figure~\ref{Fig:unf} for 
non-simply-laced $\fg$.
\begin{itemize}
\item 
For $\fg$ simply-laced: for each $i \in \Delta_{0} = I$,
we set 
$$\xi_{i} \seq \begin{cases}
b-\epsilon_{i} & \text{if $b \in 2\Z$,} \\
b-1+\epsilon_{i} & \text{if $b \in 2\Z + 1$}.
\end{cases}
$$
This $\xi$ corresponds to a sink-source orientation of $\Delta$. 
\item For $\fg$ of type $\mathrm{B}_{n}$: if $b \in 2\Z$, we set
$$
\xi_{\im} \seq \begin{cases}
b -\epsilon_{\bar{\im}} & \text{if $\im - n \in (2\Z_{\le 0} -1) \sqcup (2\Z_{\ge 0})$}, \\
b-2 - \epsilon_{\bar{\im}} & \text{if $\im-n \in (2\Z_{< 0}) \sqcup (2\Z_{\ge 0}+1)$}.
\end{cases}
$$
If $b \in 2\Z +1$, we set
$$
\xi_{\im} \seq \begin{cases}
b -1+\epsilon_{\bar{\im}} & \text{if $\im - n \in (2\Z_{\le 0} -1) \sqcup (2\Z_{\ge 0})$}, \\
b-3 +\epsilon_{\bar{\im}} & \text{if $\im-n \in (2\Z_{< 0}) \sqcup (2\Z_{\ge 0}+1)$}.
\end{cases}
$$
\item For $\fg$ of type $\mathrm{C}_{n}$: if $b \in 2\Z$, we set
$$
\xi_{\im} \seq \begin{cases}
b -\epsilon_{\bar{\im}}& \text{if $1 \le \im \le n$}, \\
b-2 -\epsilon_{\bar{\im}} & \text{if $\im=n+1$}. 
\end{cases}
$$
If $b \in 2\Z +1$, we set
$$
\xi_{\im} \seq \begin{cases}
b -1+\epsilon_{\bar{\im}}& \text{if $1 \le \im \le n$}, \\
b-3 +\epsilon_{\bar{\im}} & \text{if $\im=n+1$}. 
\end{cases}
$$
\item For $\fg$ of type $\mathrm{F}_{4}$: we set 
$$(\xi_{1}, \ldots, \xi_{6}) \seq
\begin{cases}
 (b, b, b-2, b-1, b, b-2) & \text{if $b \equiv \epsilon_{1} \pmod 2$,} \\
(b-1, b-1, b-3, b, b-1, b-3) & \text{if $b \not \equiv \epsilon_{1} \pmod 2$.}
\end{cases}
$$ 
\item For $\fg$ of type $\mathrm{G}_{2}$: we set
\[(\xi_{1}, \ldots, \xi_{4}) \seq \begin{cases}
(b, b-1, b-2, b-4) & \text{if $b \equiv \epsilon_{1} \pmod 2$,} \\
(b-1, b, b-3, b-5) & \text{if $b \not \equiv \epsilon_{1} \pmod 2$}.
\end{cases} \qedhere
\]
\end{itemize}
\end{proof}

\begin{Ex}\label{Ex:tAR}
We exhibit explicit examples of the twisted Auslander-Reiten quivers in types $\mathrm{A}_5$ and $\mathrm{B}_3$ associated with some Q-data satisfying the condition in Lemma~\ref{Lem:si-so}.
Here the labeling of $\Delta_0 = \{1,2,3,4,5 \}$ is as in Figure~\ref{Fig:unf}.
Via the map $\bphi_{\cQ}$, the vertices of $\Gamma_{\cQ}$ are labeled by positive roots of type $\mathrm{A}_5$, for which we use the abbreviation $[k,l] \seq \alpha_k + \alpha_{k+1} + \cdots + \alpha_l$ and $[k] \seq \alpha_k$.     
\begin{enumerate}
\item \label{Ex:tAR:A} 
Let $\cQ=(\Delta, \id, \xi)$ be the Q-datum of type $\mathrm{A}_5$ given by $(\xi_1, \ldots, \xi_5) = (0,-1,0,-1,0)$. 
Then the quiver $\Gamma_{\cQ}$ is depicted as:
$$
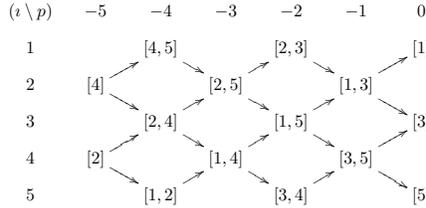
\raisebox{3mm}{
\scalebox{0.60}{\xymatrix@!C=6mm@R=2mm{
(\im\setminus p)  &-5&-4 &-3& -2 &-1& 0   \\
1&&[4,5] \ar@{->}[dr] && [2,3] \ar@{->}[dr] && [1] \\
2& [4] \ar@{->}[dr]\ar@{->}[ur] && [2,5] \ar@{->}[dr]\ar@{->}[ur] && [1,3] \ar@{->}[dr]\ar@{->}[ur]& \\ 
3&& [2,4] \ar@{->}[dr] \ar@{->}[ur] && [1,5] \ar@{->}[dr] \ar@{->}[ur] && [3] \\
4& [2] \ar@{->}[ur]\ar@{->}[dr]&& [1,4] \ar@{->}[ur]\ar@{->}[dr] && [3,5] \ar@{->}[ur]\ar@{->}[dr]& \\
5 && [1,2] \ar@{->}[ur] && [3,4] \ar@{->}[ur] && [5] }}}
$$
\item \label{Ex:tAR:B}
Let $\cQ=(\Delta, \vee, \xi)$ be the Q-datum of type $\mathrm{B}_3$ given by $(\xi_1, \ldots, \xi_5) = (-2,0,-1,-2,0)$. 
Then the quiver $\Gamma_{\cQ}$ is depicted as:
$$
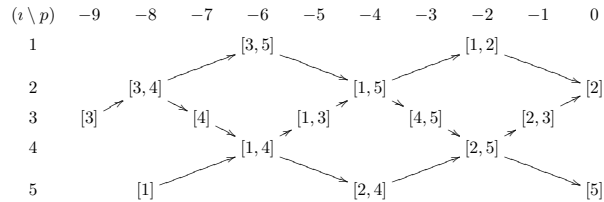
\raisebox{3mm}{
\scalebox{0.60}{\xymatrix@!C=4mm@R=0.5mm{
(\im\setminus p)  &-9&-8 &-7& -6 &-5& -4 & -3& -2 & -1& 0\\
1&&&& [3,5] \ar@{->}[ddrr]&&&&  [1,2] \ar@{->}[ddrr] &&\\ \\
2&&[3,4]\ar@{->}[dr] \ar@{->}[uurr]&&&& [1,5] \ar@{->}[dr]\ar@{->}[uurr] &&&&  [2]  \\
3& [3] \ar@{->}[ur] &&[4] \ar@{->}[dr] && [1,3] \ar@{->}[ur] && [4,5] \ar@{->}[dr] && [2,3] \ar@{->}[ur] & \\
4&&&&[1,4]\ar@{->}[ddrr]\ar@{->}[ur] &&&&  [2,5] \ar@{->}[ddrr]\ar@{->}[ur] && \\ \\
5&&[1] \ar@{->}[uurr] &&&& [2,4] \ar@{->}[uurr]&&&& [5] }}}
$$
\end{enumerate}
\end{Ex}

\section{Monoidal subcategories associated with Q-data}
\label{sec:C_Q}

In this section, we see some applications of the theory of Q-data to 
the representation theory of quantum loop algebras.
Although most of the topics have already appeared in \cite{FO21},
we shall give some detailed explanations for completeness.  
In Section~\ref{ssec:Qwt}, we briefly review the notion of $\cQ$-weights.
In Section~\ref{ssec:CX}, we define a subcategory $\Cc_{X} \subset \Cc_{\Z}$
supported (in a sense) on an arbitrary subset $X \subset \hI$ and give a sufficient condition for it to become a monoidal subcategory. 
We apply this result to the case $X = \hI_{\cQ}$ to define the subcategory $\Cc_{\cQ}$ in Section~\ref{ssec:CQ}. 
It generalizes the categories $\Cc_{Q}$ introduced in \cite{HL15} for simply-laced $\fg$ associated with a Dynkin quiver $Q$.
In Section~\ref{ssec:tr}, we introduce the truncation of $q$-characters and that of $(q,t)$-characters
with respect to a general height function $\xi$.   


\subsection{$\cQ$-weights}
\label{ssec:Qwt}

First, we recall the notion of $\cQ$-weights from \cite[Section~5]{FO21}.

\begin{Def} \label{Def:Qwt}
Let $\cQ$ be a Q-datum for $\fg$. For a monomial 
$m$ in $\cY$, we define its \emph{$\cQ$-weight} $\wt_{\cQ}(m)$ 
as the element of $\sQ$ given by
$$
\wt_{\cQ}(m) \seq \sum_{(i,p) \in \hI} u_{i,p}(m) \pi(\bphi_{\cQ}(i,p)).
$$
For a monomial $\tm \in \cY_{t}$, we set $\wt_{\cQ}(\tm) \seq \wt_{\cQ}(\evt(\tm))$.
\end{Def}

By the $\cQ$-weight, the quantum torus $\cY_{t}$ is equipped with a $\sQ$-grading. 

\begin{Prop}[cf.~{\cite[Section 5]{FO21}}] \label{Prop:Qwt}
For any $(i,p) \in \hI$, we have
$\wt_{\cQ}(A_{i,p+d_{i}}) = 0.$
In particular, all the elements 
$F_{t}(m), E_{t}(m)$ and $L_{t}(m)$ for $m \in \cM$ are homogeneous with
respect to the $\cQ$-weight so that
$$
\wt_{\cQ}(F_{t}(m)) = \wt_{\cQ}(E_{t}(m)) = \wt_{\cQ}(L_{t}(m)) = \wt_{\cQ}(m).
$$ 
Hence the quantum Grothendieck ring $\cK_{t}$ inherits the $\sQ$-grading.
\end{Prop}

The next fact is a consequence of Corollary~\ref{Cor:fAR}~(\ref{fAR:shift}). 
\begin{Lem} \label{Lem:wtD}
Let $\cQ$ be a Q-datum for $\fg$.
For any $k \in \Z$ and a homogeneous element $y \in \cY$ $($resp.~$\cY_{t})$ with respect to $\wt_\cQ$, we have
$$
\wt_\cQ( \fD^k(y)) = (-1)^k \wt_{\cQ}(y), \quad (\text{resp.~$\wt_\cQ( \fD_t^k(y)) = (-1)^k \wt_{\cQ}(y)$}). 
$$
\end{Lem}

The following assertion relates the $\cQ$-weights with the integers $\Nn(i,p;j,s)$ (see (\ref{eq:defNn})) in the  defining  relation of the quantum torus $\cY_t$.
It is an immediate consequence of Lemma~\ref{Lem:bphi} and Theorem~\ref{Thm:tc}. 

\begin{Prop}[{\cite[Proposition 5.21]{FO21}}] \label{Prop:Nnwt}
Let $\cQ$ be a Q-datum for $\fg$.
For any two distinct elements $(i,p), (j,s) \in \hI$ with $p \le s$, we have
$$
\Nn(i,p;j,s)= (\wt_{\cQ}(Y_{i,p}), \wt_{\cQ}(Y_{j,s})).
$$
\end{Prop}


\subsection{Subcategories supported on twisted-convex sets}
\label{ssec:CX}

Let $X \subset \hI$ be a subset. Consider the $\Z$-subalgebra $\cY_X \subset \cY$ 
(resp.~$\Z[t^{\pm 2/1}]$-subalgebra $\cY_{t, X} \subset \cY_t$) 
generated by the variables $Y_{i,p}^{\pm 1}$ (resp.~$\tY_{i,p}^{\pm 1}$) with $(i,p) \in X$, and
denote by $\cM_X$ the set of dominant monomials in $\cY_X$.  
Then we define the category $\Cc_{X}$  as the Serre subcategory of $\Cc_{\Z}$
satisfying $\Irr \Cc_{X} = \{ [L(m)] \mid m \in \cM_X \}$ and the $\Z[t^{\pm 1/2}]$-module  $\cK_{t, X}$ as the $\Z[t^{\pm 1/2}]$-submodule of $\cK_t$ given by
$$
\cK_{t, X} \seq \bigoplus_{m \in \cM_X} \Z[t^{\pm 1/2}] L_t(m).
$$ 
In what follows, we shall give a sufficient condition for the subcategory $\Cc_X$ to be monoidal (i.e.~closed under taking tensor products)
and at the same time for the $\Z[t^{\pm 1/2}]$-submodule $\cK_{t, X}$ to be a $\Z[t^{\pm 1/2}]$-subalgebra.   
Our discussion here is a straightforward generalization of \cite[Lemmas 3.26 and 7.9]{HO19}.

\begin{Def} \label{Def:conv}
Let $(\Delta, \sigma)$ be the unfolding of $\fg$.
\begin{enumerate}
\item A subset $\widetilde{X} \subset \hDs_0$ is said to be \emph{convex} if it satisfies the following condition:
for any oriented path $(x_1 \to x_2 \to \cdots \to x_l)$ in the repetition quiver $\hDs$,
we have $\{ x_{1}, x_{2}, \ldots , x_{l} \} \subset  \widetilde{X}$ if and only if $\{x_{1}, x_{l}\} \subset \widetilde{X}$.   
\item We say that a subset $X \subset \hI$ is \emph{twisted-convex} if $f^{-1}(X) \subset \hDs_0$ is convex.
\end{enumerate}
\end{Def}

\begin{Lem} \label{Lem:ideal}
If $X \subset \hI$ is twisted-convex, the set $\cM_X$ is an ideal 
of the partially ordered set $(\cM, \le)$, i.e.,~it is closed under taking  smaller elements in $\cM$ with respect to the 
Nakajima ordering $\le$.
\end{Lem}
\begin{proof}
Suppose $m \in \cM_{X}$ and $mM \in \cM$ with $M$ being a monomial in the variables $A^{-1}_{i,p}$'s. 
Let $A^{-1}_{i,p}$ be a factor of $M$ and let $(\im, p \pm d_i) \seq f^{-1}(i, p \pm d_i)$.
Since the product $mM$ is dominant, 
we have either (a) $Y_{i,p+d_i}$ is a factor of $m$, or (b) there is a factor $A_{i',p'}^{-1}$ of $M$ which contains $Y_{i,p+d_i}$ as a factor. 
In the case (a), we have $(\im, p+d_i) \in f^{-1}(X)$ by the assumption $m \in \cM_{X}$.
In the case (b), we find that there is an oriented path in the repetition quiver $\hDs$ from $(\im,p+d_i)$ to $(\im',p'+d_{i'}) \seq f^{-1}(i',p'+d_{i'})$. Then, we repeat the same argument replacing $A^{-1}_{i,p}$ with $A^{-1}_{i',p'}$, to find either $(\im',p'+d_{i'}) \in f^{-1}(X)$, or there is yet another factor $A_{i'',p''}^{-1}$ of $M$ with an oriented path in $\hDs$ from $(\im',p'+d_{i'})$ to $f^{-1}(i'',p''+d_{i''})$.
Repeating it finitely many times, we finally find an oriented path (possibly of length zero) in $\hDs$ from $(\im,p+d_i)$ to a vertex in $f^{-1}(X)$.
In the similar way, we find an oriented path (possibly of length zero) from a vertex in $f^{-1}(X)$ to $(\im, p-d_i)$. 
As a result, we obtain an oriented path in $\hDs$ whose end points belong to $f^{-1}(X)$ and which factors through both $(\im,p+d_i)$ and $(\im, p-d_i)$. 
Then, by the convexity of $f^{-1}(X)$, we have $A^{-1}_{i,p} \in \cY_{X}$.
Therefore, $M \in \cY_{X}$ and hence $mM \in \cY_{X} \cap \cM = \cM_{X}$ hold.
\end{proof}

\begin{Prop} \label{Prop:CX}
Assume that $X \subset \hI$ is twisted-convex.  Then
\begin{enumerate}
\item \label{CX:monoidal} the category $\Cc_X$ is a monoidal subcategory of $\Cc_\Z$, and
\item the $\Z[t^{\pm 1/2}]$-module $\cK_{t, X}$ is a $\Z[t^{\pm 1/2}]$-subalgebra of $\cK_{t}$ $($hence we call it the quantum Grothendieck ring of $\Cc_X)$. 
\end{enumerate} 
Moreover, under the same assumption, we have 
$$
\cK_{t, X} = \bigoplus_{m \in \cM_X} \Z[t^{\pm 1/2}] E_t(m) = \bigoplus_{m \in \cM_X} \Z[t^{\pm 1/2}] F_t(m)
$$
and hence $\evt(\cK_{t, X}) = \chi_q(K(\Cc_{X}))$.
\end{Prop}
\begin{proof}
Let $m_1, m_2 \in \cM_X$. Suppose that a simple module $L(m)$ with $m \in \cM$ appears as a composition factor of the tensor product module $L(m_1)\otimes L(m_2)$.
Then the dominant monomial $m$ occurs in $\chi_q(L(m_1)) \chi_q(L(m_2))$. It implies that $m \le m_1 m_2$ by Theorem~\ref{Thm:FM}.
From the twisted convexity of $X$, we have  $m \in \cM_X$ by Lemma~\ref{Lem:ideal}. This proves that $L(m_1) \otimes L(m_2) \in \Cc_{X}$
and hence the assertion (\ref{CX:monoidal}). The other assertions can be proved in a similar way
using Lemma~\ref{Lem:ideal} together with Theorem~\ref{Thm:Ft}, the property (\ref{eq:EtFt}) and (S2) in Theorem~\ref{Thm:qtch}.
\end{proof}


\subsection{The category $\Cc_{\cQ}$}
\label{ssec:CQ}

Let $\cQ = (\Delta, \sigma, \xi)$ be a Q-datum for $\fg$.
Recall the finite subset $\hI_{\cQ} = f((\Gamma_\cQ)_0) \subset \hI$ defined in Section~\ref{ssec:tAR}.
Since the set $(\Gamma_\cQ)_0$ is convex by Proposition~\ref{Prop:tAR}~(\ref{tAR:conv}),
the set $\hI_{\cQ}$ is twisted-convex in the sense of Definition~\ref{Def:conv}.  
Let us apply the construction in the previous subsection to the case $X=\hI_\cQ$.
We use the notation $(\cY_\cQ, \cY_{t, \cQ}, \cM_\cQ, \Cc_\cQ, \cK_{t, \cQ})$ for simplicity
to denote $(\cY_X, \cY_{t, X}, \cM_X, \Cc_X, \cK_{t, X})$ in this case.  
Thanks to Proposition~\ref{Prop:CX}, 
$\Cc_{\cQ} \subset \Cc_{\Z}$ is a monoidal subcategory and 
$\cK_{t, \cQ} \subset \cK_t$ is a $\Z[t^{\pm 1/2}]$-subalgebra.
For clarity, we shall state the definition of the category $\Cc_\cQ$ in a direct way.
\begin{Def}
We define the category $\Cc_{\cQ}$ as the Serre subcategory of the category $\Cc_{\Z}$
satisfying $\Irr \Cc_{\cQ} = \{ [L(m)] \mid m \in \cM_\cQ \}$, where 
$\cM_\cQ$ is the set of dominant monomials in the variables $Y_{i,p}$ with $(i,p) \in \hI_\cQ$.
\end{Def}

\begin{Ex}
If we consider the Q-datum $\cQ$ of type $\mathrm{A}_5$ in Example~\ref{Ex:tAR}~\eqref{Ex:tAR:A},
the corresponding category $\Cc_{\cQ}$ is monoidally generated by the fundamental modules associated with the variables
$\{ Y_{i,-2k} \mid i \in\{ 1,3,5\}, 0\le k \le 2\}\cup\{Y_{i,-2k-1} \mid i \in \{ 2,4\}, 0\le k \le 2\}$.
Similarly, if we consider the Q-datum $\cQ$ of type $\mathrm{B}_3$ in Example~\ref{Ex:tAR}~\eqref{Ex:tAR:B},
the corresponding category $\Cc_{\cQ}$ is monoidally generated by the fundamental modules associated with the variables
$\{ Y_{i,-2k} \mid i \in \{1,2\}, 0\le k \le 4 \} \cup \{Y_{3,-2k-1} \mid 0\le k \le 4 \}$.
\end{Ex}

For a positive root $\alpha \in \sR^{+}$, we set
$$
L^{\cQ}(\alpha) \seq L(Y_{i,p}), \quad L^\cQ_t(\alpha) \seq L_t(Y_{i,p}) \quad \text{where $(i,p) = \bphi_{\cQ}^{-1}(\alpha, 0)$}. 
$$
Note that we have
$\wt_{\cQ}(L_t^\cQ(\alpha)) = \alpha$ by Proposition~\ref{Prop:Qwt} and
$\evt(L^\cQ_t(\alpha)) = \chi_q(L^\cQ(\alpha))$ by Theorem~\ref{Thm:qtfund}~(\ref{qtfund:evt}).
The set $\{ L^{\cQ}(\alpha) \mid \alpha \in \sR^+\}$ gives
a complete and irredundant collection of fundamental modules in $\Cc_{\cQ}$ up to isomorphisms.

\begin{Lem} \label{Lem:CQ}
The followings hold.
\begin{enumerate}
\item \label{CQ:K} The $\Z$-algebra $K(\Cc_{\cQ})$ is isomorphic to the polynomial ring 
$\Z[x_{\alpha} \mid \alpha \in \sR^{+}]$
in $\ell_{0}$-many variables
via $[L^{\cQ}(\alpha)] \mapsto x_{\alpha}$.
\item \label{CQ:Kt}
The set $\{ L^{\cQ}_{t}(\alpha)\mid \alpha \in \sR^{+} \}$ generates the $\Z[t^{\pm 1/2}]$-algebra $\cK_{t, \cQ}$.
\item \label{CQ:fund}
For each $\alpha \in \sR^{+}$ and $k \in \Z$, we have
$$
\cD^{k}(L^{\cQ}(\alpha)) \cong L(Y_{i,p}), \quad \text{where $(i,p) = \bphi_{\cQ}^{-1}(\alpha, k)$}.
$$
Thus, the set $\{ \cD^{k}(L^{\cQ}(\alpha)) \mid (\alpha, k) \in \sR^{+} \times \Z\}$
gives a complete and irredundant collection of fundamental modules in the category $\Cc_{\Z}$ up to isomorphisms.
\end{enumerate}
\end{Lem}
\begin{proof}
(\ref{CQ:K}) and (\ref{CQ:Kt}) are immediate from Proposition~\ref{Prop:KCZ} and Proposition~\ref{Prop:CX} respectively.
(\ref{CQ:fund}) follows from Theorem~\ref{Thm:fD} and Corollary~\ref{Cor:fAR}~(\ref{fAR:shift}).
\end{proof}
\subsection{Truncation}
\label{ssec:tr}

Let $\cQ = (\Delta, \sigma, \xi)$ be a Q-datum for $\fg$.
In this subsection, we apply the construction given in Section~\ref{ssec:CX} above 
to the case $X = \hI_{\le \xi} \subset \hI$, where
$$
\hI_{\le \xi} \seq f\left(\{ (\im, p) \in \hDs_{0} \mid p \le \xi_{\im}  \}\right).
$$
We use the notation $(\cY_{\le \xi}, \cY_{t, \le \xi}, \cM_{\le \xi}, \Cc_{\le \xi}, \cK_{t, \le \xi})$
to denote $(\cY_X, \cY_{t, X}, \cM_X, \Cc_X, \cK_{t, X})$ in this case.  
Since the set $\hI_{\le \xi}$ is obviously twisted-convex, 
$\Cc_{\le \xi} \subset \Cc_{\Z}$ is a monoidal subcategory and 
$\cK_{t, \le \xi} \subset \cK_t$ is a $\Z[t^{\pm 1/2}]$-subalgebra by Proposition~\ref{Prop:CX}.

For a (non-commutative) Laurent polynomial $y \in \cY_{t}$, we denote by $y_{\le \xi}$
the element of $\cY_{t, \le \xi}$ obtained from $y$ by discarding all the monomials containing  
the factors $\tY_{\bar{\im},p}^{\pm 1}$ with $(i,p) \in \hI \setminus \hI_{\le \xi}$.
This assignment $y \mapsto y_{\le \xi}$ defines 
a $\Z[t^{\pm 1/2}]$-linear map $(\cdot)_{\le \xi} \colon \cY_{t} \to \cY_{t, \le \xi}$,
which is not an algebra homomorphism. 
By an abuse of notation, we denote the similar $\Z$-linear homomorphism defined for $\cY$ 
by the same symbol $(\cdot)_{\le \xi} \colon \cY \to \cY_{\le \xi}$. Then we have
$\evt \circ (\cdot)_{\le \xi} = (\cdot)_{\le \xi} \circ \evt$.
\begin{Prop} \label{Prop:tr}
The above linear maps $(\cdot)_{\le \xi}$ restrict to the injective homomorphisms of algebras
$$
\text{
$\cK_{t, \le \xi} \hookrightarrow \cY_{t, \le \xi}$ \quad and \quad
$\chi_{q}\left(K(\Cc_{\le \xi})\right) \hookrightarrow \cY_{\le \xi}$.} 
$$ 
\end{Prop}
\begin{proof}
A proof can be similar to the discussions in \cite[Section~6.2]{HL10} and \cite[Proposition 3.10]{HL16},
which treated special cases.   
In principle, the assertions follow from Theorem~\ref{Thm:Ft} and \cite[Lemma 6.17]{FM01} respectively, with the help of Lemma~\ref{Lem:ideal}.
\end{proof}
We refer to the elements $\chi_q(V)_{\le \xi}$ with $V \in \Cc_{\le \xi}$ as \emph{the truncated $q$-characters},
and refer to the elements $L_t(m)_{\le \xi}$ and $E_t(m)_{\le \xi}$ with $m \in \cM_{\le \xi}$ as \emph{the truncated $(q,t)$-characters}.

\begin{Cor} \label{Cor:tr}
The above linear maps $(\cdot)_{\le \xi}$ restrict to the injective homomorphisms of algebras
$$
\text{
$\cK_{t, \cQ} \hookrightarrow \cY_{t, \cQ}$ \quad and \quad
$\chi_{q}\left(K(\Cc_{\cQ})\right) \hookrightarrow \cY_{\cQ}$.} 
$$ 
\end{Cor}
\begin{proof}
We have to prove $(\cK_{t, \cQ})_{\le \xi} \subset \cY_{t, \cQ}$
and $\chi_{q}(K(\Cc_{\cQ}))_{\le \xi} \subset \cY_{\cQ}$.
In view of Lemma~\ref{Lem:CQ} (\ref{CQ:K}) (resp.~(\ref{CQ:Kt})), 
it is enough to show that the $q$-characters (resp.~the $(q, t)$-characters) of fundamental modules 
do not involve the variables $Y_{i,p}$ (resp.~$\tY_{i,p}$) with $(i,p) \in \hI_{\le \xi^\prime}$,
where $\xi^\prime$ is another height function on $(\Delta, \sigma)$ 
given by $\xi^\prime_\im \seq \xi_{\im^*} - rh^\vee$ for $\im \in \Delta_0$.
For $q$-characters, we can check this property from the twisted-convexity of the set $\hI \setminus \hI_{\le \xi^\prime}$ and 
the fact that the $q$-characters of fundamental modules 
are computed by Frenkel-Mukhin's algorithm, see \cite[Theorem 5.9]{FM01}. 
For $(q,t)$-characters, it follows from the case of $q$-characters via Theorem~\ref{Thm:qtfund}. 
\end{proof}

\begin{Rem} \label{Rem:trS}
By the injectivity of $(\cdot)_{\le \xi}$, its compatibility with $\ol{(\cdot)}$
and Theorem~\ref{Thm:qtch},
the truncation of a $(q,t)$-character $L_{t}(m)_{\le \xi}$ for each $m \in \cM_\cQ$ can be characterized in the image of truncation
$(\cK_{t, \cQ})_{\le \xi} \subset \cY_t$ 
by the following properties:
\begin{itemize}
\item[(tr-S1)] $\ol{L_{t}(m)_{\le \xi}} = L_{t}(m)_{\le \xi}$, and
\item[(tr-S2)] $L_{t}(m)_{\le \xi} - E_{t}(m)_{\le \xi} \in \sum_{m^{\prime} \in \cM_{\cQ}} t \Z[t] E_{t}(m^{\prime})_{\le \xi}$.
\end{itemize}
\end{Rem}

When the height function $\xi$ satisfies the condition in Lemma~\ref{Lem:si-so} for a given integer $b \in \Z$, 
we obtain
$$
\hI_{\le \xi} = \hI_{\le b} \seq \hI \cap (I \times \Z_{\le b}).
$$
Thus we will denote the corresponding truncation by $(\cdot)_{\le b}$ instead of $(\cdot)_{\le \xi}$ for this case.
This notation is compatible with the notation in \cite{HL10} where truncated $q$-characters for certain categories $\mathcal{C}_\ell$ of simply-laced quantum loop algebras are introduced.


\section{Quantum $T$-systems}
\label{sec:qTsys}

In this section, we review the classical $T$-system identities satisfied by the $q$-characters of Kirillov-Reshetikhin modules 
and we establish a quantum analog of the $T$-system identities for all simple Lie algebra $\fg$
in a unified way. As far as the authors know, these relations are new for non simply-laced types (except in type $\mathrm{B}$). 
We also prove some commutation relations among the Kirillov-Reshetikhin modules involved in the $T$-systems and some $t$-commutation relations among related $(q,t)$-characters, which will be useful for us in the following.

\subsection{Kirillov-Reshetikhin modules}
\label{ssec:KR}

For each $(i,p) \in \hI$ and $k \in \Z_{\ge 0}$, set
$$
W^{(i)}_{k, p} \seq L(m^{(i)}_{k, p}), \qquad m^{(i)}_{k,p} \seq \prod_{s=1}^{k}Y_{i, p + 2(s-1)d_{i}}. 
$$
These simple modules are called \emph{Kirillov-Reshetikhin modules} (or \emph{KR modules} for short).
Note that $W^{(i)}_{0,p}$ is the trivial $U_{q}(L\fg)$-module
and $W^{(i)}_{1,p} = L(Y_{i,p})$ is a fundamental module.

\begin{Thm}[cf.~{\cite[Theorem 3.2]{Nakajima03II}, \cite[Theorem 4.1, Lemma 4.4]{Hernandez06}}] \label{Thm:KR}
For $(i,p) \in \hI$ and $k \in \Z_{\ge 0}$, the $q$-character of the corresponding KR module $W_{k,p}^{(i)}$ is of the form
\begin{equation}\label{eq:form}
\chi_{q}(W_{k,p}^{(i)}) = m_{k,p}^{(i)}(1 + A_{i,p + (2k - 1)d_{i}}^{-1} \chi)
\end{equation}
where $\chi$ is a polynomial in the variables $A_{j,s+d_{j}}^{-1}$, $(j,s) \in \hI$.
\end{Thm}

In this paper, it will be useful to introduce yet another notation for KR modules. 
For two integers $a,b \in \Z$, we set 
\begin{align*} 
[a,b] &\seq \{ x \in \Z \mid a \le x \le b \}, & (a,b] &\seq \{ x \in \Z \mid a < x \le b \}, \\
[a,b) &\seq \{ x \in \Z \mid a \le x < b \}, & (a,b) &\seq \{ x \in \Z \mid a < x < b \}.
\end{align*}
We refer to subsets of these forms as \emph{intervals}. 
Let $(\Delta, \sigma)$ be the unfolding of $\fg$ and 
$\hDs$ its repetitive quiver (see Section~\ref{sec:Qdata}).
For an interval $[a,b] \subset \Z$ and a vertex $\im \in \Delta_{0}$, we define
a dominant monomial $m^{(\im)}[a,b] \in \cM$ by
$$
m^{(\im)}[a,b] \seq \prod_{(\im, p) \in \hDs_{0}, p \in [a,b]}Y_{\bar{\im}, p}. 
$$  
Note that the corresponding simple module $W^{(\im)}[a,b] \seq L(m^{(\im)}[a,b])$ 
is a KR module. 
Similarly, we define the dominant monomials $m^{(\im)}(a,b], m^{(\im)}[a,b), m^{(\im)}(a,b)$
and denote the corresponding KR modules by 
$W^{(\im)}(a,b], W^{(\im)}[a,b), W^{(\im)}(a,b)$ respectively.

In this notation, Theorem~\ref{Thm:KR} implies the following.
\begin{Lem} \label{Lem:KR}
For any $\im \in \Delta_0$ and integers $a \le b$, we have
 \begin{align*} 
\chi_q\left(W^{(\im)}[a,b]\right)_ {\le b}&= m^{(\im)}[a,b], & \chi_q\left(W^{(\im)}(a,b]\right)_ {\le b}&= m^{(\im)}(a,b], \allowdisplaybreaks\\
\chi_q\left(W^{(\im)}[a,b)\right)_ {\le b-1}&= m^{(\im)}[a,b), & \chi_q\left(W^{(\im)}(a,b)\right)_ {\le b-1}&= m^{(\im)}(a,b),
\end{align*}
where $(\cdot)_{\le b}$ denotes the truncation defined in {\rm Section~\ref{ssec:tr}}. 
\end{Lem}
\subsection{Classical $T$-systems}
\label{ssec:clTsys}

The $q$-characters of KR modules satisfy a system of identities called \emph{the $T$-systems}. 
\begin{Thm}[{\cite[Theorem 1.1]{Nakajima03II}}, {\cite[Theorem 3.4]{Hernandez06}}] \label{Thm:Tsys}
For each $(\im, p), (\im, s) \in \hDs_{0}$ with
$p<s$, we have the following equality in $K(\Cc_{\Z})$:
\begin{equation} \label{eq:Tsys}
\left[W^{(\im)}[p,s) \right] \left[W^{(\im)}(p,s]\right] = \left[W^{(\im)}[p,s]\right] \left[W^{(\im)}(p,s)\right] + \prod_{\jm \in \Delta_0; \jm \sim \im} \left[W^{(\jm)}(p,s)\right],
\end{equation}
where the two terms on right hand side are simple tensor products of Kirillov-Reshetikhin modules.
\end{Thm}

 The second term in the right hand side, there is a change of notations in comparison 
to \cite{Nakajima03II, Hernandez06} as the indices for the product are not picked from $I$ but from $\Delta_0$. This allows to have a uniform formulation of 
the $T$-systems, including non simply-laced types.

Let us consider the following two sets of dominant monomials:
\begin{align}
\label{eq:cM+}
\cM^{+}(\im; p, s) &\seq \left\{ m^{(\im)}[p,s], m^{(\im)}(p,s), m^{(\im)}(p,s] \right\}
\sqcup \left\{ m^{(\jm)}(p,s) \;\middle|\; \jm \sim \im \right\}, \allowdisplaybreaks\\
\cM^{-}(\im; p, s) &\seq \left(\cM^{+}(\im; p, s) \setminus\left\{ m^{(\im)}(p,s] \right\}\right) \sqcup \left\{ m^{(\im)}[p,s) \right\}.
\end{align}
Note that the KR modules corresponding to the dominant monomials in $\cM^{-}(\im; p, s) \cup \cM^{+}(\im; p, s)$ 
are exactly the ones involved in the $T$-system identity~(\ref{eq:Tsys}).

\begin{Prop} \label{Prop:clTsys}
Let $(\im, p), (\im, s) \in \hDs_{0}$ with $p<s$ as in {\rm Theorem~\ref{Thm:Tsys}}.
For $\varepsilon \in \{ +, - \}$ and any $m,m^\prime \in \cM^{\varepsilon}(\im; p, s)$, 
the tensor product $L(m) \otimes L(m^\prime)$ is simple. 
In particular, the set $\{ L(m) \mid m \in \cM^{\varepsilon}(\im; p, s)\}$  forms a mutually commuting family of KR modules.  
\end{Prop}
\begin{proof}
We only consider the case $\varepsilon = +$, as a proof of the other case $\varepsilon =-$ is similar.
Our basic tool is the truncation defined in Section~\ref{ssec:tr}. 
We divide the situation into the following four cases: 
\begin{align} 
\{ m,m^\prime\} &\subset \left\{ m^{(\im)}[p,s], m^{(\im)}(p,s] \right\}\sqcup \left\{ m^{(\jm)}(p,s) \;\middle|\; \jm \sim \im \right\}, \label{Tsys:case1} \allowdisplaybreaks\\
\{ m,m^\prime\} &\subset \left\{ m^{(\im)}(p,s)\right\} \sqcup \left\{ m^{(\jm)}(p,s) \;\middle|\; \jm \sim \im \right\}, \label{Tsys:case2} \allowdisplaybreaks\\
\{ m,m^\prime\}  &= \left\{m^{(\im)}(p,s), m^{(\im)}(p,s] \right\}, \label{Tsys:case3} \allowdisplaybreaks\\
\{ m,m^\prime\}  &= \left\{m^{(\im)}[p,s], m^{(\im)}(p,s) \right\}. \label{Tsys:case4} 
\end{align}

For the case~(\ref{Tsys:case1}), 
noting $m^{(\jm)}(p,s) = m^{(\jm)}(p,s]$ for $\jm \sim \im$,
we have $\chi_{q}(L(m))_{\le s} = m$ and $\chi_{q}(L(m^\prime))_{\le s} = m^\prime$ by Lemma~\ref{Lem:KR}.
Therefore we obtain
$$\chi_{q}(L(m) \otimes L(m^\prime))_{\le s} = \chi_{q}(L(m))_{\le s} \cdot \chi_{q}(L(m^\prime))_{\le s} = mm^\prime.$$
By Proposition~\ref{Prop:tr}, this shows that $L(m) \otimes L(m^\prime) \cong L(mm^\prime)$ as desired.
For the case~(\ref{Tsys:case2}), we apply the same argument with the truncation $(\cdot)_{\le s}$ replaced by $(\cdot)_{\le s-1}$.  

To deal with the case~(\ref{Tsys:case3}), we recall the involution $\omega$ defined in the last paragraph of Section~\ref{ssec:duality}.
In our case, it yields
$$
\omega^* W^{(\im)}(p,s) \cong W^{(\im^*)}(-s-rh^\vee, -p-rh^\vee), \quad 
\omega^* W^{(\im)}(p,s] \cong W^{(\im^*)}[-s-rh^\vee, -p-rh^\vee).
$$
Thus we can apply the same argument as in the case~(\ref{Tsys:case2}) to obtain the conclusion. 

The last case~(\ref{Tsys:case4}) was treated in \cite[Theorem 6.1 (2)]{Hernandez06}. 
\end{proof}
\subsection{Quantum $T$-systems} 
\label{ssec:qTsys}
For a general dominant monomial $m \in \cM$, 
the element $F_t({m})$ is given by an algorithm \cite[Section 5.7.1]{Hernandez04}.
As a result, we obtain the following $t$-analog of Theorem~\ref{Thm:KR}.

\begin{Thm} \label{Thm:Fform}
For $(i,p) \in \hI$ and $k \in \Z_{\ge 0}$, the element $F_{t}({m_{k,p}^{(i)}})$ is of the form
\begin{equation}\label{eq:Fform}
F_{t}(m_{k,p}^{(i)}) = \ul{m_{k,p}^{(i)}}(1 + \tA_{i,p + (2k - 1)d_{i}}^{-1} \chi)
\end{equation}
where $\chi$ is a $($non-commutative$)$ polynomial in the $\tA_{j,s+d_{j}}^{-1}$, $(j,s) \in \hI$.
\end{Thm}

Note that $F_{t}(m_{k,p}^{(i)})$ is $\ol{(\cdot)}$-invariant 
as a unique dominant monomials occurs, and this monomial
is $\ol{(\cdot)}$-invariant.
Although it is known that at $t = 1$, the $q$-character of 
$L(m_{k,p}^{(i)})$ has a unique dominant monomial, it is not clear a priori that 
$L_{t}(m_{k,p}^{(i)}) = F_{t}(m_{k,p}^{(i)})$. But we conjecture that they are equal.

\begin{Conj} \label{Conj:L=F}
For any $(i,p) \in \hI$ and $k \in \Z_{\ge 0}$, 
we have $$L_{t}(m_{k,p}^{(i)}) = F_{t}(m_{k,p}^{(i)}).$$
\end{Conj}

It is known for $\mathrm{ADE}$ types by the results of Nakajima. In general it is clear for $k = 1$ by definition of $L_t(m_{1,p}^{(i)})$. 
Also one can check it for $k=2$ by a direct computation. 
We will establish this result in several other situation below.

In this subsection, we establish \emph{the quantum $T$-systems} in $\cK_t$ satisfied by the $F_{t}(m_{k,p}^{(i)})$.
For an interval $[a, b] \subset \Z$ and $\im \in \Delta_{0}$, we set 
$F_{t}^{(\im)}[a,b] \seq F_{t}(m^{(\im)}[a,b])$. We define $F_{t}^{(\im)}(a,b), F_{t}^{(\im)}(a,b]$ and $F^{(\im)}_{t}[a,b)$
in the similar way. With this notation, the next assertion follows immediately from Theorem~\ref{Thm:Fform}.

\begin{Lem} \label{Lem:qKR}
For any $\im \in \Delta_0$ and integers $a \le b$, we have
 \begin{align*} 
F_{t}^{(\im)}[a,b]_{\le b}&= \ul{m^{(\im)}[a,b]}, & F_{t}^{(\im)}(a,b]_{\le b}&= \ul{m^{(\im)}(a,b]}, \\
F_{t}^{(\im)}[a,b)_{\le b-1}&= \ul{m^{(\im)}[a,b)}, & F_{t}^{(\im)}(a,b)_{\le b-1}&=\ul{m^{(\im)}(a,b)}.
\end{align*}
\end{Lem}

\begin{Thm}[Quantum $T$-systems] \label{Thm:qTsys}
For $(\im, p), (\im, s) \in \hDs_{0}$ with $p < s$, 
there holds the following equality in $\cK_t$:
\begin{equation} \label{eq:qTsys}
F_{t}^{(\im)}[p,s) F_{t}^{(\im)}(p,s] = t^{x} F_{t}^{(\im)}[p,s] F_{t}^{(\im)}(p,s) + t^{y} F_{t}(M(\im; p, s))
\end{equation}
for some $x, y \in \frac{1}{2}\Z$, where 
the dominant monomial $M(\im; p, s) \in \cM$ is given by
$$
M(\im; p,s) \seq \prod_{\jm \sim \im} m^{(\jm)}(p,s).
$$
\end{Thm}

Note that the quantum $T$-system is already known for $\mathrm{ADE}$ types \cite[Propsition 5.6]{HL15} 
and for type $\mathrm{B}$ \cite[Theorem 9.6]{HO19}.

\begin{proof}
Let $i\seq \bar{\im}$ and $k\seq(s-p)/d_{i} \in \Z_{\ge 1}$. Note that
\begin{align*}
F_{t}^{(\im)}[p,s)&=F_{t}(m_{k,p}^{(i)}), & F_{t}^{(\im)}(p,s]&= F_{t}(m_{k,p+2d_{i}}^{(i)}), \\
F_{t}^{(\im)}[p,s]&=F_{t}(m_{k-1,p+2d_{i}}^{(i)}), & F_{t}^{(\im)}(p,s)&= F_{t}(m_{k+1,p}^{(i)}).
\end{align*}
Using the algorithm \cite[Section 5.7.1]{Hernandez04}, 
we get a more precise result than formula (\ref{eq:Fform}), that is
the analog of the statement of \cite[Lemma 5.5]{Hernandez06}. 
Then the same proof as in \cite{Hernandez06} gives the complete list of dominant monomials
occurring in the terms of the quantum $T$-system: 
$F_{t}(m_{k,p}^{(i)}) F_{t}(m_{k,p+2d_{i}}^{(i)})$ has exactly $k + 1$ dominant monomials 
$M, M_1, \ldots, M_k$ where $M, M_1,\ldots, M_{k-1}$ are the dominant monomials occurring in 
$F_{t}(m_{k-1,p+2d_{i}}^{(i)}) F_{t}(m_{k+1,p}^{(i)})$ and $M_k = M(\im; p,s)$. 
But an element in $\cK_t$ is characterized by the multiplicities of its dominant monomials \cite[Theorem 5.11]{Hernandez06}, 
so it suffices to check that the multiplicities on the monomials match. 
Here these multiplicities are just certain powers of $t$, 
so we can conclude as in \cite[Section 9]{HO19}. 
\end{proof}

\begin{Rem}
The exponents $x,y \in \frac{1}{2} \Z$ in \eqref{eq:qTsys} can be computed from the $t$-commutation relations in the quantum torus 
$\cY_{t}$ as in \cite[Proposition 5.6]{HL15}. In fact, we have 
$$
x +1 = y = 
\frac{\tc_{ii}(2(s-p)+d_{i}) + \tc_{ii}(2(s-p)-d_{i})}{2}.
$$ 
However we do not use this fact in this paper.
\end{Rem}

The following is the $t$-analog of Proposition~\ref{Prop:clTsys}

\begin{Prop} \label{Prop:qTsys}
Let $(\im, p), (\im, s) \in \hDs_{0}$ with $p<s$ as in {\rm Theorem~\ref{Thm:Tsys}}.
For $\varepsilon \in \{ +, - \}$ and any $m,m^\prime \in \cM^{\varepsilon}(\im; p, s)$, 
we have 
$$F_{t}(m)  F_{t}(m^\prime) = t^{\Nn(m,m^\prime)}F_{t}(m^\prime) F_{t}(m^\prime)$$
in the quantum Grothendieck ring $\cK_{t}$,
where $\Nn(m,m^\prime) \in \Z$ is as in~{\rm(\ref{eq:Nnmm})}. 
In other words, the elements in $\{ F_{t}(m) \mid m \in \cM^{\varepsilon}(\im; p, s)\}$  mutually commute 
up to powers of $t$.  
\end{Prop}
\begin{proof}
We prove the assertions similarly as in the proof of Proposition~\ref{Prop:clTsys} by using the truncation.
Again we consider only the case $\varepsilon = +$ and divide the situation into four cases (\ref{Tsys:case1})--(\ref{Tsys:case4}).
For the first two cases (\ref{Tsys:case1}) and (\ref{Tsys:case2}), we simply apply Lemma~\ref{Lem:qKR} to obtain the conclusion.

To deal with the third case (\ref{Tsys:case3}), 
we recall the $\Z$-algebra involution $\omega_{t}$ of $\cK_t$ from Section~\ref{ssec:fDt}.
Temporarily, let us assume the following equality
\begin{equation} \label{eq:omgtKR}
\omega_t\left(F_t^{(\im)}[p,s]\right) = F_t^{(\im^*)}[-s-rh^\vee,-p-rh^\vee]
\end{equation}
holds for all $(\im,p), (\im,s) \in \hDs_0$ with $p \le s$.
Then, as special cases, we obtain 
$$
\omega_t\left(F_t^{(\im)}(p,s)\right) = F_t^{(\im^*)}(-s-rh^\vee,-p-rh^\vee), \quad 
\omega_t\left(F_t^{(\im)}[p,s]\right) = F_t^{(\im^*)}[-s-rh^\vee,-p-rh^\vee].
$$
Combining with the same argument as in the case~(\ref{Tsys:case2}), we get the conclusion.
We postpone the proof of \eqref{eq:omgtKR} to the last paragraph. 

For the remaining case~(\ref{Tsys:case4}), we apply an $\mathfrak{sl}_2$-reduction argument
as in \cite[Remark 9.10]{HO19}:
 the multiplicities of commutative dominant monomials in the polynomials $F_t^{(\im)}[p,s]F_t^{(\im)}(p,s)$ and $F_t^{(\im)}(p,s)F_t^{(\im)}[p,s]$
are the same as those for the corresponding polynomial for the $\mathfrak{sl}_2$-case up to overall power of $t^{1/2}$. 
This follows from the fact that, 
in the quantum torus, the $t$-commutation relations between $A_{i,s}^{-1}$, $A_{i,s'}^{-1}$ 
and between $Y_{i,s}$,$A_{i,s'}^{-1}$ (with the same index $i \in I$) are the same as in the $\mathfrak{sl}_2$-case 
($q$ being replaced by $q_i$). Consequently, as the corresponding polynomial in the $\mathfrak{sl}_2$-case commute up to a power of $t$, we obtain the result. 

Finally, we prove the above equality \eqref{eq:omgtKR} by induction on $s-p \ge 0$.
For the case $s=p$, it follows immediately from the fact $F_t^{(\im)}[p,p] = F_t(Y_{\bar{\im}, p}) = L_t(Y_{\bar{\im},p})$ 
and Lemma~\ref{Lem:omgt}. For the induction step, we apply the \emph{anti}-automophism $\bar{\omega}_t \seq \omega_t \circ \ol{(\cdot)}$ 
to the quantum $T$-system identity \eqref{eq:qTsys}. Using the induction hypothesis and the $t$-commutativity
between $F_{t}^{(\im)}[p,s]$ and $F_{t}^{(\im)}(p,s)$ proved in the last paragraph, we have
$$
F_{t}^{(\im^*)}[s',p') F_{t}^{(\im^*)}(s',p'] = t^{x'} \bar{\omega}_t\left(F_{t}^{(\im)}[p,s]\right) F_{t}^{(\im^*)}(s',p') + t^{y} F_{t}(M(\im^*; s', p'))
$$
for some $x' \in \frac{1}{2} \Z$, where $s' \seq -s-rh^\vee, p' \seq -p-rh^\vee$.
Comparing it with the quantum $T$-system identity~\eqref{eq:qTsys} for $(\im^*, s')$ and $(\im^*, p')$, we find
$\bar{\omega}_t\left(F_{t}^{(\im)}[p,s]\right) = t^{x''}F_{t}^{(\im^*)}[s',p']$ for some $x'' \in \frac{1}{2}\Z$.
The bar-invariance forces $x'' = 0$, which completes the proof.
\end{proof}

\begin{Rem} \label{Rem:qTsys}
The above argument implies that
an ordered product 
of $\{F_{t}^{(\jm)}(p,s)\}_{\jm \sim \im}$ 
with respect to any total ordering of the set $\{ \jm \in \Delta_{0} \mid \jm \sim \im \}$
is equal to $F_{t}(M(\im; p,s))$ up to a power of $t$.
Therefore, the quantum $T$-system (\ref{eq:qTsys}) can be expressed as
$$
F_{t}^{(\im)}[p,s) F_{t}^{(\im)}(p,s]=
t^{a}F_{t}^{(\im)}[p,s] F_{t}^{(\im)}(p,s) + 
t^{b} \prod_{\jm \sim \im}^{\to} F_{t}^{(\jm)}(p,s)
$$
for some $a, b \in \frac{1}{2}\Z$.
\end{Rem}

\section{Quantized coordinate algebras}
\label{sec:coord}

Let $\sg$ be a complex simple Lie algebra of type $\mathrm{ADE}$
and $\Delta$ denote its Dynkin diagram. 
In this section, we collect some preliminaries 
on the quantum coordinate algebra of the maximal unipotent subgroup $N_-$ associated with $\sg$ 
and its quantum cluster algebra structure.
We will freely use the notation in Section~\ref{ssec:notr}. 

\subsection{Quantized coordinate algebra $\cA_{v}[N_{-}]$} 
\label{ssec:coord}

Let $v$ be an indeterminate 
with a formal square root $v^{1/2}$.
Let $\cU_{v}(\sn_{-})$ be the $\Q(v^{1/2})$-algebra presented by the set of generators 
$\{ f_{\im} \mid \im \in \Delta_{0} \}$ satisfying
\emph{the quantum Serre relations}:
\begin{equation} \label{eq:qSerref}
\begin{cases}
f_{\im} f_{\jm} - f_{\jm} f_{\im} = 0 & \text{if $\im \not \sim \jm$}, \\
f_{\im}^{2}f_{\jm} - (v+v^{-1}) f_{\im} f_{\jm} f_{\im} + f_{\jm} f_{\im}^{2}
= 0 & \text{if $\im \sim \jm$}.
\end{cases}
\end{equation}
This is regarded as a quantized enveloping algebra 
of the nilpotent Lie subalgebra $\sn_{-}$ of $\sg$
corresponding to the negative part. 
For each $\im \in \Delta_{0}$, 
we write $\wt(f_{\im}) = - \alpha_{\im}$, which endows the algebra $\cU_{v}(\sn_-)$ with a natural $(-\sQ^+)$-grading.  

The quantized coordinate algebra $\cA_{v}[N_-]$ is a $\Z[v^{\pm 1/2}]$-subalgebra of $\cU_{v}(\sn_-)$ 
which is dual to the canonical integral form of $\cU_{v}(\sn_-)$ generated by the divided powers 
$\{ f_{\im}^{m}/[m]_{v}! \mid \im \in \Delta_{0}, m \in \Z_{\ge 0} \}$ with respect to a suitable pairing.
See Appendix~\ref{ssec:cA} for its precise definition.
Here we recall that it satisfies the following properties:
$$
\Q(v^{\pm 1/2}) \otimes_{\Z[v^{\pm 1/2}]} \cA_{v}[N_{-}] \simeq \cU_{v}(\sn_{-}), \qquad
\C \otimes_{\Z[v^{\pm 1/2}]} \cA_{v}[N_{-}] \simeq \C[N_{-}],
$$
where $\C$ is regarded as a $\Z[v^{\pm 1/2}]$-algebra
via $v^{1/2} \mapsto 1$,
and $\C[N_{-}]$ is the coordinate algebra of the complex unipotent algebraic group $N_{-} \seq \exp (\sn_{-})$.
Note that $\C[N_{-}]$ is commutative.
We have an obvious algebra homomorphism 
$$
\evv \colon \cA_{v}[N_{-}] \to \C \otimes_{\Z[v^{\pm 1/2}]} \cA_{v}[N_{-}] \simeq \C[N_{-}]
$$  
called \emph{the specialization at $v=1$}.

For each $\lambda \in \sP^{+} \seq \sum_{\im \in \Delta_{0}} \Z_{\ge 0} \varpi_{\im}$
and Weyl group elements $w, w^\prime \in \sW$, 
we can define a specific homogeneous element $\tD_{w\lambda, w^{\prime} \lambda}$ of $\cA_{v}[N_{-}]$ 
called \emph{a normalized quantum unipotent minor}, as in \cite[Section~6]{Kimura12}, \cite[Section~5]{GLS13}.
Its precise definition is recalled in Appendix~\ref{ssec:minor}. 
From the construction, we have $\wt( \tD_{w\lambda, w^{\prime} \lambda}) = w \lambda - w^{\prime}\lambda$ when $\tD_{w\lambda, w^{\prime} \lambda}\neq 0$,
and 
$$
\tD_{w\lambda, w^{\prime}\lambda}
= \begin{cases}
1 & \text{if $w\lambda= w^{\prime}\lambda$}, \\
0 & \text{if $w\lambda- w^{\prime}\lambda \not \in - \sQ^{+}$}.
\end{cases}
$$ 

\subsection{Commutation classes} 
\label{ssec:cc}

Before describing the quantum cluster algebra structure of the quantized coordinate algebra $\cA_{v}[N_-]$,
we have to recall the generalities on the commutation classes
in the Weyl group $\sW$.

Two sequences $\bfi$ and $\bfi^{\prime}$ of elements of $\Delta_{0}$
are said to be \emph{commutation equivalent}
if $\bfi^{\prime}$ is obtained from $\bfi$ by applying a sequence of operations
which transform some adjacent components $(\im, \jm)$ such that $\im \not \sim \jm$ into $(\jm,\im)$.
This defines an equivalence relation.

For $w \in \sW$,
let $\cI(w) \subset \Delta_0^{\ell(w)}$ denote the set of all reduced words for $w$.
We refer to the equivalent class containing $\bfi$ as \emph{the commutation class of $\bfi$}
and
denote it by $[\bfi]$.
Note that the set $\cI(w)$ is divided into a disjoint union of
commutation classes.
In this paper, we only consider the case $w = w_{0}$.

For a reduced word  $\bfi = (\im_{1}, \ldots, \im_{\ell_{0}}) \in \cI(w_0)$,
we have
$\sR^{+}=\{  \beta^{\bfi}_{k} \mid 1 \le k \le \ell_{0} \}$, where
$\beta^{\bfi}_{k} \seq s_{\im_{1}}\cdots s_{\im_{k-1}}(\alpha_{\im_{k}})$.
Thus the reduced word $\bfi$ defines a total order $<_{\bfi}$ on $\sR^{+}$,
namely we write $\beta^{\bfi}_k <_{\bfi} \beta^{\bfi}_l$ if $k < l$.
Note that the ordering $<_{\bfi}$ is \emph{convex} in the sense that if
$\alpha,\beta,\alpha+\beta \in \sR^+$,
we have either $\alpha <_{\bfi} \alpha+\beta <_{\bfi} \beta$
or $\beta <_{\bfi} \alpha+\beta <_{\bfi} \alpha$.
For a commutation class $[\bfi]$ in $\cI(w_{0})$,
we define the convex partial ordering $\preceq_{[\bfi]}$ on $\sR^{+}$ so that
$$
\text{$\alpha \preceq_{[\bfi]} \beta$ if and only if $\alpha <_{\bfi^{\prime}} \beta$ for all $\bfi^{\prime} \in [\bfi]$}.
$$

For a commutation class $[\bfi]$ in $\cI(w_{0})$ and a positive root $\alpha \in \sR^{+}$,
we define \emph{the $[\bfi]$-residue of $\alpha$}, denoted by $\res^{[\bfi]}(\alpha)$,
to be $\im_k \in \Delta_{0}$ if we have $\alpha = \beta^{\bfi}_{k}$ with
$\bfi = (\im_{1}, \ldots, \im_{\ell_{0}})$.
Note that this is well-defined, i.e.~if
$\alpha = \beta^{\bfi^{\prime}}_{l}$
for another $\bfi^{\prime} = (\im_{1}^{\prime}, \ldots, \im_{\ell_{0}}^{\prime}) \in [\bfi]$, then we have $\im_{l}^{\prime} = \im_{k}$.

In~\cite{OS19a}, \emph{the combinatorial Auslander-Reiten quiver}
$\Upsilon_{\bfi}$ was introduced for each reduced word $\bfi = (\im_{1}, \ldots, \im_{\ell_{0}}) \in \cI(w_{0})$.
By definition, it is a quiver whose vertex set is $\sR^{+}$
and we have an arrow in $\Upsilon_{\bfi}$ from $\beta^{\bfi}_k$ to $\beta^{\bfi}_j$
if and only if $1\leq j < k \leq \ell_0$, $\im_{j} \sim \im_{k}$
and there is no index $j'$ between $j$ and $k$ such that $\im_{j'} \in \{  \im_j, \im_k\}$.
The quiver $\Upsilon_{\bfi}$ satisfies the following nice properties.
We say that a total ordering $\sR^{+} = \{\beta_{1}, \beta_{2}, \ldots, \beta_{\ell_0} \}$
is \emph{a compatible reading of $\Upsilon_{\bfi}$} if we have $k<l$
whenever there is an arrow $\beta_{l} \to \beta_{k}$ in $\Upsilon_{\bfi}$.

\begin{Thm}[\cite{OS19a}]
\label{Thm:OS19a}
For a commutation class $[\bfi]$ in $\cI(w_{0})$, we have the followings$\colon$
\begin{enumerate}
\item
If $\bfi^{\prime} \in [\bfi]$, then $\Upsilon_{\bfi} = \Upsilon_{\bfi^{\prime}}$. Hence $\Upsilon_{[\bfi]}$ is well-defined.
\item \label{OS19a:Hasse}
For $\alpha, \beta \in \sR^{+}$, we have
$\alpha \preceq_{[\bfi]} \beta$ if and only if there exists an oriented  path from $\beta$ to $\alpha$ in $\Upsilon_{[\bfi]}$.
In other words, the quiver $\Upsilon_{[\bfi]}$ is identical to the Hasse quiver of the partial ordering $\preceq_{[\bfi]}$.
\item A sequence $\bfi^{\prime} = (\im^{\prime}_{1}, \ldots, \im_{\ell_0}^{\prime}) \in \Delta_{0}^{\ell_0}$
belongs to the commutation class $[\bfi]$ if and only if there is a compatible reading
$\sR^{+} = \{\beta_{1}, \ldots, \beta_{\ell_0} \}$ of $\Upsilon_{[\bfi]}$
such that $\im^{\prime}_{k} = \res^{[\bfi]}(\beta_{k})$ for all $1 \le k \le \ell_0$.
\end{enumerate}
\end{Thm}

For a reduced word $\bfi = (\im_{1}, \im_{2}, \ldots, \im_{\ell_0}) \in \cI(w_{0})$,
we have $\bfi^{\prime} = (\im_{2}, \ldots, \im_{\ell_0}, \im^{*}_{1}) \in \cI(w_0)$ and
$[\bfi] \neq [\bfi^{\prime}]$.
This operation is referred to as \emph{a combinatorial reflection functor} and we write $\bfi^{\prime} = r_{\im_{1}} \bfi$.
We have the induced operation on commutation classes (i.e.,~$r_{\im_1} [\bfi] \seq [r_{\im_{1}} \bfi]$ is well-defined).
The relations $[\bfi] \overset{r}{\sim} [r_{\im}\bfi]$ for
various $\im \in \Delta_{0}$
generate an equivalence relation, called \emph{the reflection equivalent relation} $\overset{r}{\sim}$,
on the set of commutation classes in $\cI(w_{0})$.
For a given reduced word $\bfi \in \cI(w_0)$, the family of commutation classes
$[\![\bfi]\!] \seq \{ [\bfi^{\prime}] \mid [\bfi^{\prime}] \overset{r}{\sim} [\bfi] \}$
is called \emph{an $r$-cluster point}.

\begin{Thm}[\cite{OS19b, FO21}] \label{Thm:OS19b}
For each Q-datum $\cQ$ for $\fg$, there exists a unique commutation class $[\cQ]$ in the set $\cI(w_0)$
of reduced words for $w_0$
such that 
\begin{enumerate}
\item $[\cQ]$ consists of all the reduced words adapted to $\cQ$ in the sense of Definition~\ref{Def:adapt},
\item there is an isomorphism $\Upsilon_{[\cQ]} \to \Gamma_{\cQ}$ of quivers
whose underlying bijection $\sR^{+} \to (\Gamma_{\cQ})_{0}$ between the vertex sets 
coincides with the bijection $\phi_{\cQ}^{-1}(\cdot , 0)$,
\item for each $\alpha \in \sR^{+}$, we have $\res^{[\cQ]}(\alpha) = \im$ 
if $\phi_{\cQ}^{-1}(\alpha, 0) = (\im, p)$.
\end{enumerate}
Moreover, the set 
$ \{ [\cQ] \mid \text{\rm $\cQ$ is a Q-datum for $\fg$}\}$ forms a single $r$-cluster point.  
\end{Thm}
\subsection{Quantum cluster algebra structure}
\label{ssec:cl}
In this subsection, we briefly describe the quantum cluster algebra structure on the quantum coordinate algebra $\cA_{v}[N_-]$
based on \cite{GLS13}. Here we do not recall the general definition of a quantum cluster algebra,
which is reviewed in Appendix~\ref{sec:QCA}. 

We fix a reduced word $\bfi = (\im_{1}, \ldots, \im_{\ell_0}) \in \cI(w_{0})$ and 
use the following notation:
\begin{align*}
w_{\le k} &\seq s_{\im_1} s_{\im_2} \cdots s_{\im_k}, \quad w_{\le 0} \seq 1, \\
k^{+} &\seq \min(\{k+1 \le j \le \ell_0 \mid \im_j = \im_k \} \cup \{ \ell_0 + 1 \}), \\
k^{-} &\seq \max(\{0\} \cup \{1 \le j \le k-1 \mid \im_{j} = \im_k \}), \\
k^{-}(\im) &\seq \max(\{0\} \cup \{1 \le j \le k-1 \mid \im_j = \im \})  
\end{align*}
for $k = 1,2, \ldots, \ell_0$ and $\im \in \Delta_{0}$ with $\im \sim \im_{k}$.

\begin{Prop}[{\cite[Proposition 10.1, Lemma 11.3]{GLS13}}] \label{Prop:cp}
For a fixed reduced word $\bfi = (\im_{1}, \ldots, \im_{\ell_0}) \in \cI(w_{0})$,
we set
$$
J\seq \{ 1, 2, \ldots, \ell_{0} \},  \qquad
J_{f} \seq \{ j \in J \mid j^{+} = \ell_0 +1\}, \qquad
J_{e} \seq J \setminus J_{f}.  
$$
Define the $\Z$-valued $J\times J_{e}$-matrix $\tB^{\bfi} = (b_{st})_{s \in J, t \in J_{e}}$ as
$$
b_{st} = \begin{cases}
1 & \text{if either {\rm (i)} $t=s^{+}$, or  {\rm (ii)} $t<s<t^{+}<s^{+}$ and $\im_{s} \sim \im_{t}$}, \\
-1 & \text{if either {\rm (i)} $s=t^{+}$, or {\rm (ii)} $s<t<s^{+}<t^{+}$ and $\im_{s} \sim \im_{t}$}, \\
0 & \text{otherwise}.
\end{cases}
$$
Define the $\Z$-valued $J \times J$-skew-symmetric matrix $\Lambda^{\bfi} = (\Lambda_{st})_{s,t \in J}$ as
$$
\Lambda_{st} = (\varpi_{\im_s} - w_{\le s}\varpi_{\im_s}, \varpi_{\im_t} + w_{\le t}\varpi_{\im_t}) \quad \text{for $s<t$}.
$$ 
Then $(\Lambda^{\bfi}, \tB^{\bfi})$ forms a compatible pair, which satisfies
$$
\sum_{k \in J} b_{ks}\Lambda_{kt} = 2 \delta_{s, t}
$$
for all $s \in J_{e}$ and $t \in J$.
\end{Prop}

Recall the notation $\beta^{\bfi}_{k} = w_{\le k-1} \alpha_{\im_{k}}$ for $1 \le k \le \ell_0$
from Section~\ref{ssec:cc}. For each $\alpha \in \sR^{+}$, there exists a unique $k$ such that
$\alpha = \beta^{\bfi}_{k}$. Then we define
\begin{align*}
\alpha^{+} &\seq \begin{cases}
\beta^{\bfi}_{k^{+}} & \text{if $k^{+} \le \ell_{0}$}, \\
0 & \text{if $k^{+} = \ell_{0} + 1$},
\end{cases}
&
\alpha^{-} &\seq \begin{cases}
\beta^{\bfi}_{k^{-}} & \text{if $k^{-} > 0$}, \\
0 & \text{if $k^{-} = 0$},
\end{cases}
\\
\alpha^{-}(\im) &\seq \begin{cases}
\beta^{\bfi}_{k^{-}(\im)} & \text{if $k^{-}(\im) > 0$}, \\
0 & \text{if $k^{-}(\im) = 0$},
\end{cases}
&
\lambda_{\alpha}
&\seq w_{\le k}\varpi_{\im_k}, 
\end{align*}
where $\im \in \Delta_{0}$ is an element such that $\im \sim \im_k$.
Since $s_{\im} \varpi_{\jm} = \varpi_{\jm} - \delta_{\im,\jm}\alpha_{\im}$, we have
\begin{equation} \label{eq:lamalph}
\lambda_{\alpha^{-}} = \lambda_{\alpha} + \alpha
\end{equation}
for each $\alpha \in \sR^{+}$. 
Here we understand $\lambda_{\alpha^{-}} = \varpi_{\im_{k}}$
if $\alpha^{-} = 0$ (i.e.~if $k^{-} = 0$).

Recall that $\im_k = \res^{[\bfi]}(\alpha)$ depends only on the commutation class $[\bfi]$.
It is easy to see that the assignments 
$\alpha \mapsto \alpha^{+}, \alpha^{-}, \alpha^{-}(\im), \lambda_{\alpha}$
also depend only on the commutation class $[\bfi]$. 
Thus Proposition~\ref{Prop:cp} can be rephrased as follows.

\begin{Prop} \label{Prop:cp2}
For a fixed commutation class $[\bfi]$ in $\cI(w_{0})$,
we set
$$
J\seq \sR^{+},  \qquad
J_{f} \seq \{ \alpha \in \sR^{+} \mid \alpha^{+} = 0 \}, \qquad
J_{e} \seq J \setminus J_{f}.  
$$
Define the $\Z$-valued $J\times J_{e}$-matrix $\tB^{[\bfi]} = (b_{\alpha \beta})_{\alpha \in J, \beta \in J_{e}}$ as
$$
b_{\alpha \beta} = \begin{cases}
1 & \text{if either {\rm (i)} $\beta=\alpha^{+}$, or {\rm (ii)} $\beta \prec_{[\bfi]} \alpha \prec_{[\bfi]} \beta ^{+}\prec_{[\bfi]} \alpha^{+}$ and $\im\sim\jm$}, \\
-1 & \text{if either {\rm (i)} $\alpha=\beta^{+}$, or {\rm (ii)} $\alpha \prec_{[\bfi]} \beta \prec_{[\bfi]} \alpha^{+} \prec_{[\bfi]} \beta^{+}$ and $\im\sim\jm$}, \\
0 & \text{otherwise}.
\end{cases}
$$
Here $\im = \res^{[\bfi]}(\alpha), \jm = \res^{[\bfi]}(\beta)$.
Define the $\Z$-valued $J \times J$-skew-symmetric matrix $\Lambda^{[\bfi]} = (\Lambda_{\alpha \beta})_{\alpha, \beta \in J}$ as
$$
\Lambda_{\alpha \beta} = 
(\varpi_{\im} - \lambda_{\alpha}, \varpi_{\jm} + \lambda_{\beta}) \qquad \text{for $\beta \not \preceq_{[\bfi]} \alpha$},
$$
where $\im = \res^{[\bfi]}(\alpha), \jm = \res^{[\bfi]}(\beta)$.
Then $(\Lambda^{[\bfi]}, \tB^{[\bfi]})$ forms a compatible pair, which satisfies
$$
\sum_{\gamma \in J} b_{\gamma \alpha}\Lambda_{\gamma \beta} = 2 \delta_{\alpha, \beta}
$$
for all $\alpha \in J_{e}$ and $\beta \in J$. 
\end{Prop}

Let us fix a commutation class $[\bfi]$ in $\cI(w_{0})$. 
Associated with the $\Z$-valued $\sR^{+} \times \sR^{+}$-skew-symmetric matrix 
$\Lambda^{[\bfi]}$ in Proposition~\ref{Prop:cp2}, we define
the quantum torus $\cT_{v, [\bfi]}$  (denoted by $ \cT(\Lambda^{[\bfi]})$ in Appendix~\ref{ssec:Qseed}) as the $\Z[v^{\pm 1/2}]$-algebra
presented  by the set of generators $\{ X_{\alpha}^{\pm 1} \mid \alpha \in \sR^{+}\}$
and the following relations:
\begin{itemize}
\item $X_{\alpha} X_{\alpha}^{-1} = X_{\alpha}^{-1} X_{\alpha} = 1$ for $\alpha \in \sR^{+}$, 
\item $X_{\alpha} X_{\beta} = v^{\Lambda_{\alpha \beta}} X_{\beta} X_{\alpha}$ for $\alpha, \beta \in \sR^{+}$.
\end{itemize}
Let $\cA(\Lambda^{[\bfi]}, \tB^{[\bfi]})$ denote the quantum cluster algebra
associated with the compatible pair $(\Lambda^{[\bfi]}, \tB^{[\bfi]})$ 
in Proposition~\ref{Prop:cp2}. This is the $\Z[v^{\pm 1/2}]$-subalgebra 
of the quantum torus $\cT_{v, [\bfi]}$ generated by the union of the elements
called \emph{the quantum cluster variables}, which are obtained from 
all possible sequences of \emph{mutations} (see Appendix~\ref{ssec:mut}).
 
For $\alpha, \beta \in \sR^{+}$ with $\res^{[\bfi]}(\alpha) = \res^{[\bfi]}(\beta) = \im$, we write
$$
\tD^{[\bfi]}(\alpha, \beta) \seq \tD_{\lambda_{\alpha}, \lambda_{\beta}}, \qquad
\tD^{[\bfi]}(\alpha, 0) \seq \tD_{\lambda_{\alpha}, \varpi_{\im}}. 
$$
Note that $\tD^{[\bfi]}(\alpha, \beta) = 0$ unless $\beta \preceq_{[\bfi]} \alpha$.
We set $\tD^{[\bfi]}(0,0) \seq 1$ by convention.

\begin{Thm}[{\cite[Theorem 12.3]{GLS13}, \cite[Corollary 11.2.8]{KKKO18}}] \label{Thm:cl} 
For a commutation class $[\bfi]$ in $\cI(w_{0})$, there exists a $\Z[v^{\pm 1/2}]$-algebra isomorphism
$$
\CL \colon \cA(\Lambda^{[\bfi]}, \tB^{[\bfi]}) \xrightarrow{\sim} \cA_{v}[N_{-}] 
$$
given by 
$$
X_{\alpha} \mapsto \tD^{[\bfi]}(\alpha, 0)
$$
for all $\alpha \in \sR^{+}$. 
Moreover, the normalized quantum unipotent minors $\tD^{[\bfi]}(\alpha, \beta)$ and 
$\tD^{[\bfi]}(\alpha, 0)$ correspond to some quantum cluster variables.  
\end{Thm}

In what follows, we regard $\cA_{v}[N_{-}]$ as a subalgebra of $\cT_{v, [\bfi]}$ 
through $\CL^{-1}$.

The following system of equalities are regarded as 
a quantum analog of determinantal identities, which are also obtained as
specific exchange relations of quantum clusters
when we consider the quantum cluster algebra structure given in Theorem~\ref{Thm:cl}.

\begin{Prop}[{\cite[Proposition 5.5]{GLS13}}] \label{Prop:di}
Let $[\bfi]$ be a commutation class in $\cI(w_{0})$.
For any $\alpha, \beta \in \sR^{+}$ with
$\res^{[\bfi]}(\alpha) = \res^{[\bfi]}(\beta) = \im$ and $\beta \prec_{[\bfi]} \alpha$, 
there exist $a, b \in \frac{1}{2}\Z$ such that 
\begin{equation} \label{eq:di} 
\tD^{[\bfi]}(\alpha, \beta) \tD^{[\bfi]}(\alpha^{-}, \beta^{-})
= v^{a} \tD^{[\bfi]}(\alpha, \beta^{-})\tD^{[\bfi]}(\alpha^{-}, \beta)
+ v^{b} \prod^{\to}_{\jm \sim \im} \tD^{[\bfi]}(\alpha^{-}(\jm), \beta^{-}(\jm))
\end{equation}
in $\cA_{v}[N_{-}]$. Here the factors in the second term of the RHS mutually commute
up to powers of $v^{1/2}$ and thus
we can take any total ordering on the set 
$\{ \jm \in \Delta_{0} \mid \jm \sim \im\}$. 
\end{Prop}

\subsection{Dual canonical basis} 
\label{ssec:dcb}

We fix a commutation class $[\bfi]$ in $\cI(w_{0})$ as before.
For $\alpha \in \sR^{+}$, we set
$$
\tF(\alpha, [\bfi]) \seq \tD^{[\bfi]}(\alpha, \alpha^{-}).
$$
Note that $\deg \tF(\alpha, [\bfi]) = -\alpha$.
This is an analog of the dual function of a root vector
corresponding to $-\alpha$.
For $\bfc = (c_{\alpha})_{\alpha \in \sR^{+}} \in (\Z_{\ge 0})^{\sR^{+}}$, we set
$$\nu(\bfc) \seq 
-\frac{1}{2}\sum_{\alpha, \beta \in \sR^{+}} c_{\alpha}c_{\beta}(\alpha, \beta) +
\sum_{\alpha \in \sR^{+}} c_{\alpha}^{2}
\quad \in \Z
$$
and define
$$
\tF(\bfc, [\bfi]) \seq v^{\nu(\bfc)/2} \prod^{\to}_{\alpha\in\sR^{+}}\tF(\alpha, [\bfi])^{c_{\alpha}}.
$$
Here we take any total ordering on $\sR^{+} = \{\beta_{1}, \ldots, \beta_{\ell_0}\}$
such that we have $k > l$ whenever $\beta_{k} \prec_{[\bfi]} \beta_{l}$ 
(i.e.~the opposite of $<_{\bfi^{\prime}}$ for some $\bfi^{\prime} \in [\bfi]$) for the ordered product.
Since $\tF(\alpha, [\bfi])$ and $\tF(\beta, [\bfi])$
mutually commute if $\alpha$ and $\beta$ are not comparable with respect to $\prec_{[\bfi]}$,
the element $\tF(\bfc, [\bfi])$ is well-defined. 

\begin{Prop}[see, for instance, {\cite[Chapter 40, 41]{LusztigIntro}}] \label{Prop:pbw}
The set $\{ \tF(\bfc, [\bfi]) \mid \bfc \in (\Z_{\ge 0})^{\sR^{+}} \}$ forms a $\Z[v^{\pm 1/2}]$-basis of $\cA_{v}[N_{-}]$. 
\end{Prop}

We refer to the basis $\{ \tF(\bfc, [\bfi]) \mid \bfc \in (\Z_{\ge 0})^{\sR^{+}} \}$
as \emph{a normalized dual Poincar\'e-Birkhoff-Witt type basis}
(henceforth, \emph{a normalized dual PBW-type basis}).
In particular, $\cA_{v}[N_-]$ is generated by the set
$\{ \tF(\alpha, [\bfi]) \mid \alpha \in \sR^{+} \} = \{ \tD^{[\bfi]}(\alpha, \alpha^{-}) \mid \alpha \in \sR^{+}\}$.

There exists a natural $\Z$-algebra anti-involution $\ol{(\cdot)}$ on the quantum torus $\cT_{v, [\bfi]}$
given by 
$$
v^{\pm 1/2} \mapsto v^{\mp 1/2}, \qquad
X_{\alpha}^{\pm 1} \mapsto X_{\alpha}^{\pm 1}
$$
for $\alpha \in \sR^{+}$. 
This involution $\ol{(\cdot)}$ preserves the quantum cluster algebra 
$\cA(\Lambda^{[\bfi]}, \tB^{[\bfi]}) \simeq \cA_{v}[N_-]$
and fixes each quantum cluster variable.

Lusztig~\cite{Lus90, Lusztig91, LusztigIntro} and Kashiwara~\cite{Kas91} have constructed a specific 
$\Z[v^{\pm 1/2}]$-basis $\bfB$ of $\cA_{v}[N_{-}]$,
which is called \emph{the dual canonical basis}.
Moreover, we can consider the normalized dual canonical basis $\tbfB$.
We do not recall its precise definition here. 
Instead, we refer the following characterization of $\tbfB$ 
by using the normalized dual PBW-type basis.

\begin{Thm}[{\cite[Theorem 4.29]{Kimura12}}] \label{Thm:ndcb}
For each commutation class $[\bfi]$ in $\cI(w_{0})$,
we have
$$
\tbfB = \{ \tG(\bfc, [\bfi]) \mid \bfc \in (\Z_{\ge 0})^{\sR^{+}} \},
$$
where, for each $\bfc \in (\Z_{\ge 0})^{\sR^{+}}$,  $\tG(\bfc, [\bfi])$ 
is the unique element of $\cA_{v}[N_{-}]$ 
characterized by the following conditions$\colon$
\begin{itemize}
\item[(NDCB1)] $\ol{\tG(\bfc, [\bfi])} = \tG(\bfc, [\bfi])$, and
\item[(NDCB2)] $\tG(\bfc, [\bfi]) - \tF(\bfc, [\bfi]) \in \sum_{\bfc^{\prime} <_{[\bfi]} \bfc}
v \Z[v] \tF(\bfc^{\prime}, [\bfi])$.
\end{itemize}
Here the condition $\bfc^{\prime} = (c_{\alpha})_{\alpha \in \sR^{+}} <_{[\bfi]} \bfc=(c_{\alpha})_{\alpha \in \sR^{+}}$
means that there exists $\beta \in \sR^{+}$ such that $c_{\beta}^{\prime} < c_{\beta}$ and
$c_{\alpha}^{\prime} = c_{\alpha}$ for all $\beta \not \preceq_{[\bfi]} \alpha$.
\end{Thm}  

\begin{Rem} \label{Rem:tripbw}
The unitriangular property of the normalized dual PBW-type basis with respect to $\ol{(\cdot)}$
is also included in Theorem~\ref{Thm:ndcb}. That is, we have
$$
\tF(\bfc, [\bfi]) - \ol{\tF(\bfc, [\bfi])} \in \sum_{\bfc^{\prime} <_{[\bfi]} \bfc} \Z[v^{\pm 1}] \tF(\bfc^{\prime}, [\bfi]).
$$  
\end{Rem}

\section{The HLO isomorphisms}
\label{sec:HLO}

Throughout this section, we fix a Q-datum $\cQ = (\Delta, \sigma, \xi)$ for $\fg$.

The aim of this section is to establish an isomorphism $\Phi_{\cQ}$ for each Q-datum $\cQ$ between the quantum Grothendieck ring $\cK_{t, \cQ}$ of the category $\Cc_{\cQ}$
and the quantized coordinate algebra $\cA_v[N_-]$ associated with the simply-laced Lie algebra $\sg$ unfolding $\fg$. 
This is a generalization of the results due to Hernandez-Leclerc~\cite{HL15} and Hernandez-Oya~\cite{HO19}, and hence we call $\Phi_\cQ$ the HLO isomorphism. 

In Section~\ref{ssec:tori}, we prove an isomorphism $\Phi_{\cQ}^T$ between the quantum tori $\cY_{t, \cQ}$ and $\cT_{v, \cQ}$ 
from properties of the inverse of the quantum Cartan matrix.   
Then in Section~\ref{ssec:HLO}, we verify that $\Phi_{\cQ}^T$ induces the desired isomorphism $\Phi_{\cQ}$ by comparing
of the quantum $T$-system identities in $\cK_{t, \cQ}$ (Theorem~\ref{Thm:qTsys}) and 
the quantum determinantal identities in Proposition~\ref{Prop:di}.
In Section~\ref{ssec:cor}, as corollaries, we obtain the affirmative answer to the analog of Kazhdan-Lusztig conjecture in \cite{Hernandez04} for simple modules in the subcategory $\Cc_{\cQ}$ (see Conjecture \ref{Conj:KL}) and various positivities for its quantum Grothendieck ring $\cK_{t, \cQ}$.  
As far as the authors know, these are the first results in this direction beyond types $\mathrm{ABDE}$. Moreover, our approach is uniform for all non simply-laced types.

Henceforth, we often write $\cQ$ instead of $[\cQ]$ to simplify the notation.
For example, we apply it to write
$\cT_{v, \cQ}, \tD^{\cQ}(\alpha, \beta)$ and $\tF(\alpha, \cQ)$ 
instead of $\cT_{v, [\cQ]}, \tD^{[\cQ]}(\alpha, \beta)$ and $\tF(\alpha, [\cQ])$ respectively.

\subsection{An isomorphism between the quantum tori}
\label{ssec:tori}

Recall the notation for KR modules from Section~\ref{ssec:KR}.
Note that, for each $(\im,p) \in (\Gamma_{\cQ})_{0}$, the dominant monomial 
$$
m^{(\im)}[p, \xi_{\im}] = \prod_{k=0}^{(\xi_{\im} - p)/2d_{\bar{\im}}} Y_{\bar{\im}, p+2kd_{\bar{\im}}} 
$$
belongs to $\cM_{\cQ}$.

The goal of this subsection is to prove 
the following theorem, which is a generalization of \cite[Theorem 6.3]{HO19}.
\begin{Thm} \label{Thm:isomT}
For each Q-datum $\cQ$ for $\fg$,
there is the algebra isomorphism 
$$\Phi_{\cQ}^{T} \colon \cT_{v, \cQ} \to \cY_{t, \cQ}$$
given by
$$
v^{\pm 1/2} \mapsto t^{\pm 1/2}, \qquad 
X_\alpha \mapsto \ul{m^{(\im)}[p, \xi_{\im}]}  
$$  
for all $\alpha \in \sR^{+}$, where $(\im,p) = \phi_{\cQ}^{-1}(\alpha, 0)$.
Moreover, we have $\Phi_{\cQ}^{T} \circ \ol{(\cdot)} = \ol{(\cdot)} \circ \Phi_{\cQ}^{T}$.
\end{Thm}

The last assertion $\Phi_{\cQ}^{T} \circ \ol{(\cdot)} = \ol{(\cdot)} \circ \Phi_{\cQ}^{T}$ is obvious from the definition of $\Phi_{\cQ}^{T}$.
To prove the existence of $\Phi_{\cQ}^{T}$, 
we notice that the quantum torus $\cY_{t, \cQ}$ defined in Section~\ref{ssec:CQ} has another presentation given by the set of generators
$\{ \ul{m^{(\im)}[p, \xi_{\im}]}^{\pm 1} \mid (\im,p) \in (\Gamma_{\cQ})_{0}\}$ and the following two relations:
\begin{itemize}
\item $\ul{m^{(\im)}[p, \xi_{\im}]} \cdot \ul{m^{(\im)}[p, \xi_{\im}]}^{-1} 
= \ul{m^{(\im)}[p, \xi_{\im}]}^{-1}\cdot \ul{m^{(\im)}[p, \xi_{\im}]} = 1$ for $(\im,p) \in (\Gamma_{\cQ})_{0}$, 
\item $\ul{m^{(\im)}[p, \xi_{\im}]} \cdot \ul{m^{(\jm)}[s, \xi_{\jm}]} 
= t^{\kappa(\im,p; \jm,s)} \ul{m^{(\jm)}[s, \xi_{\jm}]} \cdot \ul{m^{(\im)}[p, \xi_{\im}]}$ for $(\im,p), (\jm,s) \in (\Gamma_{\cQ})_{0}$,
\end{itemize} 
where 
$$
\kappa(\im,p;\jm,s) \seq \Nn\left(m^{(\im)}[p, \xi_{\im}], m^{(\jm)}[s, \xi_{\jm}]\right) =  
\sum_{k=0}^{(\xi_{\im}-p)/2d_{\bar{\im}}} \sum_{l=0}^{(\xi_{\jm} - s)/2d_{\bar{\jm}}} 
\Nn(\bar{\im},p+2kd_{\bar{\im}}; \bar{\jm}, s+2ld_{\bar{\jm}}). 
$$
Therefore, it is enough to show that,
for any $(\im,p), (\jm,s) \in (\Gamma_{\cQ})_{0}$, we have the equality
\begin{equation} \label{eq:kaplam}
\kappa(\im, p ; \jm, s) = \Lambda_{\alpha\beta},
\end{equation}
where $\alpha, \beta \in \sR^{+}$ are given by $(\alpha, 0) = \phi_{\cQ}(\im,p)$ and $(\beta, 0) = \phi_{\cQ}(\jm,s)$.
A proof of (\ref{eq:kaplam}) will be given in the end of this subsection after some preparation.

Recall the notation $\alpha^{-}$ and $\lambda_{\alpha}$ for each $\alpha \in \sR^{+}$
from Section~\ref{ssec:cl}, which is associated with the adapted class $[\cQ]$ in the present situation.  

\begin{Lem} \label{Lem:cQalpha}
Let $\alpha \in \sR^{+}$ and set $\im \seq \res^{\cQ}(\alpha)$.
\begin{enumerate}
\item \label{cQalpha1}
We have $\alpha = \tau_{\cQ}^{d_{\bar{\im}}}\alpha^{-}$ provided that $\alpha^{-} \neq 0$.
In other words, we have 
$$\phi_{\cQ}^{-1}(\alpha^-, 0) = (\im, p + 2d_{\bar{\im}}) \quad \text{if $\phi_{\cQ}^{-1}(\alpha, 0) = (\im, p)$ with $p < \xi_{\im}$}. $$
\item \label{cQalpha2}
If $\alpha^{-}=0$, we have $\alpha = \gamma_{\im}^{\cQ} = \varpi_{\im} - \tau_{\cQ}^{d_{\bar{\im}}}\varpi_{\im}$.
\item \label{cQlambda}
We have $\lambda_{\alpha} = \tau_{\cQ}^{(\xi_{\im} - p)/2 + d_{\bar{\im}}}\varpi_{\im}$, where 
$p$ is such that $(\alpha, 0) = \phi_{\cQ}(\im, p)$.
\end{enumerate}
\end{Lem}
\begin{proof}
The assertions~(\ref{cQalpha1}) and (\ref{cQalpha2}) follow from Theorem~\ref{Thm:OS19b}.
The assertion~(\ref{cQlambda}) is proved inductively by the relation~(\ref{eq:lamalph})
along with (\ref{cQalpha1}) and (\ref{cQalpha2}).
\end{proof}

Given $i, j \in I$, we set
\begin{equation} \label{eq:Xi}
\Xi_{ij} \seq \max \{ \xi'_{\jm} - \xi'_{\im} \in \Z \mid \text{ $\xi'$ is a height function on $(\Delta, \sigma)$}, \im \in i, \jm \in j\}.
\end{equation}
We can easily check that $0 \le \Xi_{ij} \le rh^{\vee}$ holds for any $i,j \in I$. 

\begin{Lem} \label{Lem:tcxi}
Let $i,j \in I$.
For any $l \in \Z_{\ge 0}$, we have
\begin{equation} \label{eq:tcXi}
\tc_{ij}(\Xi_{ij} - d_{i} - 2ld_{ij}) = 0
\end{equation}
In particular, for any Q-datum $\cQ = (\Delta, \sigma, \xi)$ for $\fg$, $\im \in i, \jm \in j$ and $l \in \Z_{\ge 0}$, we have
\begin{equation} \label{eq:tcxi}
\tc_{ij}(\xi_{\jm} - \xi_{\im} - d_{i} - 2ld_{ij}) = 0.
\end{equation}
\end{Lem}
\begin{proof}
For any height function $\xi$ on $(\Delta, \sigma)$, $\im \in i$ and $\jm \in j$, there is $l \in \Z_{\ge 0}$ such that 
$\xi_{\jm} - \xi_{\im} = \Xi_{ij} - ld_{ij}$. 
Thus (\ref{eq:tcxi}) follows from (\ref{eq:tcXi}).
Since $\tc_{ij}(u) =0$ for $u \le 0$, it is enough to prove (\ref{eq:tcXi}) for any integer $l$ with $0 \le 2ld_{ij} \le \Xi_{ij} - d_i$.
First we consider the case $d_i \ge d_j$, or equivalently $d_{ij} = d_j$.
We choose a height function $\xi$ together with $\im \in i$ and $\jm \in j$ so that we have $\Xi_{ij} = \xi_{\jm} - \xi_{\im}$ and $\xi_\jm \ge \xi_{\jm^*}$.
Applying the formula (\ref{eq:tc}) in Theorem~\ref{Thm:tc} with $u = -(\Xi_{ij} - d_{i} - 2ld_{j})$, we have
$$
\tc_{ij}(u) - \tc_{ij}(-u) = (\varpi_{\im}, \tau_\cQ^{ld_{j}}\gamma_\jm^\cQ). 
$$
By the assumption $0 \le -u \le \Xi_{ij}-d_i < rh^\vee$, we have $\tc_{ij}(u) = 0$ and hence the LHS is equal to 
$-\tc_{ij}(-u)$, which is non-positive by Lemma~\ref{Lem:tc}~(\ref{tc:po}).  On the other hand, the RHS is non-negative because 
$\tau_\cQ^{ld_{j}}\gamma_\jm^\cQ$ is a positive root thanks to the fact $ld_{j} \le \Xi_{ij} \le \xi_{\jm} - \xi_{\jm^*} + rh^{\vee}$ and Proposition~\ref{Prop:tAR}. 
Thus we get $\tc_{ij}(-u) = 0$, which proves (\ref{eq:tcXi}) for the case $d_i \ge d_j$.
For the other case $d_i < d_j$, we necessarily have $(d_i, d_j) = (1,r)$ and hence $d_{ij}=1$. 
Moreover, we observe that $\Xi_{ij} = \Xi_{ji}-2(r-1)$, or equivalently $\Xi_{ij}-d_i = \Xi_{ji}-d_j-r+1$ holds by Definition~\ref{def:Qdata}.
Then we get
$$
\tc_{ij}(\Xi_{ij} - d_{i} - 2ld_{ij}) = \sum_{k=0}^{r-1} \tc_{ji}(\Xi_{ji} - d_{j} - 2(l+k))=0
$$
as desired, 
where the first equality follows from Lemma~\ref{Lem:tc}~(\ref{tc:tp}).  
\end{proof}

\begin{Prop} \label{Prop:NnKR}
Let $\cQ=(\Delta, \sigma, \xi)$ be a Q-datum for $\fg$.
Let $(\im, p), (\im,p^\prime), (\jm, s), (\jm, s^\prime) \in \hDs_0$
with $p \le p^\prime$ and $s \le s^\prime$. Assuming
$p - s \le \Xi_{\bar{\im}\bar{\jm}}$ and $s^\prime - p^\prime \le \Xi_{\bar{\im}\bar{\jm}}$,
we have
$$
\Nn\left(m^{(\im)}[p,p^\prime], m^{(\jm)}[s, s^\prime]\right) 
= \left(\tau_{\cQ}^{(\xi_{\im} -p)/2 + d_i}\varpi_{\im} + \tau_{\cQ}^{(\xi_{\im} - p^\prime)/2}\varpi_{\im}, 
\tau_{\cQ}^{(\xi_{\jm}-s)/2 + d_j}\varpi_{\jm} - \tau_{\cQ}^{(\xi_\jm -s^\prime)/2}\varpi_\jm\right).
$$
\end{Prop}

\begin{proof}
Set $i \seq \bar{\im}, j \seq \bar{\jm} \in I.$
By the definition of $\Nn(i,p; j, s)$ in (\ref{eq:defNn}), we have
\begin{align*}
&\Nn\left(m^{(\im)}[p,p^\prime], m^{(\jm)}[s, s^\prime] \right) = \sum_{k=0}^{(p^\prime-p)/2d_{i}} \sum_{l=0}^{(s^\prime-s)/2d_{j}} 
\Nn(i, p+2kd_{i} ; j, s+ 2ld_{j}) \allowdisplaybreaks\\
& =  \sum_{k=0}^{(p^\prime - p)/2d_{i}} \sum_{l=0}^{(s^\prime-s)/2d_{j}}
\{ 
\tc_{ij}((p+2kd_{i}) - (s+2ld_{j})-d_{i}) - \tc_{ij}((p+2kd_{i}) - (s+2ld_{j})+d_{i})\\
& \qquad \qquad \qquad \qquad 
-\tc_{ij}((s+2ld_{j}) - (p+2kd_{i})-d_{i}) + \tc_{ij}((s+2ld_{j}) - (p+2kd_{i})+d_{i}) 
\} \allowdisplaybreaks\\
&= \sum_{l=0}^{(s^\prime-s)/2d_{j}}\{ 
\tc_{ij}(p - (s+2ld_{j})-d_{i}) - \tc_{ij}(p^\prime - (s^\prime-2ld_{j})+d_{i}) \\
& \qquad \qquad \qquad
-\tc_{ij}((s^\prime-2ld_{j}) - p^\prime-d_{i}) + \tc_{ij}((s+2ld_{j}) - p+d_{i}) 
\} \allowdisplaybreaks\\
&= \sum_{l=0}^{(s^\prime-s)/2d_{j}}\{ 
 - \tc_{ij}(2ld_{j} - s^\prime +p^\prime+d_{i}) + \tc_{ij}(2ld_{j}+s - p+d_{i}) \},
\end{align*}
where the last equality follows from Lemma~\ref{Lem:tcxi}.
On the other hand, we compute
\begin{align*}
&\left(\tau_{\cQ}^{(\xi_{\im} -p)/2 + d_i}\varpi_{\im} + \tau_{\cQ}^{(\xi_{\im} - p^\prime)/2}\varpi_{\im}, 
\tau_{\cQ}^{(\xi_{\jm}-s)/2 + d_j}\varpi_{\jm} - \tau_{\cQ}^{(\xi_\jm -s^\prime)/2}\varpi_\jm\right) \allowdisplaybreaks \\
&=
-\left(\tau_{\cQ}^{(\xi_{\im} -p)/2 + d_i}\varpi_{\im} + \tau_{\cQ}^{(\xi_{\im} - p^\prime)/2}\varpi_{\im}, 
\tau_{\cQ}^{(\xi_{\jm}-s^\prime)/2}(1 - \tau_{\cQ}^{(s^\prime-s)/2+d_j})\varpi_\jm\right) \allowdisplaybreaks \\
&= 
-\sum_{l = 0}^{(s^\prime -s)/2d_j} 
\left(\tau_{\cQ}^{(\xi_{\im} -p)/2 + d_i}\varpi_{\im} + \tau_{\cQ}^{(\xi_{\im} - p^\prime)/2}\varpi_{\im}, 
\tau_{\cQ}^{(2ld_j+\xi_{\jm}-s^\prime)/2}\gamma_{\jm}^{\cQ}\right) \allowdisplaybreaks\\
&= -\sum_{l=0}^{(s^\prime-s)/2d_{j}}
\left(\varpi_{\im}, \tau_{\cQ}^{((2ld_{j}-s^\prime+p-d_{i}) + \xi_{\jm} -\xi_{\im}-d_{i})/2}\gamma^{\cQ}_{\jm} 
+ \tau_{\cQ}^{((2ld_{j}-s^\prime+p^\prime+d_{i}) + \xi_{\jm} -\xi_{\im}-d_{i})/2}\gamma^{\cQ}_{\jm}\right) 
\allowdisplaybreaks\\
&=- \sum_{l=0}^{(s^\prime-s)/2d_{j}}\{\tc_{ij}(2ld_{j}-s^\prime+p-d_{i}) - \tc_{ij}(-2ld_{j}+s^\prime-p+d_{i}) \\ 
& \qquad \qquad \quad 
+\tc_{ij}(2ld_{j}-s^\prime+p^\prime+d_{i}) - \tc_{ij}(-2ld_{j}+s^\prime-p^\prime-d_{i})
\}\allowdisplaybreaks\\
&= -\sum_{l=0}^{(s^\prime-s)/2d_{j}}\{\tc_{ij}(2ld_{j}-s^\prime+p^\prime+d_{i}) - \tc_{ij}(-2ld_{j}+s'-p+d_{i})\} \\
&= -\sum_{l=0}^{(s^\prime-s)/2d_{j}}\{\tc_{ij}(2ld_{j}-s^\prime+p^\prime+d_{i}) - \tc_{ij}(2ld_{j}+s-p+d_{i})\} 
\end{align*}
where the second equality follows from~\eqref{eq:gamma}, the fourth equality follows by Theorem~\ref{Thm:tc}, and we again used Lemma~\ref{Lem:tcxi} for the fifth equality. 
From the above two computations, we obtain the conclusion.
\end{proof}


\begin{Rem} \hfill
\begin{enumerate}
\item In \cite{KKOP20}, a $\Z$-valued invariant $\Lambda(V, W)$ was defined for any \emph{rationally renormalizable} pair $(V,W)$ of non-zero $U_q(L\fg)$-modules. 
It plays a quite important role in the theory of monoidal categorification of cluster algebras. 
If $(V, W)$ is a \emph{commuting} pair of simple modules, say $(L(m), L(m^\prime))$, 
we have $\Lambda(L(m), L(m^\prime)) = \Nn(m, m^\prime)$ (see \cite[Section 6]{FO21}). 
 From this point of view, Proposition~\ref{Prop:NnKR} gives a computation 
of the invariants $\Lambda(V, W)$ for a larger class of commuting pairs of KR modules.
\item When $m^{(\im)}[p,p^\prime], m^{(\jm)}[s, s^\prime]$ in Proposition \ref{Prop:NnKR} are contained in some $\cM_\cQ$, the $\Z$-integer value
$\Nn(m^{(\im)}[p,p^\prime], m^{(\jm)}[s, s^\prime])$ coincides also with 
the $\Z$-integer value in \cite[Proposition 10.2.3 (ii)]{KKKO18} by \cite[Theorem 4.12]{KKOP20c} (see also \cite[Theorem 4.12]{KKOP18}), since there are certain two simple commuting modules over symmetric quiver Hecke algebra of type $\mathsf{g}$ corresponding to $L(m^{(\im)}[p,p^\prime])$ and  $L(m^{(\jm)}[s, s^\prime])$ via the generalized quantum affine Schur-Weyl duality functor $F_\cQ$ (see Subsection~\ref{ssec:SW} below).
\end{enumerate}
\end{Rem}

\begin{proof}[End of the proof of {\rm Theorem~\ref{Thm:isomT}}]
To prove the equality (\ref{eq:kaplam}), we apply Proposition~\ref{Prop:NnKR} to the case 
$(\im,p) = \phi_{\cQ}^{-1}(\alpha, 0)$, $(\jm, s) = \phi_{\cQ}^{-1}(\beta, 0)$,
$p^\prime = \xi_\im$ and $s^\prime = \xi_\jm$.
Since both sides of (\ref{eq:kaplam}) are skew-symmetric, we may assume that $p \le s$. 
Then we have $\alpha \not \preceq_{[\bfi]} \beta$ and the conditions $p-s \le \Xi_{\bar{\im}\bar{\jm}}$ and $s^\prime - p^\prime \le \Xi_{\bar{\im}\bar{\jm}}$ trivially hold.
Applying Proposition~\ref{Prop:NnKR} and Lemma~\ref{Lem:cQalpha}~(\ref{cQlambda}),
we obtain 
$$ 
\Nn\left(m^{(\im)}[p, \xi_\im], m^{(\jm)}[s, \xi_\jm]\right) 
= (\lambda_\alpha + \varpi_{\im}, 
\lambda_\beta - \varpi_\jm) = - \Lambda_{\beta \alpha},
$$
which proves (\ref{eq:kaplam}).
\end{proof}

\subsection{The HLO isomorphisms} 
\label{ssec:HLO}

In this subsection, we prove the following result which generalizes \cite[Theorem 6.1]{HL15} (for simply-laced $\fg$)
and \cite[Theorem 10.1]{HO19} (for $\fg$ of type $\mathrm{B}$) to any simple Lie algebra $\fg$.
Our proof is similar to that of \cite{HO19}.
Let $\cK^{T}_{t, \cQ} \subset \cY_{t, \cQ}$ denote the image of the truncation map 
$(\cdot)_{\le \xi} \colon \cK_{t, \cQ} \hookrightarrow \cY_{t, \cQ}$ (see Corollary~\ref{Cor:tr}). 

\begin{Thm} \label{Thm:HLO}
For any Q-datum $\cQ$ for $\fg$, the isomorphism 
$\Phi_{\cQ}^{T} \colon \cT_{v, \cQ} \to \cY_{t,\cQ}$ in {\rm Theorem~\ref{Thm:isomT}}
restricts to an isomorphism 
$$
\cA_{v}[N_{-}] \to \cK_{t, \cQ}^{T},
$$ 
which sends the normalized dual canonical basis $\tbfB$ in $\cA_{v}[N_{-}]$
to the basis formed by the truncated $(q,t)$-characters of simple modules in $\cK_{t, \cQ}^{T}$.
More precisely, for each $\bfc = (c_{\alpha})_{\alpha \in \sR^{+}} \in (\Z_{\ge 0})^{\sR^{+}},$ we have 
$\Phi_{\cQ}^{T}(\tG(\bfc, \cQ)) = L_{t}(m(\bfc))_{\le \xi}$, where $m(\bfc) \in \cM_{\cQ}$ is given by
\begin{equation} \label{eq:mc}
u_{i,p}(m(\bfc)) = c_{\alpha}  \qquad \text{when $\bphi_{\cQ}(i,p) = (\alpha, 0)$}
\end{equation}
for each $(i,p) \in \hI_{\cQ}$.
\end{Thm}
\begin{proof}
We prepare some notation.
Recall
$(\Gamma_{\cQ})_{0} = \{ (\im,p) \in \hDs_{0} \mid \xi_{\im^*} - rh^{\vee} < p \le \xi_{\im}\}$
from Proposition~\ref{Prop:tAR}. 
Let
$$
S_{\cQ} \seq \{ (\im; p, s) \in \Delta_{0} \times \Z^{2} \mid (\im, p), (\im, s) \in \hDs_{0},  
\xi_{\im^*} - rh^{\vee} < p \le s \le \xi_{\im} + 2d_{\bar{\im}} \}
$$
and choose a total ordering $\prec$ on $S_{\cQ}$ satisfying
$$
\text{$(\im; p, s) \prec (\im^{\prime}; p^{\prime}, s^{\prime} )$ 
if either $s > s^{\prime}$, or $s=s^{\prime}$ and $p>p^{\prime}$.} 
$$ 
For each $(\im; p, s) \in S_{\cQ}$, we set
$$
\tD(\im; p, s)  \seq \begin{cases}
\tD^{\cQ}(\alpha, \beta) & \text{if $s \le \xi_{\im}$}, \\
\tD^{\cQ}(\alpha, 0) & \text{if $s=\xi_{\im} +2d_{\bar{\im}}$}, \\
1 & \text{if $p=s$},
\end{cases}
$$
where $\alpha, \beta \in \sR^{+}$ correspond to $(\im,p), (\im, s) \in (\Gamma_{\cQ})_{0}$
respectively via the bijection $\phi_{\cQ}^{-1}(\cdot, 0)$. 
Note that the equality~(\ref{eq:di}) can be rewritten as
\begin{equation} \label{eq:di2}
\tD(\im; p,s) \tD(\im; p^{+}, s^{+}) = 
v^{a} \tD(\im; p,s^{+}) \tD(\im; p^{+}, s) 
+v^{b} \prod_{\jm \sim \im}^{\to}\tD(\jm; p^{+}(\jm), s^{+}(\jm)) 
\end{equation}
where $p^{+} \seq p+2d_{\bar{\im}}$ and 
$p^{+}(\jm) \seq \min\{ p^{\prime} \in \Z \mid (\jm, p^{\prime}) \in \hDs_{0}, p<p^{\prime}\}$
(see Theorem~\ref{Thm:OS19a} and Lemma~\ref{Lem:cQalpha}).
Proposition~\ref{Prop:di} tells that the equality (\ref{eq:di2}) holds 
for each $(\im; p, s) \in S_{\cQ}$ with $p<s\le\xi_{\im}$ if we choose suitable $a, b \in \frac{1}{2}\Z$ (with an ordering of $\{\jm \in \Delta_0 \mid \jm \sim \im \}$ fixed).

To prove $\Phi_{\cQ}^{T}(\cA_{v}[N_{-}]) = \cK_{t, \cQ}^{T}$, we only have to show
\begin{equation} \label{eq:DtoKR}
\Phi_{\cQ}^{T}\left(\tD(\im; p,s)\right) = F_{t}^{(\im)}[p,s)_{\le \xi}
\end{equation}
for each $(\im; p,s) \in S_{\cQ}$ in view of Proposition~\ref{Prop:pbw}.
We proceed by induction with respect to the ordering $\prec$.  
In the case when $s=\xi_{\im}+2d_{\bar{\im}}$, we have
$$
\Phi_{\cQ}^{T}\left(\tD(\im; p,s)\right) = \Phi_{\cQ}^{T}\left( \tD^{\cQ}(\alpha, 0)\right) = \Phi_{\cQ}^{T}(X_{\alpha}) = 
\ul{m^{(\im)}[p, \xi_{\im}]} = F_{t}^{(\im)}[p, \xi_{\im}]_{\le \xi} = F_{t}^{(\im)}[p,s)_{\le \xi}
$$
thanks to Theorem~\ref{Thm:Fform}, where $(\alpha, 0) = \phi_{\cQ}(\im, p)$. In the case $p=s$, it is trivial. 
In what follows, we focus on the other case $p<s\le\xi_{\im}$.

We denote by $\F(\cT_{v, \cQ})$ (resp.~$\F(\cY_{t, \cQ})$) be the skew-field of fractions 
of $\cT_{v, \cQ}$ (resp.~$\cY_{t, \cQ}$). We can extend the $\Z$-algebra anti-involution $\ol{(\cdot)}$
on $\cT_{v, \cQ}$ (resp.~$\cY_{t, \cQ}$) to the $\Q$-algebra anti-involution on $\F(\cT_{v, \cQ})$ (resp.~$\F(\cY_{t,\cQ})$),
denoted again by $\ol{(\cdot)}$, and extend the $\Z$-algebra isomorphism $\Phi_{\cQ}^{T}$
to the $\Q$-algebra isomorphism $\Phi_{\cQ}^{T} \colon \F(\cT_{v, \cQ}) \to \F(\cY_{t, \cQ})$. They still satisfy
$\Phi_{\cQ}^{T} \circ \ol{(\cdot)} = \ol{(\cdot)} \circ \Phi_{\cQ}^{T}$. 

It is shown in~\cite[Section 13]{GLS13} that the equality (\ref{eq:di2}) (more generally (\ref{eq:di})) 
corresponds to a specific exchange relation of quantum clusters in the quantum cluster algebra structure
given in Theorem~\ref{Thm:cl}. In particular, the (non-trivial) normalized quantum unipotent minors 
corresponding to the following elements of $S_{\cQ}$
\begin{equation} \label{eq:list}
(\im; p^{+}, s^{+}), \qquad (\im; p, s^{+}), \qquad (\im; p^{+}, s), \qquad (\jm; p^{+}(\jm), s^{+}(\jm))
\end{equation}
simultaneously belong to a certain quantum cluster. 
Therefore the element $\tD(\im; p,s)$ is characterized as the $\ol{(\cdot)}$-invariant element of the form
$$
\left( 
v^{a} \tD(\im; p,s^{+}) \tD(\im; p^{+}, s) 
+v^{b} \prod_{\jm \sim \im}^{\to}\tD(\jm; p^{+}(\jm), s^{+}(\jm)) 
\right)  \tD(\im; p^{+}, s^{+})^{-1}
$$
for some $a,b \in \frac{1}{2}\Z$ in $\F(\cT_{v, \cQ})$. 
Since all the elements listed in (\ref{eq:list}) are smaller than $(\im; p, s)$ with respect to $\prec$,
we can apply the induction hypothesis to find that 
the element $\Phi_{\cQ}^{T}(\tD(\im; p,s))$ is characterized as the $\ol{(\cdot)}$-invariant element of the form
$$
\left( 
t^{a}  F_{t}^{(\im)}[p,s^{+})_{\le \xi}\cdot F_{t}^{(\im)}[p^{+}, s)_{\le \xi} 
+t^{b} \prod_{\jm \sim \im}^{\to}F_{t}^{(\jm)}[p^{+}(\jm), s^{+}(\jm))_{\le \xi} 
\right) F_{t}^{(\im)}[p^{+}, s^{+})^{-1}_{\le \xi}
$$
for some $a,b \in \frac{1}{2}\Z$ in $\F(\cY_{t, \cQ})$. 
On the other hand, by Theorem~\ref{Thm:qTsys} (see also Remark~\ref{Rem:qTsys}), 
$F_{t}^{(\im)}[s,p)_{\le \xi}$ is the $\ol{(\cdot)}$-invariant element of the form
$$
\left(t^{a^{\prime}} 
F_{t}^{(\im)}[p,s]_{\le \xi} \cdot F_{t}^{(\im)}(p,s)_{\le \xi} + 
t^{b^{\prime}} \prod_{\jm \sim \im}^{\to} F_{t}^{(\jm)}(p,s)_{\le \xi}
\right) F_{t}^{(\im)}(p,s]^{-1}_{\le \xi} 
$$
for some $a^{\prime}, b^{\prime} \in \frac{1}{2}\Z$ in $\F(\cY_{t, \cQ})$.
Therefore, taking the following equalities into the account
\begin{align*} 
F_{t}^{(\im)}[p^{+}, s^{+}) &= F_{t}^{(\im)}(p,s], & 
F_{t}^{(\im)}[p,s^{+}) &= F_{t}^{(\im)}[p,s], \\ 
F_{t}^{(\im)}[p^{+}, s) &= F_{t}^{(\im)}(p,s), &
F_{t}^{(\jm)}[p^{+}(\jm), s^{+}(\jm)) &= F_{t}^{(\jm)}(p,s),
\end{align*}
where $\jm \sim \im$, we conclude $\Phi_{\cQ}^{T}(\tD(\im; p,s)) = F_{t}^{(\im)}[p,s)_{\le \xi}$ as desired.

Next, we show the bijection between the distinguished bases.
It follows from (NDCB1) in Theorem~\ref{Thm:ndcb} and the identity 
$\Phi_{\cQ}^{T} \circ \ol{(\cdot)} = \ol{(\cdot)} \circ \Phi_{\cQ}^{T}$
that
the element $\Phi_{\cQ}^{T}(\tG(\bfc, \cQ))$ is $\ol{(\cdot)}$-invariant for any $\bfc \in (\Z_{\ge 0})^{\sR^{+}}$.
Hence, by Theorem~\ref{Thm:qtch} (see also Remark~\ref{Rem:trS}) and Theorem~\ref{Thm:ndcb}, it remains to show that
\begin{equation}
\Phi_{\cQ}^{T}\left( \tF_{t}(\bfc, \cQ)\right) = E_{t}(m(\bfc))_{\le \xi}
\end{equation}  
for every $\bfc \in (\Z_{\ge 0})^{\sR^{+}}$.
As a special case of~(\ref{eq:DtoKR}), we have already proved 
$$\Phi_{\cQ}^{T}\left(\tF_{t}(\alpha, \cQ)\right) = \Phi_{\cQ}^{T}\left(\tD^{\cQ}(\alpha, \alpha^{-})\right) = F_{t}(Y_{i,p})_{\le \xi}$$
for each $\alpha \in \sR^{+}$, where $(i,p) = \bphi_{\cQ}^{-1}(\alpha, 0)$. 
Moreover, the ordering of the product in the definition of 
$\tF(\bfc, \cQ)$ matches the one in the definition of $E_{t}(m(\bfc))$
in view of Theorem~\ref{Thm:OS19a}. 
Therefore, we find
$$
\Phi_{\cQ}^{T}\left(\tF(\bfc, \cQ)\right) = t^{a(\bfc)} E_{t}(m(\bfc))_{\le \xi} 
$$  
for some $a(\bfc) \in \frac{1}{2}\Z$.
It suffices to show that $a(\bfc)=0$ for every $\bfc \in (\Z_{\ge 0})^{\sR^{+}}$.
By Remark~\ref{Rem:triE} and Remark~\ref{Rem:tripbw}, we have
\begin{align*}
\ol{E_{t}(m)}  &\in E_{t}(m) + \sum_{m^{\prime} < m} \Z[t^{\pm 1}] E_{t}(m^{\prime}), \\
\ol{\tF(\bfc, \cQ)} & \in \tF(\bfc, \cQ) + \sum_{\bfc^{\prime} <_{\cQ} \bfc} \Z[v^{\pm 1}] \tF(\bfc^{\prime}, \cQ).
\end{align*}
Therefore, we obtain  
\begin{align*}
\ol{\Phi_{\cQ}^{T}\left(\tF(\bfc, \cQ) \right)} &= \Phi_{\cQ}^{T}\left( \ol{\tF(\bfc, \cQ)}\right) 
\in t^{a(\bfc)}E_{t}(m(\bfc))_{\le \xi} + \sum_{m^{\prime} \neq m(\bfc)}\Z[t^{\pm 1/2}]E_{t}(m^{\prime})_{\le \xi}, \\
\ol{\Phi_{\cQ}^{T}\left(\tF(\bfc, \cQ) \right)} &= \ol{t^{a(\bfc)} E_{t}(m(\bfc))_{\le \xi}} 
\in t^{-a(\bfc)}E_{t}(m(\bfc))_{\le \xi} + \sum_{m^{\prime} \neq m(\bfc)}\Z[t^{\pm 1/2}]E_{t}(m^{\prime})_{\le \xi}.
\end{align*}
Hence, $a(\bfc) = -a(\bfc)$, which implies $a(\bfc)=0$.
\end{proof}
  
\begin{Cor}[The HLO isomorphism] \label{Cor:HLO}
For each Q-datum $\cQ$ for $\fg$, 
there exists a $\Z$-algebra isomorphism 
$$\Phi_{\cQ} \colon \cA_{v}[N_-] \to \cK_{t, \cQ}$$ such that
$(\cdot)_{\le \xi} \circ \Phi_{\cQ} = \Phi_{\cQ}^{T}$. This isomorphism induces a bijection between the bases
$\tbfB$ and 
$\{ L_{t}(m) \mid m \in \cM_{\cQ} \}$.   
\end{Cor}

\subsection{Corollaries of the HLO isomorphism} 
\label{ssec:cor}

Fix a Q-datum $\cQ = (\Delta, \sigma, \xi)$ for $\fg$.
We have the following corollaries of the HLO isomorphism $\Phi_{\cQ}$.
Since the proofs can be completely same as those in~\cite[Section 11.1]{HO19}, we omit them.

\begin{Cor}[Positivity of the analog of Kazhdan-Lusztig polynomials]
Recall the notation in {\rm Theorem~\ref{Thm:HLO}} and {\rm Corollary~\ref{Cor:HLO}}. We have the following$\colon$
\begin{enumerate}
\item $\Phi_{\cQ}\left( \tF(\bfc, \cQ) \right) = E_{t}(m(\bfc))$ for all $\bfc \in (\Z_{\ge 0})^{\sR^{+}}$.
\item For $\bfc \in (\Z_{\ge 0})^{\sR^{+}}$ and $m \in \cM_{\cQ}$, write
$$
\tF(\bfc, \cQ) = \sum_{\bfc^{\prime}} p_{\bfc, \bfc^{\prime}}(v)\tG(\bfc^{\prime}, \cQ), \qquad
E_{t}(m) = \sum_{m^{\prime}} P_{m, m^{\prime}}(t) L_{t}(m^{\prime})
$$
with $p_{\bfc, \bfc^{\prime}}(v) \in \Z[v]$ and $P_{m, m^{\prime}}(t) \in \Z[t]$. Then we have
$$
p_{\bfc, \bfc^{\prime}}(t) = P_{m(\bfc), m(\bfc^{\prime})}(t). 
$$
Moreover, $P_{m(\bfc), m(\bfc^{\prime})}(t) \in \Z_{\ge 0}[t]$.  
\end{enumerate}
\end{Cor} 

\begin{Cor}[Positivity of structure constants]
\label{Cor:P2}
For $m_1, m_2 \in \cM_{\cQ}$, write
$$
L_{t}(m_{1}) L_{t}(m_{2}) = \sum_{m \in \cM_{\cQ}} c^{m}_{m_{1}, m_{2}}(t) L_{t}(m).
$$ 
Then we have 
$
c^{m}_{m_{1}, m_{2}}(t) \in \Z_{\ge 0}[t^{\pm 1/2}].
$
\end{Cor}

\begin{Cor}[Positivity of the coefficients of the truncated $(q,t)$-characters]
For $m \in \cM_{\cQ}$, write 
$$
L_{t}(m)_{\le \xi} =
\sum_{\text{$m^{\prime}:$ monomial in $\cY_{\le \xi}$}} a[m; m^{\prime}] \ul{m^{\prime}}
$$
with $a[m;m^{\prime}] \in \Z[t^{\pm 1/2}]$. Then we have
$
a[m;m^{\prime}] \in \Z_{\ge 0}[t^{\pm 1/2}].
$
\end{Cor}

The next corollary gives the affirmative answer to Conjecture~\ref{Conj:L=F} in $\Cc_{\cQ}$.   

\begin{Cor}[The $(q,t)$-characters of Kirillov-Reshetikhin modules]
For a Kirillov-Reshetikhin module $W^{(i)}_{k,p}=L(m^{(i)}_{k,p})$ belonging to the subcategory $\Cc_{\cQ}$ 
for some Q-datum $\cQ$ for $\fg$, we have
\begin{equation} \label{eq:KRqtch}
L_{t}(m^{(i)}_{k,p}) = F_{t}(m^{(i)}_{k, p}). 
\end{equation}
In particular, {\rm (\ref{eq:KRqtch})} holds true whenever $2kd_{i} \le rh^{\vee} $. 
\end{Cor}
\begin{proof}
Recall Lemma~\ref{Lem:si-so} for the last assertion.
\end{proof}

\section{$R$-matrices and commutation relations}
\label{sec:Rmat}

In this section, we first review some properties of the normalized $R$-matrices $R_{V, W}(z)$ for $V, W \in \Irr \Cc$
and the $\Z$-valued invariants $\fd(V,W)$ measuring the singularities of $R_{V, W}(z)$ introduced in \cite{KKOP20}. 
Then we check that the specialization of the HLO isomorphism $\Phi_{\cQ}$ for an arbitrary $\cQ$ coincides with 
the isomorphism between the Grothendieck rings induced from 
the generalized quantum affine Schur-Weyl duality functor $F_\cQ$ studied in \cite{KKK15, KKKOIV, KO17, OhS19, Fujita20, Naoi21}.    
From these results, we deduce several commutation relations in the quantum Grothendieck ring $\cK_t$, 
which will be important in the sequel.  

\subsection{$R$-matrices and related invariants}
\label{ssec:Rmat}

Let $z$ be an indeterminate.
For any $U_{q}(L\fg)$-module $V \in \Cc$, we can consider its \emph{affinization} $V_{z}$.
This is the $\kk[z^{\pm}]$-module $V_{z} \seq V \otimes_{\kk} \kk[z^{\pm1}]$ endowed with a structure of
left $U_{q}(L\fg)$-module by
$$
k_{i}(v \otimes \varphi) = (k_{i}v) \otimes \varphi, \
x^{\pm}_{i, k} (v \otimes \varphi) = (x^{\pm}_{i,k}v)\otimes z^{k} \varphi, \
h_{i,l}(v \otimes \varphi) = (h_{i,l}v) \otimes z^{l}\varphi, 
$$
for $v \in V$ and $\varphi \in \kk[z^{\pm 1}]$.
For $a \in \kk^{\times}$, we have
$$
T_a(V) \cong V_{z} / (z-a)V_{z}
$$
as $U_q(L\fg)$-modules (recall the definition of $T_a$ from Section~\ref{ssec:duality}).

Let $V, W \in \Cc$ be simple modules. We fix
their $\ell$-highest weight vectors
$v \in V, w \in W$.
It is known that there exists a unique isomorphism of $U_{q}(L\fg) \otimes_{\kk} \kk(z)$-modules
$$
R_{V, W}(z) \colon (V_z \otimes W) \otimes_{\kk[z^{\pm 1}]} \kk(z) \to (W \otimes V_z) \otimes_{\kk[z^{\pm 1}]} \kk(z) 
$$  
satisfying $R_{V, W}(z) (v_z \otimes w) = w \otimes v_z$, where $v_z \seq v \otimes 1 \in V_z$ (see~\cite[Corollary 2.5]{AK97}). 
We call this isomorphism $R_{V,W}(z)$ \emph{the normalized R-matrix} between $V$ and $W$.
Let $d_{V, W}(z) \in \kk[z]$ denote the monic polynomial of the smallest degree
such that the image of $d_{V, W}(z) R_{V, W}(z)$ is contained in 
$W \otimes V_z \subset (W \otimes V_z) \otimes_{\kk[z^{\pm 1}]} \kk(z) $.
This polynomial $d_{V,W}(z)$ is called  \emph{the denominator} of $R_{V,W}(z)$.
Note that $d_{V,W}(z)$ does not depend on the choice of $\ell$-highest weight vectors $v$ and $w$.
Moreover, we have 
\begin{equation} \label{eq:spd}
d_{T_a(V), T_b(W)}(z) = d_{V,W}((a/b)z) 
\end{equation}
for $a, b \in \kk^{\times}$.

For the case of fundamental modules, we simply write
$$d_{i,j}(z) \seq d_{L(Y_{i,0}), L(Y_{j,0})}(z)$$ 
for any $i,j \in I$.
It follows that $d_{i,j}(z) = d_{j,i}(z)$
 by \cite[(A.7)]{AK97} (see also \cite[Lemma 6.4]{FO21}).
The denominators $d_{i,j}(z)$ were computed in \cite{AK97, DO94, Fujita19, KKK15, Oh15R, OhS19}.\footnote{As we specified in the footnote in Section~\ref{ssec:KtCZ}, our coproduct of $U_q(L\fg)$ is opposite to the one in these references.}

\begin{Prop}[{\cite[Section~6.2]{FO21}}] \label{Prop:dfund}
Let $i,j \in I$ with $d_{i} \ge d_{j}$. Unless $d_{i}=d_{j} =1 <r$,  we have
$$
d_{i,j}(z) = d_{j,i}(z) = \prod_{u=0}^{rh^{\vee}}(z-q^{u+d_{i}})^{\tc_{ij}(u)}.
$$
Even if $d_{i}=d_{j} = 1<r$, the denominator $d_{i,j}(z) (= d_{j,i}(z))$ divides the RHS. 
\end{Prop}

\begin{Def}[cf.~{\cite[Section~3]{KKOP20}}] \label{Def:fd}
For each pair $(V, W)$
of simple modules in $\Cc$, we define the $\Z_{\ge 0}$-valued invariant $\fd(V, W)$ by  
$$
\fd(V, W) \seq \zero_{z=1}(d_{V, W}(z)) + \zero_{z=1}(d_{W,V}(z)),
$$
where $\zero_{z=a}(f(z))$ denotes the order of zeros of $f(z)$ at $z=a$
for $f(z) \in \kk[z]$ and $a \in \kk$. 
By definition, we have $\fd(V, W) = \fd(W,V)$.
\end{Def}

\begin{Rem}
In \cite{KKOP20}, an invariant $\fd(V,W)$ was defined for any \emph{rationally renormalizable} pair $(V, W)$ of non-zero modules of $\Cc$. 
When both $V$ and $W$ are simple, it coincides with the one in Definition~\ref{Def:fd}
thanks to \cite[Proposition 3.16]{KKOP20}.
\end{Rem}

A simple $U_{q}(L\fg)$-module $V \in \Cc$ is said to be \emph{real}
if its tensor square $V \otimes V$ is also simple.
For instance, every KR module is known to be real (see~e.g.~the proof of Proposition~\ref{Prop:clTsys} above).

\begin{Thm}[{\cite[Section 3.2]{KKKO15}}] \label{Thm:fd}
Let $V, W$ be two simple modules in $\Cc$ such that at least one of them is real.
Then the following three conditions are mutually equivalent:
\begin{itemize}
\item $\fd(V, W) = 0$.
\item $V$ and $W$ commute.
\item The tensor product $V \otimes W$ is simple.
\end{itemize}
In particular, we have $\fd(V, V)=0$ whenever $V$ is a real simple module. 
\end{Thm} 

\begin{Lem} \label{Lem:fd}
For $(i,p), (j,s) \in \hI$, we have
$$
\fd\left(L(Y_{i,p}), L(Y_{j,s})\right) \le \begin{cases}
\tc_{ij}\left(|p-s|-d_{i}\right) & \text{if $|p-s| \le rh^{\vee}$}, \\
0 & \text{otherwise}.
\end{cases}
$$ 
The equality holds unless $d_{i}=d_{j} =1 <r$.
\end{Lem}
\begin{proof}
When $d_{i} \ge d_{j}$, this is a direct consequence of Proposition~\ref{Prop:dfund}
along with (\ref{eq:spd}).
For the other case $d_{i} < d_{j}$, we use Lemma~\ref{Lem:tc}~(\ref{tc:tp}) \& (\ref{tc:po}). 
\end{proof}

\begin{Prop} \label{Prop:wtfd}
For any $m_{1}, m_{2} \in \cM$, we have
$$
(\wt_{\cQ}(m_{1}), \wt_{\cQ}(m_{2})) = \sum_{k \in \Z} (-1)^{k}\ \fd\left(L(m_{1}), \cD^{k+1}L(m_{2})\right).
$$
\end{Prop}
\begin{proof}
This is a combination of \cite[Proposition 3.22]{KKOP20} and \cite[Theorem 6.16]{FO21}.
\end{proof}

\subsection{Generalized quantum affine Schur-Weyl dualities}
\label{ssec:SW}

Let $\cJ = \{ V^{x}\}_{x \in J}$ be a collection 
of real simple modules in $\Cc$ indexed by an arbitrary set $J$.
We consider the graph $\Gamma_\cJ$
whose vertex set is $J$ and such that the number of edges between $x$ and $y$
is equal to $\fd(V^{x}, V^{y})$ for any $x, y \in J$.  
Note that the graph $\Gamma_\cJ$ has no edge loops by Theorem~\ref{Thm:fd}.

Given such a collection $\cJ$, 
Kang-Kashiwara-Kim~\cite{KKK18} constructed a monoidal functor 
$$
F_{\cJ} \colon \Aa(\Gamma_\cJ) \to \Cc
$$ 
called the \emph{generalized quantum affine Schur-Weyl duality functor},
with some good properties. 
Here $\Aa(\Gamma_\cJ)$ 
denotes the category of finite-dimensional modules
over the (completed) symmetric quiver Hecke algebras (defined over $\kk$) 
associated with the graph $\Gamma_\cJ$.    
This is a $\kk$-linear abelian monoidal category
and it gives a categorification of
the coordinate algebra of the maximal (pro)unipotent group 
for the symmetric Kac-Moody algebra associated with $\Gamma_{\cJ}$.
Moreover, the $\Z$-basis formed by the simple objects in $\Aa(\Gamma_\cJ)$ coincides with 
the specialization of the dual canonical basis.

Now we fix a Q-datum $\cQ$ for $\fg$.
As an important example of the above construction,
we consider the following collection of simple $U_q(L\fg)$-modules indexed by $J = \Delta_{0}$:
\begin{equation} \label{eq:JQ}
\cJ_{\cQ} \seq \{ L^{\cQ}(\alpha_{\im}) \}_{\im \in \Delta_{0}},
\end{equation}
where $L^{\cQ}(\alpha)$ for $\alpha \in \sR^{+}$ is the fundamental module defined in Section~\ref{ssec:CQ}.

\begin{Thm}[\cite{KKK15, KKKOIV, KO17, OhS19, Fujita20, Naoi21}] \label{Thm:SW}
Let $\cQ$ be a Q-datum for $\fg$. Then we have
\begin{equation} \label{eq:SW}
\fd(L^{\cQ}(\alpha_{\im}), L^{\cQ}(\alpha_{\jm})) = \delta(\im \sim \jm) 
\end{equation}
for any $\im, \jm \in \Delta_{0}$, and hence $\Gamma_{\cJ_{\cQ}} = \Delta$. 
The corresponding functor $F_{\cJ_\cQ}$ gives an equivalence of monoidal categories 
$
F_{\cQ} \colon \Aa(\Delta) \to \Cc_{\cQ}.
$ 
The induced $\Z$-algebra isomorphism
$$ [F_{\cQ}] \colon \evv(\cA_{v}[N_{-}]) \simeq K(\Aa(\Delta)) \to K(\Cc_{\cQ}),$$
is given by 
$$
\evv\left(\tF(\alpha, \cQ) \right) \mapsto \left[L^{\cQ}(\alpha)\right]
$$
for $\alpha \in \sR^{+}$, which maps the dual canonical basis $\evv(\tbfB)$ specialized at $v=1$ 
to the basis $\Irr \Cc_{\cQ}$ formed by the simple isomorphism classes in $\Cc_\cQ$.
\end{Thm}

Recall that $\evt(L_{t}(Y_{i,p})) = \chi_{q}(L(Y_{i,p}))$ holds for any $(i,p) \in \hI$ by Theorem~\ref{Thm:qtfund}~(\ref{qtfund:evt}).
Therefore, as a corollary of Theorem~\ref{Thm:HLO} and Theorem~\ref{Thm:SW}, we obtain the following.

\begin{Cor}[cf.~{\cite[Corollary 11.8]{HO19}}] \label{Cor:SW} The followings hold.
\begin{enumerate}
\item Let $\Phi_{\cQ}|_{v=t=1}$ denote the $\Z$-algebra isomorphism
$\evv(\cA_{v}[N_-]) \to \evt(\cK_{t, \cQ}) \simeq K(\Cc_{\cQ})$ induced from $\Phi_{\cQ}$.
Then we have $\Phi_{\cQ}|_{v=t=1} = [F_{\cQ}]$.
\item \label{evOK} $\chi_{q}(L(m)) = \evt(L_{t}(m))$ for all $m \in \cM_{\cQ}$.
\end{enumerate}
\end{Cor}
Corollary~\ref{Cor:SW}~(\ref{evOK}) gives the affirmative answer to Conjecture~\ref{Conj:KL} for $\Cc_{\cQ}$.

\subsection{Some commutation relations}
\label{ssec:comm}

In this subsection, we deduce several useful commutation relations in $\cK_t$
from the results we have obtained in the previous subsections. 

\begin{Lem}  \label{Lem:comm1}
Let $\cQ$ be a Q-datum for $\fg$ and $m_{1}, m_{2} \in \cM_{\cQ}$.
Assume that at least one of the simple modules $L(m_{1})$ and $L(m_{2})$ is real.
If $ \fd(L(m_{1}), L(m_{2})) = 0$, 
we have the $t$-commutation relation
$$
L_{t}(m_{1}) L_{t}(m_{2}) = t^{\Nn(m_{1}, m_{2})} L_{t}(m_{2}) L_{t}(m_{1})
$$
in the quantum Grothendieck ring $\cK_{t}(\Cc_\cQ)$.
\end{Lem}
\begin{proof} 
By Corollary~\ref{Cor:P2} we have 
$$
L_{t}(m_{1}) L_{t}(m_{2}) = \sum_{m \in \cM_{\cQ}} c^{m}_{m_{1}, m_{2}}(t) L_{t}(m).
$$ 
in $\cK_{t}$ with $c^{m}_{m_{1}, m_{2}}(t) \in \Z_{\ge 0}[t^{\pm 1/2}]$.
Applying $\evt$ to the both sides yields
$$
\chi_{q}(L(m_{1})) \chi_{q}(L(m_{2})) = \sum_{m \in \cM_{\cQ}} c^{m}_{m_{1}, m_{2}}(1) \chi_{q}(L(m))
$$ 
by Corollary~\ref{Cor:SW}~(\ref{evOK}). 
On the other hand, we have 
$L(m_{1}) \otimes L(m_{2}) \cong L(m_{1}m_{2})$
from our assumption and Theorem~\ref{Thm:fd}. 
Therefore we obtain $c^{m}_{m_{1}, m_{2}}(1) = \delta(m = m_{1}m_{2})$,
which implies $c^{m}_{m_{1}, m_{2}}(t) = \delta(m=m_{1}m_{2}) \cdot t^{a}$
for some $a \in \frac{1}{2}\Z$. 
Similarly, we can prove that $c^{m}_{m_{2}, m_{1}}(t) = \delta(m=m_{1}m_{2}) \cdot t^{b}$
for some $b \in \frac{1}{2}\Z$.
Thus the $(q,t)$-characters $L_{t}(m_{1})$ and $L_{t}(m_{2})$ commute up to
a power of $t^{\pm 1/2}$, i.e.~$L_{t}(m_{1}) L_{t}(m_{2}) = t^{c} L_{t}(m_{2}) L_{t}(m_{1})$
for some $c \in \frac{1}{2} \Z$. Comparing the leading terms, we find $c = \Nn(m_{1}, m_{2})$. 
\end{proof}

\begin{Lem} \label{Lem:tBoson}
Let $(i,p), (j,s) \in \hI$. We have the following relations in $\cK_{t}$:
\begin{enumerate}
\item \label{tBoson:Ftcomm} $L_{t}(Y_{i,p}) L_{t}(Y_{j,s}) = t^{\Nn(i,p;j,s)} L_{t}(Y_{j,s}) L_{t}(Y_{i,p})$ if \ $\fd(L(Y_{i,p}), L(Y_{j,s})) = 0$, 
\item \label{tBoson:Ftdual} $L_{t}(Y_{i,p}) L_{t}(Y_{i^{*}, p \pm rh^{\vee}}) = t^{\mp 1} L_{t}(Y_{i,p} Y_{i^{*}, p\pm rh^{\vee}})  + 1.$
\end{enumerate}
\end{Lem}
\begin{proof}
We may assume $p \le s$. 
First we consider the case $p+rh^{\vee} \le s$. 
In this case, Theorem~\ref{Thm:qtfund}~(\ref{qtfund:monomials}) implies that
the element
$L_{t}(Y_{i,p}) L_{t}(Y_{j,s}) - \delta((i^{*},p+rh^{\vee}) = (j,s))$
contains no dominant monomials other than $t^{\Nn(i,p; j,s)/2} \ul{Y_{i,p} Y_{j,s}}$.
Hence it is equal to $t^{\Nn(i,p;j,s)/2}L_{t}(Y_{i,p} Y_{j,s})$ by Theorem~\ref{Thm:qtch}.
This shows (\ref{tBoson:Ftcomm}) for $p+rh^{\vee} \le s$ and (\ref{tBoson:Ftdual}) as well
(note that $\Nn(i,p;i^* \pm rh^\vee) = \mp 2$ by Lemma~\ref{Lem:tc}).
For the other case $p\le s < p+rh^{\vee}$,
we choose a suitable Q-datum $\cQ$ for $\fg$ so that both
$Y_{i,p}$ and $Y_{j,s}$ belong to $\cM_{\cQ}$ (recall Lemma~\ref{Lem:si-so}). 
Under the assumption $\fd(L(Y_{i,p}), L(Y_{j,s})) = 0$,
we obtain the relation (\ref{tBoson:Ftcomm}) by applying Lemma~\ref{Lem:comm1}. 
\end{proof}

\begin{Lem} \label{Lem:comm2}
Let $\cQ = (\Delta, \sigma, \xi)$ be a Q-datum for $\fg$. 
Let $(\im, p), (\jm, s) \in \hDs_{0}$ be two elements satisfying the condition $\xi_{\im} \ge p \ge s > \xi_{\jm}$.  
Then we have the following equalities$\colon$ 
\begin{enumerate}
\item \label{comm2:fd} $\fd(L(Y_{i, p}), L(Y_{j, s})) = 0,$
\item \label{comm2:Nn} $(\wt_{\cQ}(Y_{i,p}),  \wt_{\cQ}(Y_{j,s})) = -\Nn(i, p ; j, s) = 0,$
\item \label{comm2:Ft} $F_{t}(Y_{i,p}) F_{t}(Y_{j,s}) = F_{t}(Y_{j,s})F_{t}(Y_{i,p}),$
\end{enumerate}
where $i \seq \bar{\im}, j \seq \bar{\jm}$.
\end{Lem}
\begin{proof}
Choose $a \in \Z_{\ge 0}$ and $b \in \Z_{>0}$ so that
$p=\xi_{\im} - 2ad_{i}$ and $s=\xi_{\jm} + 2bd_{j}$.
By Lemma~\ref{Lem:fd}, we have
$$
\fd(L(Y_{i, p}), L(Y_{j, s})) = \fd(L(Y_{j,s}), L(Y_{i, p})) 
\le \tc_{ji}(p-s-d_{j}) = \tc_{ji}(\xi_{\im}-\xi_{\jm}-d_{j} -2(ad_{i}+bd_{j})).
$$
The RHS is zero thanks to Lemma~\ref{Lem:tcxi}. This proves (\ref{comm2:fd}).

The first equality in (\ref{comm2:Nn}) is due to Proposition~\ref{Prop:Nnwt}. For the second equality, we have
\begin{align*}
\Nn(i,p;j,s) 
&= \tc_{ji}(p-s-d_{j}) - \tc_{ji}(p-s+d_{j}) \\
&= \tc_{ji}(\xi_{\im}-\xi_{\jm}-2ad_{i}-(2b+1)d_{j}) - \tc_{ji}(\xi_{\im}-\xi_{\jm}-2ad_{i}-(2b-1)d_{j})
\end{align*}
by (\ref{eq:Nnhalf}).
The RHS is zero again by Lemma~\ref{Lem:tcxi}.

The last equality (\ref{comm2:Ft}) is a consequence of 
(\ref{comm2:fd}), (\ref{comm2:Nn}) and Lemma~\ref{Lem:tBoson}~(\ref{tBoson:Ftcomm}).
\end{proof}

\begin{Lem}[see also~\cite{KKOP20c}] \label{Lem:comm3}
Let $\cQ$ be a Q-datum for $\fg$. 
\begin{enumerate} 
\item \label{comm3:k=1} 
For any simple roots $\alpha, \beta \in \sR^{+}$ with $\alpha \neq \beta$,
we have $$\fd\left(L^{\cQ}(\alpha), \cD L^{\cQ}(\beta)\right) = 0.$$
\item \label{comm3:k>1} 
For any positive roots $\alpha, \beta \in \sR^{+}$ and $k \in \Z_{\ge 2}$, we have
$$\fd\left(L^{\cQ}(\alpha), \cD^{k} L^{\cQ}(\beta)\right) = 0.$$
\end{enumerate}
\end{Lem}
\begin{proof} 
Let $\alpha, \beta \in \sR^{+}$ be distinct simple roots
and put $(\im, p) \seq \phi_{\cQ}^{-1}(\alpha, 0)$ and 
$(\jm, s) \seq \phi_{\cQ}^{-1}(\beta, 1)$. 
Then we have 
$L^{\cQ}(\alpha) = L(Y_{\bar{\im}, p})$ and 
$\cD L^{\cQ}(\beta) = L(Y_{\bar{\jm}, s})$ by Lemma~\ref{Lem:CQ}~(\ref{CQ:fund}). 
Note that we have $\xi_{\im} \ge p$ and $s > \xi_{\jm}$ by construction.
If $p \ge s$, Lemma~\ref{Lem:comm2}~(\ref{comm2:fd}) proves the assertion (\ref{comm3:k=1}).
If $p \le s-rh^{\vee}$, we have $|p-s| > rh^{\vee}$ and hence Lemma~\ref{Lem:fd} proves the assertion. 
Thus we only have to consider the case $s-rh^{\vee} < p < s$.
Recall that we have $\cD^{k} L(Y_{\bar{\jm}, s}) \cong L(Y_{(\bar{\jm})^{k*}, s+krh^{\vee}})$ 
for each $k \in \Z$ from Theorem~\ref{Thm:fD}.
Under the assumption $s-rh^{\vee} <p<s$, we have 
$\fd(L^{\cQ}(\alpha), \cD^{k+1} L^{\cQ}(\beta)) = 0$ unless $k \in \{-1, 0 \}$ by Lemma~\ref{Lem:fd}.
Combining this fact with Proposition~\ref{Prop:wtfd}, we find
$$
(\alpha, \beta) = \sum_{k \in \Z} (-1)^{k}\ \fd\left(L^{\cQ}(\alpha), \cD^{k+1} L^{\cQ}(\beta)\right) 
= -\fd\left(L^{\cQ}(\alpha), L^{\cQ}(\beta)\right) + \fd\left(L^{\cQ}(\alpha), \cD L^{\cQ}(\beta)\right).
$$
On the other hand, the equality (\ref{eq:SW}) tells that
$\fd(L^{\cQ}(\alpha), L^{\cQ}(\beta)) = -(\alpha, \beta)$ for two distinct simple roots $\alpha, \beta$. 
Therefore we have $\fd(L^{\cQ}(\alpha), \cD L^{\cQ}(\beta))=0$, which proves (\ref{comm3:k=1}).

Next we shall prove (\ref{comm3:k>1}).
Let $\alpha, \beta \in \sR^{+}$ be two positive roots and $k \in \Z_{\ge 2}$. 
Putting $(i,p) \seq \bphi_{\cQ}^{-1}(\alpha, 0), (j,s) \seq \bphi_{\cQ}^{-1}(\beta, k)$, 
we have
$L^{\cQ}(\alpha) = L(Y_{i, p})$ and 
$\cD^{k} L^{\cQ}(\beta) = L(Y_{j, s})$ by Lemma~\ref{Lem:CQ}~(\ref{CQ:fund}).  
Note that we always have $p<s$.
If $s-p>rh^{\vee}$, we have $\fd(L^{\cQ}(\alpha), \cD^{k}L^{\cQ}(\beta))=0$ by Lemma~\ref{Lem:fd} as desired.
In particular, it covers the case when $k>2$.
Thus we only have to consider the case $k=2$ with $s-rh^{\vee} \le p <s$.
In this case, we have
\begin{itemize}
\item $\fd(L^{\cQ}(\alpha), \cD^{l+1}L^{\cQ}(\beta))=0$ unless $l \in \{0, 1\}$ by Lemma~\ref{Lem:fd},
\item $\fd(L^{\cQ}(\alpha), \cD L^{\cQ}(\beta)) = \fd(L(Y_{i,p}), L(Y_{j^{*}, s-rh^{\vee}})) = 0$ by Lemma~\ref{Lem:comm2}~(\ref{comm2:fd}),
\item $(\alpha, \beta) = -\Nn(i,p; j^{*}, s-rh^{\vee})=0$ by Lemma~\ref{Lem:wtD} and Lemma~\ref{Lem:comm2}~(\ref{comm2:Nn}).
\end{itemize}
Combining these facts with Proposition~\ref{Prop:wtfd}, we have  $\fd(L^{\cQ}(\alpha), \cD^{2}L^{\cQ}(\beta)) = 0$.
This completes the proof of~(\ref{comm3:k>1}). 
\end{proof}

\section{Isomorphisms among quantum Grothendieck rings}
\label{sec:isom}

In this section, we prove the main theorem of this paper. 
In Section~\ref{ssec:facE}, we observe  that, for each Q-datum $\cQ$, every standard $(q,t)$-character $E_t(m)$ can be factorized into a product 
of smaller standard $(q,t)$-characters $E_t(m_k)$ belonging to $\cK_t(\cD^{k}\Cc_{\cQ})$ for $k \in \Z$.
In Section~\ref{ssec:pres}, we give a presentation of the localized quantum Grothendieck ring $\cK_{t} \otimes_{\Z[t^{\pm 1/2}]} \Q(t^{1/2})$. 
This is a generalization of the presentation obtained in \cite[Section 7]{HL15} for simply-laced cases. 
Based on this presentation, we construct in Section~\ref{ssec:Psi} a collection of isomorphisms $\Psi$ among the quantum Grothendieck rings
labeled by Q-data and prove that they respect the bases formed by the simple $(q,t)$-characters.
For a suitable choice of Q-data, the isomorphism $\Psi$ relates the quantum Grothendieck ring of simply-laced $\sg$ 
with that of non-simply-laced $\fg$, which propagates the known positivity for $\sg$ to the desired positivity for $\fg$.
Therefore we establish the positivities (\ref{p:KL}) and (\ref{p:sc}) for non-simply-laced cases. 
This gives an affirmative answer to a long-standing expectation in \cite{Hthese}.

\subsection{A factorization of standard $(q,t)$-character}
\label{ssec:facE}

Let $\cQ = (\Delta, \sigma, \xi)$ be a Q-datum for $\fg$. 
By Corollary~\ref{Cor:fAR}~(\ref{fAR:shift}), 
the commutative monoid $\cM$ is naturally isomorphic to the coproduct of 
the submonoids $\fD^{k}\cM_{\cQ}$ with $k$ running over $\Z$.
 In other words, each dominant monomial $m \in \cM$ 
has a unique factorization 
\begin{equation} \label{eq:Qfac}
m=\prod_{k \in \Z} \fD^{k}(m_{k})
\end{equation}
such that $m_{k} \in \cM_{\cQ}$ for all $k \in \Z$.
Here we have $m_{k} = 1$ for all but finitely many $k$.
Thanks to Lemma~\ref{Lem:comm2}~(\ref{comm2:fd}),
$L(Y_{i,p})$ and $L(Y_{j,s})$ commute 
when $Y_{i,p} \in \fD^{k}\cM_{\cQ}$ and $Y_{j,s} \in \fD^{l} \cM_{\cQ}$ with
$k<l$ and $p>s$.
Therefore, we have
\begin{equation} \label{eq:facM}
M(m) \cong \bigotimes^{\to}_{k \in \Z} \cD^{k}\left(M(m_{k})\right)
\end{equation}
as $U_q(L\fg)$-modules. 

Here we have
an analogous factorization result for its $(q,t)$-character $E_{t}(m)$.

\begin{Prop} \label{Prop:facE}
For each dominant weight $m \in \cM$, there holds 
\begin{equation} \label{eq:facE}
 E_{t}(m) = t^{\nu(m, \cQ)} \prod_{k \in \Z}^{\to} \fD_t^{k} \left(E_{t}(m_{k})\right)
\end{equation}
in the quantum Grothendieck ring $\cK_{t}$, where 
$$
\nu(m, \cQ) \seq  -\frac{1}{2} \sum_{k<l} (-1)^{k+l}(\wt_{\cQ}(m_{k}), \wt_{\cQ}(m_{l})).
$$
\end{Prop}
\begin{proof}
Thanks to Lemma~\ref{Lem:comm2}~(\ref{comm2:Ft}),
we know that $F_{t}(Y_{i,p})$ and $F_{t}(Y_{j,s})$ commute 
when $Y_{i,p} \in \fD^{k}\cM_{\cQ}$ and $Y_{j,s} \in \fD^{l} \cM_{\cQ}$ with
$k<l$ and $p>s$.
It implies that the desired equality (\ref{eq:facE}) holds
up to some power of $t^{\pm 1/2}$.
On the other hand, we have
\begin{equation} \label{eq:YYwt}
\tY_{i,p} \tY_{j,s} = t^{(\wt_{\cQ}(Y_{i,p}), \wt_{\cQ}(Y_{j,s}))} \tY_{j,s} \tY_{i,p}
\end{equation}
whenever $Y_{i,p} \in \fD^{k}\cM_{\cQ}$ and $Y_{j,s} \in \fD^{l} \cM_{\cQ}$ with $k<l$. 
This is a consequence of Proposition~\ref{Prop:Nnwt} and Lemma~\ref{Lem:comm2}~(\ref{comm2:Nn}).
It implies 
$$
\ul{m} = t^{\nu(m, \cQ)} \prod_{k \in \Z}^{\to} \fD^k_t(\ul{m_k}),
$$
which shows that the equality (\ref{eq:facE}) exactly holds.
\end{proof}

\begin{Rem}
If we choose a Q-datum $\cQ$ as in Lemma~\ref{Lem:si-so}, the above factorization results~(\ref{eq:facM}) and (\ref{eq:facE}) 
are obtained immediately from the definitions and Proposition~\ref{Prop:Nnwt}.  
\end{Rem}

\subsection{A presentation of quantum Grothendieck ring}
\label{ssec:pres}

In this subsection, we establish a presentation of the \emph{localized} quantum Grothendieck ring 
$$\cK_{\Q(t^{1/2})} \seq \cK_{t} \otimes_{\Z[t^{\pm 1/2}]} \Q(t^{1/2}).$$
We naturally regard $\cK_{t}$ as a $\Z[t^{\pm 1/2}]$-subalgebra of $\cK_{\Q(t^{1/2})}$. 
The anti-involution $\ol{(\cdot)}$ and the automorphism $\fD$ on $\cK_{t}$ are
naturally extended to $\cK_{\Q(t^{1/2})}$. 

Hereafter, we fix a Q-datum $\cQ$ for $\fg$. 
We put 
$$
x^{\cQ}_{\im, k} \seq L_{t}(Y_{i,p}) \qquad \text{with $(i,p) = \phi_{\cQ}^{-1}(\alpha_{\im}, k)$}
$$
for each $\im \in \Delta_{0}$ and $k \in \Z$. 
Note that we have 
$$x^{\cQ}_{\im,k} = \fD_{t}^{k} \left(L^{\cQ}_{t}(\alpha_\im)\right) = \fD_t^k(x^\cQ_{\im, 0})$$
by Lemma~\ref{Lem:fD} and Lemma~\ref{Lem:CQ}~(\ref{CQ:fund}). 

The HLO isomorphism $\Phi_{\cQ} \colon \cA_{v}[N_{-}] \to \cK_{t, \cQ}$ in Theorem~\ref{Thm:HLO}
induces the isomorphism
$$
\cU_{v}(\sn_-) \cong
\cA_{v}[N_{-}]\otimes_{\Z[v^{\pm 1/2}]}\Q(v^{1/2}) \xrightarrow{\sim}
\cK_{t, \cQ} \otimes_{\Z[t^{\pm 1/2}]} \Q(t^{1/2}) 
$$
which sends the dual Chevalley generator $\tF(\alpha_{\im}, \cQ) = (1-v^{2})f_{\im}$
to the element $x^{\cQ}_{\im, 0}$ for each $\im \in \Delta_{0}$.
Thus, we obtain \emph{the quantum Serre relations} satisfied for each $\im, \jm \in \Delta_{0}$ and $k \in \Z$ (recall~(\ref{eq:qSerref})):
\begin{equation} \label{eq:Serre}
\begin{cases}
x^{\cQ}_{\im, k} x^{\cQ}_{\jm, k} - x^{\cQ}_{\jm,k} x^{\cQ}_{\im, k} = 0 & \text{if $\im \not \sim \jm$}, \\
(x^{\cQ}_{\im, k})^{2}x^{\cQ}_{\jm,k} - (t+t^{-1}) x^{\cQ}_{\im,k} x^{\cQ}_{\jm,k} x^{\cQ}_{\im, k} + x^{\cQ}_{\jm,k} (x^{\cQ}_{\im,k})^{2}
= 0 & \text{if $\im \sim \jm$}.
\end{cases}
\end{equation}
In addition, we can prove \emph{the quantum Boson relations} hold
\begin{align}
\label{eq:tBoson1}
x^{\cQ}_{\im,k} x^{\cQ}_{\jm, k+1} 
&= t^{-(\alpha_{\im}, \alpha_{\jm})} x^{\cQ}_{\jm, k+1} x^{\cQ}_{\im,k} + (1-t^{-2})\delta_{\im, \jm}, \\
\label{eq:tBoson2}
x^{\cQ}_{\im, k} x^{\cQ}_{\jm, l} &= t^{(-1)^{k+l}(\alpha_{\im}, \alpha_{\jm})} x^{\cQ}_{\jm, l} x^{\cQ}_{\im, k}
\end{align}
for every $\im, \jm \in \Delta_{0}$ and $k, l \in \Z$ with $l > k+1$.
They are deduced from Lemma~\ref{Lem:tBoson}, Lemma~\ref{Lem:comm3} and the relation (\ref{eq:YYwt}). 
\begin{Def} \label{Def:Ag}
Let $\Delta$ be the Dynkin diagram of the simple Lie algebra $\sg$ of type $\mathrm{ADE}$ as before.
We consider the $\Q(t^{1/2})$-algebra $\tcA_{\Q(t^{1/2})}(\Delta)$ presented by the set of generators 
$$
\{ y_{\im, k} \mid \im \in \Delta_{0}, k \in \Z \}
$$
satisfying the following relations
\begin{enumerate} 
\renewcommand{\theenumi}{\rm R\arabic{enumi}}
\item \label{R1} for $\im, \jm \in \Delta_{0}$ and $k \in \Z$,
$$
\begin{cases}
y_{\im, k} y_{\jm, k} -y_{\jm, k} y_{\im, k} = 0 & \text{if $\im \not \sim \jm$}, \\
y_{\im,k}^{2} y_{\jm, k} - (t+t^{-1}) y_{\im, k} y_{\jm,k} y_{\im,k} + y_{\jm,k} y_{\im,k}^{2} =0 & \text{if $\im \sim \jm$},
\end{cases}
$$
\item \label{R2} for $\im, \jm \in \Delta_{0}$ and $k \in \Z$,
$$
y_{\im,k} y_{\jm, k+1} = t^{-(\alpha_{\im}. \alpha_{\jm})} y_{\jm, k+1}y_{\im,k} + \delta_{\im,\jm} (1-t^{-2}),
$$
\item \label{R3} for $\im, \jm \in \Delta_{0}$ and $k,l \in \Z$ with $l > k+1$,
$$
y_{\im,k} y_{\jm,l} = t^{(-1)^{k+l}(\alpha_{\im}, \alpha_{\jm})} y_{\jm,l} y_{\im,k}.
$$
\end{enumerate}
\end{Def}

\begin{Thm} \label{Thm:pres}
For each Q-datum $\cQ = (\Delta, \sigma, \xi)$ for $\fg$, there exists an isomorphism
$$
\Psi_{\cQ} \colon \tcA_{\Q(t^{1/2})}(\Delta) \to \cK_{\Q(t^{1/2})} 
$$
of $\Q(t^{1/2})$-algebras given by 
$y_{\im,k} \mapsto x^{\cQ}_{\im, k}$
for $\im \in \Delta_{0}$ and $k \in \Z$.
\end{Thm}
\begin{proof}
We have already checked the existence of the $\Q(t^{1/2})$-homomorphism 
$\Psi_{\cQ}$
given by the required assignment.
Indeed, the relation (\ref{eq:Serre}) (resp.~(\ref{eq:tBoson1}), (\ref{eq:tBoson2})) in $\cK_{\Q(t^{1/2})}$
corresponds to the defining relation $(\ref{R1})$ (resp.~(\ref{R2}), (\ref{R3})) in $\tcA_{\Q(t^{1/2})}(\Delta)$.  
It remains to show that $\Psi_{\cQ}$ is bijective. 

As an abuse of notation, we denote by $\fD_t$ the automorphism of the $\Q(t^{1/2})$-algebra $\tcA_{\Q(t^{1/2})}(\Delta)$
given by $\fD_t(y_{\im,k}) \seq y_{\im, k+1}$ for $\im \in \Delta_{0}, k \in \Z$.
From the definition, we have $\fD_t \circ \Psi_{\cQ} = \Psi_{\cQ} \circ \fD_t$.

Note that the $\Q(t^{1/2})$-subalgebra $\tcA_{\Q(t^{1/2})}^{(0)}$ of $\tcA_{\Q(t^{1/2})}(\Delta)$  
generated by $\{ y_{\im,0}\}_{\im \in \Delta_{0}}$ is naturally isomorphic to $\cA_{v}[N_{-}]\otimes_{\Z[v^{\pm 1/2}]}\Q(v^{1/2})$,
and the homomorphism $\Psi_{\cQ}$ restricts to (the localized version of) the HLO isomorphism $\Phi_{\cQ}$.
Therefore, for each $m \in \cM_{\cQ}$, there is a unique element $P_{t}(m) \in \tcA_{\Q(t^{1/2})}^{(0)}$ 
such that $\Psi_{\cQ}(P_{t}(m)) = E_{t}(m)$. 

Now let $m \in \cM$ be an arbitrary dominant monomial. Using the notation in (\ref{eq:Qfac}), 
we define an element $P_{t}(m) \in \tcA_{\Q(t^{1/2})}(\Delta)$ by
$$
P_{t}(m) \seq t^{\nu(m,\cQ)} \prod^{\to}_{k \in \Z} \fD_t^{k}\left(P_{t}(m_{k})\right).
$$ 
Under the homomorphism $\Psi_{\cQ}$, the set $\mathcal{P} \seq \{P_{t}(m)\mid m \in \cM\}$ 
bijectively corresponds to the basis $\{ E_{t}(m) \}_{m \in \cM}$ of $\cK_{\Q(t^{1/2})}$
by Proposition~\ref{Prop:facE}.
In particular, the set $\mathcal{P}$ is linearly independent over $\Q(t^{1/2})$. 

On the other hand, using the relations (\ref{R2}) and (\ref{R3}), we can rewrite 
every monomial $M$ in the elements of $\{y_{\im, k} \}_{\im \in \Delta_{0}, k \in \Z}$ 
as a linear combination of the elements of the form $M_{k_1} M_{k_2} \cdots M_{k_s}$ with 
$M_{k_u} \in \fD_t^{k_u} \tcA_{\Q(t^{1/2})}^{(0)}$ and $k_1 < k_2 < \cdots < k_s$.
Therefore, $\mathcal{P}$ forms a spanning set, and hence a basis of the $\Q(t^{1/2})$-vector space $\tcA_{\Q(t^{1/2})}(\Delta)$.
Since $\Psi_{\cQ}$ induces a bijection between the bases $\mathcal{P}$ and $\{ E_{t}(m) \}_{m \in \cM}$,
it is an isomorphism. 
\end{proof}

\subsection{Isomorphisms among quantum Grothendieck rings}
\label{ssec:Psi}

For $s=1,2$, let $\cQ^{(s)} = (\Delta^{(s)}, \sigma^{(s)}, \xi^{(s)})$ be a Q-datum for 
a complex simple Lie algebra $\fg^{(s)}$.
We assume that their underlying Dynkin diagrams are the same, i.e., $\Delta^{(1)} = \Delta^{(2)}$.
In what follows, we use the notation $X^{(s)}$ for a mathematical object $X$ 
to suggest that it is associated with $\fg^{(s)}$.
For example, we apply it to write $\Cc_{\Z}^{(s)}$, $\cK_{t}^{(s)}$, $\cM^{(s)}$ etc.

From the presentation of the quantum Grothendieck ring $\cK_{\Q(t^{1/2})}$ established in Theorem~\ref{Thm:pres},  
we obtain the following.

\begin{Cor}  \label{Cor:Psi}
There exists a unique $\Q(t^{1/2})$-algebra isomorphism
$$
\Psi \equiv \Psi(\cQ^{(2)}, \cQ^{(1)}) \colon \cK_{\Q(t^{1/2})}^{(1)} \to \cK_{\Q(t^{1/2})}^{(2)}
$$
satisfying the following properties$\colon$
\begin{enumerate}
\item The restriction of $\Psi$ to 
$\cK_{t, \cQ^{(1)}} \subset \cK_{\Q(t^{1/2})}^{(1)}$ coincides with
the isomorphism $$\Phi_{\cQ^{(2)}} \circ (\Phi_{\cQ^{(1)}})^{-1} \colon \cK_{t, \cQ^{(1)}} \to \cK_{t, \cQ^{(2)}}.$$ 
\item We have $\Psi \circ \ol{(\cdot)} = \ol{(\cdot)} \circ \Psi$ and $\Psi \circ \fD_t = \fD_t \circ \Psi$. 
\end{enumerate}
In particular, it satisfies
\begin{equation} \label{eq:Psix}
\Psi\left(x^{\cQ^{(1)}}_{\im, k}\right) = x^{\cQ^{(2)}}_{\im, k} \quad \text{for any $\im \in \Delta_{0}$ and $k \in \Z$.}
\end{equation}
\end{Cor}
\begin{proof}
We define $\Psi \seq \Psi_{\cQ^{(2)}} \circ (\Psi_{\cQ^{(1)}})^{-1}$, which satisfies the required properties.
\end{proof}

In what follows, we prove that  the above isomorphism $\Psi$ induces a bijection between the bases formed by the simple $(q,t)$-characters. 
First, we define the bijection $\psi_{0} \colon \cM_{\cQ^{(1)}} \to \cM_{\cQ^{(2)}}$ 
between the dominant monomials by the property
\begin{equation} \label{eq:psi0}
\Psi\left(L_{t}^{(1)}(m)\right) = L_{t}^{(2)}(\psi_{0}(m)) \qquad \text{for all $m \in \cM_{\cQ^{(1)}}$.}
\end{equation}
We extend this bijection $\psi_0$ 
to the bijection 
$\psi \colon \cM^{(1)} \to \cM^{(2)}$ by setting
$$
\psi(m) \seq \prod_{k \in \Z} \fD^{k}\psi_{0} (m_{k})
$$
for each $m \in \cM^{(1)}$, where we factorized $m$ as in (\ref{eq:Qfac}) with respect to the Q-datum $\cQ^{(1)}$.
One can easily check from the property~(\ref{eq:Psix}) that $\wt_{\cQ^{(1)}} = \wt_{\cQ^{(2)}} \circ \psi$ holds.

\begin{Thm} \label{Thm:Psi}
The isomorphism $\Psi \equiv \Psi(\cQ^{(2)}, \cQ^{(1)})$ in {\rm Corollary~\ref{Cor:Psi}} induces 
a bijection between the simple $(q,t)$-characters. 
In particular, it gives the isomorphism 
$$\Psi \colon \cK_{t}^{(1)} \to \cK_{t}^{(2)}$$
between the quantum Grothendieck ring of $\Cc_{\Z}^{(1)}$ and that of $\Cc_{\Z}^{(2)}$.
\end{Thm}
\begin{proof}
It follows from the property $\Psi \circ \ol{(\cdot)} = \ol{(\cdot)} \circ \Psi$ that the element 
$\Psi(L_t^{(1)}(m))$ is $\ol{(\cdot)}$-invariant for each $m \in \cM^{(1)}$. 
Hence, in view of  Theorem~\ref{Thm:qtch}, it suffices to show 
\begin{equation}\label{eq:goal}
\Psi\left(E_{t}^{(1)}(m)\right) - E_{t}^{(2)}(\psi(m)) \in \sum_{m^{\prime} \in \cM^{(2)}} t\Z[t]E_{t}^{(2)}(m^{\prime})
\end{equation} 
for each $m \in \cM^{(1)}$. By Proposition~\ref{Prop:facE}, we have
\begin{equation} \label{eq:PsiE}
\Psi\left(E_{t}^{(1)}(m)\right) = t^{\nu(\psi(m), \cQ^{(2)})} \prod_{k \in \Z}^{\to} \fD^{k}_t \left(\Psi\left(E_{t}^{(1)}(m_k)\right)\right).
\end{equation}
On the other hand, for any $m^\prime \in \cM_{\cQ^{(1)}}$, the property (S2) in Theorem~\ref{Thm:qtch} implies that
\begin{align*}
E_{t}^{(1)}(m^\prime) &= L_{t}^{(1)}(m') + \sum_{m^{\prime\prime} \in \cM_{\cQ^{(1)}}} a_{m^\prime, m^{\prime\prime}}(t) L_{t}^{(1)}(m^{\prime\prime}), \\ 
L_{t}^{(2)}(\psi_{0}(m^\prime)) &= E_{t}^{(2)}(\psi_{0}(m^\prime)) + \sum_{m^{\prime\prime} \in \cM_{\cQ^{(2)}}} b_{\psi_{0}(m^\prime), m^{\prime\prime}}(t) E_{t}^{(2)}(m^{\prime\prime})
\end{align*}
hold with some $a_{m^\prime, m^{\prime\prime}}(t), b_{\psi_{0}(m^\prime), m^{\prime\prime}}(t) \in t \Z[t]$.
Then, by the property (\ref{eq:psi0}), we have
\begin{equation} \label{eq:PsiEmk}
\Psi\left(E_{t}^{(1)}(m_k)\right) = E_{t}^{(2)}(\psi_{0}(m_k)) + 
\sum_{m^{\prime} \in \cM_{\cQ^{(2)}}} c_{m_k, m^{\prime}}(t) E_{t}^{(2)}(m^{\prime})  
\end{equation}
for each $k \in \Z$, where 
$$c_{m_k, m^{\prime}}(t) \seq 
a_{m_k, \psi_0^{-1}(m')}(t) + b_{\psi_0(m_k), m'}(t)+
\sum_{m^{\prime\prime} \in \cM_{\cQ^{(1)}}} a_{m_k, m^{\prime\prime}}(t)
b_{\psi_{0}(m^{\prime\prime}), m^{\prime}}(t) \in t\Z[t].$$
Note that $\wt_{\cQ^{(2)}}(m^{\prime}) = \wt_{\cQ^{(2)}}(\psi_{0}(m_k))$ holds whenever $E_{t}^{(2)}(m^{\prime})$
contributes non-trivially to the sum in (\ref{eq:PsiEmk}) by Proposition~\ref{Prop:Qwt}.  
Thus, the desired property (\ref{eq:goal}) follows
from (\ref{eq:PsiE}) and (\ref{eq:PsiEmk}) by using Proposition~\ref{Prop:facE} again. 
\end{proof}

As a consequence, we obtain the main result of this paper.

\begin{Cor}  \label{Cor:pos_st}
For any finite-dimensional complex simple Lie algebra $\fg$,
the quantum Grothendieck ring $\cK_{t}$ has 
non-negative structure constants with respect to 
the simple $(q,t)$-characters $\{ L_{t}(m) \}_{m \in \cM}$. 
More precisely, if we write
$$
L_{t}(m_1) L_{t}(m_2) = \sum_{m \in \cM} c^{m}_{m_1, m_2}(t) L_{t}(m)
$$
for $m_1, m_2 \in \cM$, we have 
$$
c^{m}_{m_1, m_2}(t) \in \Z_{\ge 0}[t^{\pm 1/2}].
$$
In particular, the analogs of Kazhdan-Lusztig polynomials are positive, i.e.,
$$P_{m, m^{\prime}}(t) \in t\Z_{\ge 0}[t]$$ 
for any $m,m' \in \cM$ with $m' < m$ $($cf.~$(\ref{eq:KL}))$.
\end{Cor}
\begin{proof}
When $\fg$ is simply-laced, the desired positivity has been established in \cite{Nakajima04} and \cite{VV03}
based on the geometry of quiver varieties (see also~\cite[Section 5.9]{HL15}). 
Now let $\fg$ be of non-simply-laced type and $(\Delta, \sigma)$ denote its unfolding ($\sigma \neq \id$).
We choose two arbitrary Q-data of the forms $\cQ^{(1)} = (\Delta, \id, \xi^{(1)})$ and 
$\cQ^{(2)} = (\Delta, \sigma, \xi^{(2)})$ so that we have  $\fg^{(1)} = \sg$ (simply-laced) and $\fg^{(2)} = \fg$.
Since the corresponding isomorphism $\Psi = \Psi(\cQ^{(2)}, \cQ^{(1)}) \colon \cK_{t}^{(1)} \to \cK_{t}^{(2)}$ respects 
the simple $(q,t)$-characters (Theorem~\ref{Thm:Psi}), it
transforms the known positivity in $\cK_{t}^{(1)}$ to the desired positivity in $\cK_{t}^{(2)}$.  
\end{proof}


\section{Further results in type B}
\label{sec:typeB}

In this section, we focus on the case when $\fg$ is of type $\mathrm{B}_n$ and hence $\sg$ is of type $\mathrm{A}_{2n-1}$.
We show that the specialization at $t=1$ of the isomorphism $\Psi$ for a specific choice of Q-data coincides 
with the Grothendieck ring isomorphism arising from a categorical relation studied in \cite{KKO19}.
As a consequence, we verify the analog of Kazhdan-Lusztig conjecture (Conjecture~\ref{Conj:KL})
and the desired positivity of the coefficients of simple $(q,t)$-characters for any simple $U_q(L\fg)$-modules in $\Cc_\Z$ of type $\mathrm{B}_n$. Hence, we have solved the long-standing problem of determining from their highest $\ell$-weight the character of simple finite-dimensional modules of type $\mathrm{B}$ quantum loop algebras, 
with a Kazhdan-Lusztig algorithm which is uniform relatively to simply-laced types (see the Introduction for more comments and references).

\subsection{A comparison between type A and B}
\label{ssec:AB}

Let $\Delta$ be the Dynkin diagram of type $\mathrm{A}_{2n-1}$ for a fixed $n \in \Z_{\ge 2}$.
We will use the labeling $\Delta_{0} = \{1,2,\ldots, 2n-1 \}$ as in Figure~\ref{Fig:unf}.
In this subsection, as a special case of the previous subsection, 
we consider the following two Q-data $\cQ^{(1)} = (\Delta, \id, \xi^{(1)})$ for $\fg^{(1)}$ of type $\mathrm{A}_{2n-1}$ 
and $\cQ^{(2)} = (\Delta, \vee, \xi^{(2)})$ for $\fg^{(2)}$ of type $\mathrm{B}_n$,
where
\begin{align*}
\xi^{(1)}_{\im} &\seq -\im, \\
\xi^{(2)}_{\im} &\seq \begin{cases}
-2\im & \text{if $1 \le \im < n$,} \\
1-2n & \text{if $\im = n$,} \\
2-2\im & \text{if $n < \im \le 2n-1$}
\end{cases} 
\end{align*}
for $\im \in \Delta_{0} = \{1,2,\ldots, 2n-1 \}$.
Note that we have 
$$
\hI_{\cQ^{(1)}} = (\Gamma_{\cQ^{(1)}})_0 = \{ (\im, -\im - 2k) \in \Delta_{0} \times \Z \mid 0 \le k \le 2n-1-\im \}
$$ 
by Proposition~\ref{Prop:tAR}.

\begin{Ex} When $n=3$, the twisted Auslander-Reiten quivers $\Gamma_{\cQ^{(1)}}$ 
and $\Gamma_{\cQ^{(2)}}$ are depicted respectively as follows:
$$
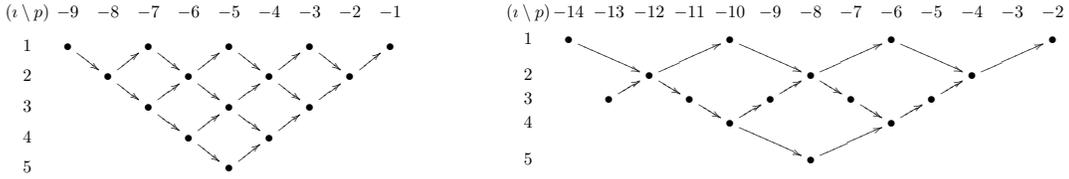
\raisebox{3mm}{
\scalebox{0.60}{\xymatrix@!C=0.5mm@R=2mm{
(\im\setminus p) & -9 & -8 & -7 &-6&-5 &-4& -3 &-2& -1  \\
1&\bullet  \ar@{->}[dr]&& \bullet  \ar@{->}[dr] &&\bullet \ar@{->}[dr]
&&\bullet  \ar@{->}[dr] && \bullet\\
2&& \bullet \ar@{->}[dr]\ar@{->}[ur]&&\bullet \ar@{->}[dr]\ar@{->}[ur]&&\bullet \ar@{->}[ur]\ar@{->}[dr] &&\bullet \ar@{->}[ur]&\\
3&&& \bullet \ar@{->}[dr] \ar@{->}[ur] &&\bullet\ar@{->}[dr] \ar@{->}[ur]
&&\bullet \ar@{->}[ur] && \\
4&&&& \bullet \ar@{->}[dr]\ar@{->}[ur] &&\bullet \ar@{->}[ur]
&&& \\ 
5&&&  &&\bullet\ar@{->}[ur]
&&&&
}} \hspace{1cm}
\scalebox{0.60}{\xymatrix@!C=0.5mm@R=0.5mm{
(\im\setminus p) & -14 & -13 & -12 &-11&-10 &-9& -8 &-7& -6 & -5& -4 & -3& -2 \\
1& \bullet  \ar@{->}[ddrr]&&&&\bullet \ar@{->}[ddrr] &&&& \bullet \ar@{->}[ddrr]
&&&& \bullet  \\ \\
2&&&\bullet\ar@{->}[uurr]\ar@{->}[dr]&&&&\bullet\ar@{->}[uurr]\ar@{->}[dr] &&&&  \bullet \ar@{->}[uurr] && \\ 
3&& \bullet \ar@{->}[ur]&& \bullet \ar@{->}[dr] &&\bullet \ar@{->}[ur] && \bullet \ar@{->}[dr] && \bullet \ar@{->}[ur]
&&&\\
4&&&&&\bullet\ar@{->}[ur] \ar@{->}[ddrr]
&&&& \bullet \ar@{->}[ur] &&&& \\ \\
5&&&&&&& \bullet \ar@{->}[uurr]&&&& &&}}}
$$
\end{Ex}

By Theorem~\ref{Thm:Psi}, we have the corresponding isomorphism 
$$
\Psi \equiv \Psi(\cQ^{(2)}, \cQ^{(1)}) \colon 
\cK_{t}^{(1)} \to \cK_{t}^{(2)}
$$
between the quantum Grothendieck ring of type $\mathrm{A}_{2n-1}$ 
and that of type $\mathrm{B}_{n}$. 
By the construction, the restriction of $\Psi$ to the subcategory $\Cc_{\cQ^{(1)}}$ 
is $\Phi_{\cQ^{(2)}} \circ (\Phi_{\cQ^{(1)}})^{-1}$, which has been computed in~\cite{HO19}.
Thus we have the following.

\begin{Thm}[{\cite[Theorem 12.9]{HO19}}] \label{Thm:PsiAB}
For any dominant monomial $m \in \cM^{(1)}$, we have
$$
\Psi\left(L_{t}^{(1)}(m)\right) = L_{t}^{(2)}(\psi(m)), 
$$
where $\psi \colon \cM^{(1)} \to \cM^{(2)}$ is a bijection 
satisfying $\fD \circ \psi = \psi \circ \fD$ and 
$$
\psi(Y_{\im, -\im - 2k}) = \begin{cases}
Y_{\bar{\im},-2\im-4k} & \text{if $\im \le n-1$ and $k \le n-\im-1$}, \\
Y_{n, -6n + 4 \im+1} & \text{if $\im \le n$ and $k = n - \im$}, \\
Y_{n, 2n +1-4\im -4k} Y_{n, -2n+3-4k} & \text{if $2 \le \im \le n$ and $n-\im+1 \le k \le n-1$}, \\
Y_{n, -2n+3-4\im} & \text{if $\im \le n-1$ and $k=n$}, \\
Y_{\bar{\im}, 4-2\im-4k} & \text{if $\im \le n-2$ and $n+1 \le k$}, \\
Y_{\bar{\im}, 2-2\im-4k} & \text{if $\im \ge n+1$}
\end{cases}
$$
for each $(\im, - \im -2k) \in \hI_{\cQ^{(1)}}$. 
\end{Thm}

On the other hand, based on the generalized Schur-Weyl duality,
Kashiwara-Kim-Oh~\cite{KKO19} have obtained the following result.

\begin{Thm}[{\cite[Corollary 3.2.2]{KKO19}}] \label{Thm:KKO}
There is an isomorphism $\tpsi \colon K(\Cc_{\Z}^{(1)}) \to K(\Cc_{\Z}^{(2)})$
such that we have
$$
\tpsi \left(\left[L^{(1)}(m)\right]\right) = \left[L^{(2)}(\psi(m))\right]
$$
for every $m \in \cM^{(1)}$, where $\psi$ is the bijection in {\rm Theorem~\ref{Thm:PsiAB}}.
\end{Thm} 
\begin{proof}
In~\cite[Section 3.2]{KKO19}, an isomorphism $\phi_{1} \colon K(\Cc_{\Z}^{(1)}) \to K(\Cc_{\Z}^{(2)})$
is defined. It induces an bijection $\Irr \Cc_{\Z}^{(1)} \to \Irr \Cc_{\Z}^{(2)}$ between 
isomorphism classes of simple modules.  
From the explicit computation of $\phi_{1}$ given in \cite[Corollary 3.2.2]{KKO19},
we see that it satisfies $[\cD] \circ \phi_{1} = \phi_{1} \circ [\cD]$
and its restriction to the subcategory $\Cc_{\cQ^{(1)}}$ 
coincides with the homomorphism $[L(Y_{\im, p})] \mapsto [L(\psi(Y_{\im, p}))]$ 
after a suitable parameter shift (see \cite[Remark 12.10]{HO19}).
Therefore the desired isomorphism $\tpsi$ is obtained as a parameter shift of $\phi_1$.  
\end{proof}

Recall that an algebra homomorphism $K(\Cc_{\Z}^{(1)}) \to K(\Cc_{\Z}^{(2)})$ 
is determined by the images of the classes $[L^{(1)}(Y_{i,p})]$ of fundamental modules 
(Proposition~\ref{Prop:KCZ}) and 
we know $\evt(L^{(s)}_{t}(Y_{i,p})) = \chi_{q}(L(Y_{i,p}))$ for 
all $(i,p) \in \hI^{(s)}$ and $s=1,2$.
Therefore, as a corollary of Theorems~\ref{Thm:PsiAB} and \ref{Thm:KKO},
we obtain the following. 

\begin{Cor} \label{Cor:PsiAB} The followings hold.
\begin{enumerate}
\item Let $\Psi|_{v=t=1}$ denote the $\Z$-algebra isomorphism
$K(\Cc_{\Z}^{(1)}) \to K(\Cc_{\Z}^{(2)})$ induced from $\Psi \equiv \Psi(\cQ^{(2)}, \cQ^{(1)})$ above.
Then we have $\Psi|_{v=t=1} = \tpsi$.
\item \label{KLtypeB} We have $\chi_{q}(L^{(2)}(m)) = \evt(L_{t}^{(2)}(m))$ for all $m \in \cM^{(2)}$.
\end{enumerate}
\end{Cor}
Corollary~\ref{Cor:PsiAB}~(\ref{KLtypeB}) gives the affirmative answer to Conjecture~\ref{Conj:KL} for 
every simple module $L(m)$ in $\Cc_\Z$ when $\fg$ is of type $\mathrm{B}_n$.

\subsection{Positivity of $(q,t)$-characters in type B}
In this subsection, we assume that $\fg$ is the simple Lie algebra of type $\mathrm{B}_n$. 
Now we obtain the following result 
by the same argument as Lemma~\ref{Lem:comm1}
using Corollaries~\ref{Cor:pos_st} and \ref{Cor:PsiAB}~(\ref{KLtypeB}).

\begin{Lem}
\label{Lem:commB}
Let $\fg$ be of type $\mathrm{B}_n$ and 
$m_{1}, m_{2} \in \cM$ any dominant monomials.
Assume that at least one of the simple modules $L(m_{1})$ and $L(m_{2})$ is real.
If we have $ \fd(L(m_{1}), L(m_{2})) = 0$, 
there holds the $t$-commutation relation
$$
L_t(m_1 m_2)=t^{-\Nn(m_1, m_2)/2}L_{t}(m_{1}) L_{t}(m_{2}) 
= t^{\Nn(m_{1}, m_{2})/2} L_{t}(m_{2}) L_{t}(m_{1})
$$
in the quantum Grothendieck ring $\cK_{t}$.
\end{Lem}

Using this lemma, we obtain the affirmative answer to Conjecture~\ref{Conj:L=F} in type $\mathrm{B}$ in the following way. 
For an interval $[a,b]$ we set
$L_{t}^{(\im)}[a,b] \seq L_{t}(m^{(\im)}[a,b])$, 
and similarly, we define $L_{t}^{(\im)}(a,b], L_{t}^{(\im)}[a,b)$ and $L_{t}^{(\im)}(a,b)$.  

\begin{Thm} \label{Thm:L=F}
Let $\fg$ be of type $\mathrm{B}_n$.
For any KR module $W^{(\im)}[p,s]$ in $\Cc_{\Z}$ with $(\im, p),  (\im, s) \in \hDs_0$ and $p < s$, 
its $(q,t)$-character $L_{t}^{(\im)}[p,s]$ 
coincides with $F_{t}^{(\im)}[p,s].$
\end{Thm}
\begin{proof}
Set $i \seq \bar{\im}$. 
We prove the assertion by induction on $s -p \in 2d_i\Z_{\ge 0}$.
When $s=p$, the KR module $W^{(\im)}[p,p]$ is the fundamental module $L(Y_{i,p})$
and hence $L_{t}(Y_{i,p}) =  F_{t}(Y_{i,p})$ by Theorem~\ref{Thm:qtfund}. 

Now we assume that $s-p >0$ and the equality 
$L_{t}^{(\jm)}[p^{\prime},s^{\prime}] = F_{t}^{(\jm)}[p^{\prime},s^{\prime}]$ holds
for any $\jm \in \Delta_0$ and $p^{\prime},s^{\prime} \in \Z$ with $s^{\prime}-p^{\prime} < s-p$. 
If we  write  
\begin{equation} \label{eq:LL} 
L_{t}^{(\im)}[p,s) L_{t}^{(\im)}(p,s] = 
\sum_{m \in \cM} c_{m}(t) L_{t}(m),
\end{equation}
we have
$c_{m}(t) \in \Z_{\ge 0}[t^{\pm 1/2}]$
by Corollary~\ref{Cor:pos_st}. 
Using Corollary~\ref{Cor:PsiAB}~(\ref{KLtypeB}) and the classical $T$-system identity (\ref{eq:Tsys}),
we specialize (\ref{eq:LL}) at $t=1$ to find 
\begin{equation} \label{eq:cm}
c_m(1) = \begin{cases}
1 & \text{if $m = m^{(\im)}[p,s] m^{(\im)}(p,s)$ or $M(\im; p,s)$}, \\
0 & \text{otherwise},
\end{cases}
\end{equation}
where $M(\im; p,s) = \prod_{\jm \sim \im} m^{(\jm)}(p,s)$ as in Theorem~\ref{Thm:qTsys}.
Note  that $c_m(1) = 0$ implies that $c_m(t) = 0$,
and $c_m(1) = 1$ implies that $c_m(t) = t^{x}$ for some $x \in \frac{1}{2}\Z$.
Recall the set 
$
\cM^+(\im; p, s)
$
of dominant monomials defined in (\ref{eq:cM+}).
Since the members of the family $\{ L(m) \mid m \in \cM^+(\im; p,s) \}$ mutually commute 
(see Proposition~\ref{Prop:clTsys}),
the members of the family $\{ L_t(m) \mid m \in \cM^+(\im; p,s)\}$ mutually $t$-commute 
and we have 
$$
L_t\left(m^{(\im)}[p,s] m^{(\im)}(p,s)\right) = t^{x}L_{t}^{(\im)}[p,s] L_{t}^{(\im)}(p,s),
\qquad 
L_t(M(\im;p,s)) = t^y \prod_{\jm \sim \im}^{\to} L_{t}^{(\jm)}(p,s)
$$
for some $x, y \in \frac{1}{2} \Z$
by Lemma~\ref{Lem:commB}.
Here we fix an arbitrary total ordering of the set $\{ \jm \in \Delta_0 \mid \jm \sim \im \}$ 
to define the ordered product.
Combining with (\ref{eq:cm}), we obtain
\begin{equation} \label{eq:qTsysL}
L_t^{(\im)}[p,s) L_t^{(\im)}(p,s] = t^a L_t^{(\im)}[p,s] L_t^{(\im)}(p,s) 
+ t^b \prod_{\jm \sim \im}^{\to} L_{t}^{(\jm)}(p,s),
\end{equation} 
where the exponents $a, b \in \frac{1}{2} \Z$ are determined uniquely from the fact 
that $L_t^{(\im)}[p,s)$ is the $\ol{(\cdot)}$-invariant element of the form
$$
\left( 
t^a L_t^{(\im)}[p,s] L_t^{(\im)}(p,s) 
+ t^b \prod_{\jm \sim \im}^{\to} L_{t}^{(\jm)}(p,s)
\right)
L_t^{(\im)}(p,s]^{-1}
$$
in the fraction field of the quantum torus $\cY_t$.
In other words, we have obtained a quantum $T$-system identity satisfied by 
the $(q,t)$-characters of KR modules.

On the other hands, we also know the quantum $T$-system identity (\ref{eq:qTsys})
satisfied by the corresponding $F_t(m)$'s (Theorem~\ref{Thm:qTsys}).
The exponents $a, b \in \frac{1}{2}\Z$ appearing there are exactly the same as in (\ref{eq:qTsysL}) 
because the family $\{ F_t(m) \mid m \in \cM^+(\im; s,p)\}$ satisfies 
the same $t$-commutation relations as the family $\{ L_t(m) \mid m \in \cM^+(\im; p,s)\}$ does (see Proposition~\ref{Prop:qTsys}).

Now, applying the induction hypothesis to (\ref{eq:qTsysL}), we get
$$
F_t^{(\im)}[p,s) F_t^{(\im)}(p,s] = t^a L_t^{(\im)}[p,s] F_t^{(\im)}(p,s) 
+ t^b \prod_{\jm \sim \im}^{\to} F_{t}^{(\jm)}(p,s)
$$
with the same $a,b \in \frac{1}{2} \Z$ as above. Comparing this equality with (\ref{eq:qTsys}),
we obtain the desired conclusion $L_t^{(\im)}[s,p] = F_t^{(\im)}[s,p]$.
\end{proof}

\begin{Thm}[Positivity of coefficients of simple $(q,t)$-characters in type B]
\label{Thm:posqtch}
Let $\fg$ be of type $\mathrm{B}_n$. For each dominant monomial $m \in \cM$, write
$$
L_t(m) = \sum_{\text{$m^\prime:$ monomial in $\cY$}} a[m; m^\prime] \ul{m^\prime}
$$
with $a[m;m^\prime] \in \Z[t^{\pm 1/2}]$. Then we have
$$
a[m;m^\prime] \in \Z_{\ge 0}[t^{\pm 1/2}].
$$
\end{Thm}
\begin{proof}
Given $m \in \cM$, we choose integers $c < b$
so that the $(q,t)$-character $L_t(m)$ 
is a (non-commutative) Laurent polynomial
in the variables $Y_{i, p}$ with $i \in I$ and $c \le p \le b$.  
Set $\hDs_0 [c,b] \seq \hDs_0 \cap (\Delta_0 \times [c,b])$.
By Lemma~\ref{Lem:qKR} and Theorem~\ref{Thm:L=F}, we have
$$
L_t^{(\im)}[p,b]_{\le b} = F_t^{(\im)}[p,b]_{\le b} = \ul{m^{(\im)}[p,b]}
$$
for each $(\im,p) \in \hDs_0[c,b]$.
Thus, choosing an arbitrary total ordering of the finite set $\hDs_0[c,b]$, we have
$$
L_t(m)_{\le b} = L_t(m) = \sum_{m^\prime} t^{x(m^\prime)} a[m;m^\prime] 
\prod^{\to}_{(\im,p) \in \hDs_0[c,b]} \left(L_t^{(\im)}[p,b]_{\le b}\right)^{v_{\im,p}(m^\prime)}
$$
for some $x(m^\prime) \in \frac{1}{2} \Z$, where
$v_{\im,p}(m^\prime) \seq u_{\bar{\im},p}(m^\prime) - u_{\bar{\im}, p+2d_{\bar{\im}}}(m^\prime) \in \Z$. 
Let $l \in \Z_{> 0}$ large 
enough so that $v_{\im,p}(m^\prime) +l \ge 0$ holds for all $(\im, p) \in \hDs_0 [a,b]$ and 
for all monomials $m^\prime$ with  
$a[m;m^\prime] \neq 0$, and put 
$
\cL \seq \overset{\to}{\prod}_{(\im,p) \in \hDs_0[c,b]} (L_t^{(\im)}[p,b])^{l}. 
$
Taking the injectivity of $(\cdot)_{\le b}$ into account, we obtain the following equality in $\cK_t$:
$$
L_{t}(m) \cdot \cL =
 \sum_{m^\prime} t^{x^\prime(m^\prime)} a[m;m^\prime] L_t \left(
\prod_{(\im,p) \in \hDs_0[c,b]} m^{(\im)}[p,b]^{v_{i,p}(m^\prime)+l} \right)
$$
for some $x^\prime(m^\prime) \in \frac{1}{2} \Z$. 
By the positivity of structure constants of $\cK_t$ (Corollary~\ref{Cor:pos_st}), 
we obtain $t^{x^\prime(m^\prime)} a[m;m^\prime] \in \Z_{\ge 0}[t^{\pm 1/2}]$ and hence
$a[m;m^\prime] \in \Z_{\ge 0}[t^{\pm 1/2}]$. 
\end{proof}

For general $\fg$, once one proves Conjecture~\ref{Conj:L=F}, the positivity of any simple $(q,t)$-characters (the analog of Theorem~\ref{Thm:posqtch}) follows by the same argument as above.
We plan to pursue this direction in a future work. 

\subsection{Thin representations}
A finite-dimensional $U_q(L\fg)$-module $V$ is said to be \emph{thin} if all $\ell$-weight spaces are of dimension $1$. As a corollary of Corollary \ref{Cor:PsiAB} (\ref{KLtypeB}) and Theorem \ref{Thm:posqtch}, we can conclude that the $(q, t)$-characters of simple thin modules are the canonical lifts of their $q$-characters: 
\begin{Cor} \label{Cor:thin}
Let $\fg$ be of type $\mathrm{B}_n$, and $m \in \cM$ be a dominant monomial such that $L(m)$ is a thin $U_q(L\fg)$-module. Let $P(m)$ be a subset set of $\cY$ consisting of Laurent monomials satisfying
\[
\chi_{q}(L(m))=\sum_{m'\in P(m)}m'. 
\]
Then we have 
\[
L_t(m) = \sum_{m'\in P(m)} \ul{m^\prime}.
\]
\end{Cor}
\begin{proof}
By Theorem \ref{Thm:posqtch}, we have 
\[
L_t(m) = \sum_{\text{$m':$ monomial in $\cY$}} a[m; m'] (t) \ul{m'}
\]
for some $a[m; m'](t) \in \Z_{\ge 0}[t^{\pm 1/2}]$. Moreover, by  Corollary \ref{Cor:PsiAB} (\ref{KLtypeB}), we have 
\[
\sum_{\text{$m':$ monomial in $\cY$}} a[m; m^\prime] (1) m' =\sum_{m'\in P(m)}m'
\]
Therefore, 
\[
a[m; m'] (t)=\begin{cases}
t^{c(m; m')},\text{ for some }c(m; m')\in \frac{1}{2} \Z&\text{ if }m'\in P(m),\\
0&\text{otherwise.}
\end{cases}
\]
Moreover, since $L_t(m)$ is $\ol{(\cdot)}$-invariant, we have $c(m; m')=0$ for all $m'\in P(m)$, which completes the proof. 
\end{proof}
The minimal affinizations were defined by Chari \cite{Chari95} as generalizations of type $A$ evaluation representations and of Kirillov-Reshetikhin modules.
By \cite[Theorem 3.10]{Hernandez07}, all minimal affinizations 
are thin when $\fg$ is of type $\mathrm{B}_n$. Therefore Corollary \ref{Cor:thin} is a vast generalization of \cite[Corollary 11.5]{HO19}.
\appendix
\section{Quantum unipotent minors}
\label{sec:minor}

In this section, we briefly review the definitions of the quantum coordinate algebra $\cA_{v}[N_-]$
and the normalized quantum unipotent minors associated with a simply-laced Lie algebra $\sg$.
We will use the notation in Section~\ref{ssec:notr}. 


\subsection{Quantum coordinate algebra $\cA_{v}[N_-]$}
\label{ssec:cA}
Let $v$ be an indeterminate with a formal square root $v^{1/2}$.
The quantized enveloping algebra $\cU_{v}(\sg)$ of $\sg$ (over $\Q(v^{1/2})$)
is the unital associative $\Q(v^{1/2})$-algebra
presented by the set of generators $$\{ e_{\im}, f_{\im}, K_{\im}^{\pm 1} \mid \im \in \Delta_{0}\}$$
and the following relations:
\begin{itemize}
\item $K_{\im} K_{\im}^{-1} = K_{\im}^{-1} K_{\im} = 1$ and $K_{\im} K_{\jm} = K_{\jm} K_{\im}$ for $\im, \jm \in \Delta_{0}$,
\item $K_{\im}e_{\jm} = v^{(\alpha_{\im}, \alpha_{\jm})}e_{\jm}K_{\im}, 
K_{\im}f_{\jm} = v^{-(\alpha_{\im}, \alpha_{\jm})}f_{\jm}K_{\im}$ for $\im, \jm \in \Delta_{0}$,
\item $(v-v^{-1})[e_{\im},f_{\jm}] = \delta_{\im, \jm} (K_{\im} - K_{\im}^{-1})$ for $\im, \jm \in \Delta_{0}$,
\item  
$e_{\im}^{2} e_{\jm} - (v + v^{-1})e_{\im} e_{\jm} e_{\im} - e_{\jm} e_{\im}^{2} = 
f_{\im}^{2} f_{\jm} - (v + v^{-1})f_{\im} f_{\jm} f_{\im} - f_{\jm} f_{\im}^{2} = 0$ 
if $\im \sim \jm$,
\item
$[e_{\im}, e_{\jm}] = [f_{\im}, f_{\jm}] = 0$  
if $\im \not \sim \jm$. \end{itemize}

For each generator, we set
$$
\wt(e_{\im}) = \alpha_{\im}, \quad \wt(f_{\im}) = - \alpha_{\im}, \quad 
\wt(K_{\im}^{\pm 1}) = 0,
$$ 
which endows $\cU_{v}(\sg)$ with a structure of $\sQ$-graded $\Q(v^{1/2})$-algebra.

We denote by $\cU_{v}(\sn_{-})$
the $\Q(v^{1/2})$-subalgebra of $\cU_{v}(\sg)$ generated by $\{ f_{\im} \}_{\im \in \Delta_{0}}$.
This is regarded as a quantized enveloping algebra 
of the nilpotent Lie subalgebra $\sn_{-} \subset \sg$
corresponding to the negative roots. 
As a $\Q(v^{1/2})$-algebra, $\cU_{v}(\sn_{-})$ is presented by the set of generators $\{ f_{\im} \}_{\im \in \Delta_{0}}$
and {the quantum Serre relations} (\ref{eq:qSerref}).

For $\im \in \Delta_{0}$, 
we define the $\Q(v^{1/2})$-linear maps
$e_{\im}^{\prime}$ and
${}_{\im}e^{\prime} \colon \cU_{v}(\sn_{-}) \to \cU_{v}(\sn_{-})$ by
\begin{align*}
e_{\im}^{\prime}(xy) &= e_{\im}^{\prime}(x)y + v^{(\alpha_{\im}, \wt (x))}x e_{\im}^{\prime}(y), 
& e_{\im}^{\prime}(f_{\jm}) &= \delta_{\im, \jm}, \\
{}_{\im}e^{\prime}(xy) &=  v^{(\alpha_{\im}, \wt (y))} {}_{\im}e^{\prime}(x)y +x \, {}_{\im}e^{\prime}(y), 
& {}_{\im}e^{\prime}(f_{\jm}) &= \delta_{\im, \jm} 
\end{align*} 
for homogeneous elements $x, y \in \cU_{v}(\sn_{-})$. There exists a unique symmetric $\Q(v^{1/2})$-bilinear form
$(\cdot, \cdot)_{L}$ on $\cU_{v}(\sn_-)$ such that
$$
(1,1)_{L} = 1, \qquad
(f_{\im}x,y)_{L} =\frac{1}{1-v^{2}}(x, e_{\im}^{\prime}(y))_{L}, \qquad
(xf_{\im}, y)_{L} = \frac{1}{1-v^{2}}(x, {}_{\im}e^{\prime}(y))_{L}.
$$ 
Then $(\cdot, \cdot)_{L}$ is non-degenerate.

Let $\cU_{v}(\sn_-)_{\Z}$ be the $\Z[v^{\pm 1/2}]$-subalgebra of $\cU_{v}(\sn_-)$ generated by the divided powers 
$\{ f_{\im}^{(m)} \seq f_{\im}^{m}/[m]_{v}! \mid \im \in \Delta_{0}, m \in \Z_{\ge 0} \}$.
Then \emph{the quantized coordinate algebra} $\cA_{v}[N_-]$ is defined as 
$$
\cA_{v}[N_-] \seq \{ x \in \cU_{v}(\sn_-) \mid (x, \cU_{v}(\sn_-)_{\Z})_L \subset \Z[v^{\pm 1/2}]\}.
$$

\subsection{Quantum unipotent minors} 
\label{ssec:minor}

Let $V$ be a $\cU_{v}(\sg)$-module. For $\mu \in \sP$, we set
$$
V_{\mu}\seq\{u \in V \mid \text{$K_{\im} u = v^{(\alpha_{\im}, \mu)}u$ for all $\im \in \Delta_{0}$}\}.
$$
This is called \emph{the weight space of $V$ of weight $\mu$}.
A $\cU_{v}(\sg)$-module with weight space decomposition $V=\bigoplus_{\mu \in \sP}V_{\mu}$
is said to be \emph{integrable} if $e_{\im}$ and $f_{\im}$ act locally nilpotently on $V$ for all $\im \in \Delta_{0}$.

For each $\lambda \in \sP^{+} \seq \sum_{\im \in \Delta_{0}} \Z_{\ge 0} \varpi_{\im}$,
we denote by $V(\lambda)$ the (finite-dimensional) irreducible highest weight
$\cU_{v}(\sg)$-module generated by a highest weight vector $u_{\lambda} \in V(\lambda)_{\lambda}$.
The module $V(\lambda)$ is integrable.
There exists a unique $\Q(v^{1/2})$-bilinear form 
$(\cdot, \cdot)^{\varphi}_{\lambda} \colon V(\lambda) \times V(\lambda) \to \Q(v^{1/2})$
such that
$$
(u_{\lambda}, u_{\lambda})^{\varphi}_{\lambda} = 1, \qquad 
(x u_{1}, u_{2})^{\varphi}_{\lambda} = (u_{1}, \varphi(x)  u_{2})^{\varphi}_{\lambda}
$$
for $u_{1}, u_{2} \in V(\lambda)$ and $x \in \cU_{v}(\sg)$,
where $\varphi$ is the $\Q(v^{1/2})$-algebra anti-involution on $\cU_{v}(\sg)$ defined by
$$
\varphi(e_{\im}) = f_{\im}, \qquad
\varphi(f_{\im}) = e_{\im}, \qquad
\varphi(K_{\im}) = K_{\im}
$$    
for $\im \in \Delta_{0}$.
The form $(\cdot, \cdot)^{\varphi}_{\lambda}$ is non-degenerate and symmetric.

For $w \in \sW$, we define the element $u_{w\lambda} \in V(\lambda)_{w \lambda}$ by
$$
u_{w \lambda} = f_{\im_1}^{(m_1)} \cdots f_{\im_{l-1}}^{(m_{l-1})}f_{\im_{l}}^{(m_{l})} \cdot u_{\lambda}, \qquad
$$ 
for $(\im_{1}, \ldots, \im_{l}) \in \cI(w)$, where
$m_{k} \seq (\alpha_{\im_{k}}, s_{\im_{k+1}}\dots s_{\im_{l-1}}s_{\im_{l}} \lambda) \in \Z_{\ge 0}$.
It is known that this element does not depend on the choice of 
$(\im_{1}, \ldots, \im_{l}) \in \cI(w)$ and 
$w \in \sW$ (depends only on $w\lambda$).
See, for example, \cite[Proposition 39.3.7]{LusztigIntro}.
Then $(u_{w\lambda}, u_{w\lambda})^{\varphi}_{\lambda} = 1$. 

For $\lambda \in \sP^{+}$ and $w, w^{\prime} \in \sW$, we define an element 
$D_{w\lambda, w^{\prime}\lambda} \in \cU_{v}(\sn_{-})$ by the following property:
$$
(D_{w\lambda, w^{\prime}\lambda}, x)_{L} = (u_{w\lambda}, x u_{w^{\prime}\lambda})^{\varphi}_{\lambda}
$$
for $x \in \cU_{v}(\sn_{-})$.
By the nondegeneracy of the bilinear form $(\cdot, \cdot)_{L}$,
this element is uniquely determined. 
An element of this form is called a \emph{quantum unipotent minor}.
Moreover, we set
$$
\tD_{w\lambda, w^{\prime}\lambda} \seq v^{-(w\lambda - w^{\prime}\lambda, w\lambda - w^{\prime}\lambda)/4
+ (w\lambda-w^{\prime}\lambda, \rho)/2} D_{w\lambda, w^{\prime}\lambda},
$$  
where 
$\rho \seq \sum_{\im \in \Delta_{0}} \varpi_{\im} \in \sP^{+}$.
This element is called \emph{a normalized quantum unipotent minor}. 

\section{Quantum cluster algebras}
\label{sec:QCA}

In this subsection, we briefly review the definition of a quantum cluster algebra following~\cite{BZ05}. 

\subsection{Quantum seed}
\label{ssec:Qseed}

Let $v$ be an indeterminate with its formal square root $v^{1/2}$.
Let $J$ be a finite set. For a $\Z$-valued $J\times J$-skew-symmetric matrix 
$\Lambda = (\Lambda_{ij})_{i,j \in J}$, we define \emph{the quantum torus} 
$\cT(\Lambda)$ as the $\Z[v^{\pm 1/2}]$-algebra presented by the set of generators 
$\{ X_{j}^{\pm 1} \mid j \in J\}$ and the relations:
\begin{itemize}
\item $X_j X_j^{-1} = X_j^{-1} X_j  =1$ for $j \in J$, 
\item $X_{i} X_{j} = v^{\Lambda_{ij}} X_j X_i$ for $i,j\in J$.
\end{itemize} 
For $\bfa = (a_j)_{j \in J} \in \Z^J$, we write 
$$
X^{\bfa} \seq v^{-\frac{1}{2}\sum_{i < j}a_i a_j \Lambda_{ij}} \prod^{\to}_{j \in J} X_j^{a_j},
$$
where we fixed an arbitrary total ordering $<$ of the set $J$. Note that the resulting element $X^{\bfa} \in \cT(\Lambda)$
is independent from the choice of the total ordering $<$.
Since $\cT(\Lambda)$ is an Ore domain, it is embedded in its skew field of fractions $\F(\cT(\Lambda))$.

Let $J_f \subset J$ be a subset and set $J_e \seq J \setminus J_e$.
Let $\tB = (b_{ij})_{i \in J, j \in J_e}$ be a $\Z$-valued $J \times J_e$-matrix whose 
\emph{principal part} $(b_{ij})_{i,j \in J_e}$ is skew-symmetric. 
Such a matrix $\tB$ is called \emph{an exchange matrix}. 

We say that a pair $(\Lambda, \tB)$ is \emph{compatible} if we have 
$$
\sum_{k \in J} b_{ki}\Lambda_{kj} = d \delta_{i,j} \quad \text{$(i\in J_e, j \in J)$}
$$
for some positive integer $d$.

If $(\Lambda, \tB)$ is a compatible pair, the datum $\Sigma = ((X_j)_{j \in J}, \Lambda, \tB)$
is called \emph{a quantum seed.}
Then the set $\{ X_j \}_{j \in J} \subset \F(\cT(\Lambda))$ is called \emph{the quantum cluster of} $\Sigma$
and each element $X_{j}$ is called \emph{a quantum cluster variable}.

\subsection{Mutations}
\label{ssec:mut}

Given a quantum seed $\Sigma = ((X_j)_{j \in J}, \Lambda, \tB)$
and an element $k \in J_e$, we can associate a new quantum seed 
$\mu_{k} (\Sigma)$ as follows. 

Define the $J \times J$-matrix $E=(e_{ij})_{i,j \in J}$
and the $J_e \times J_e$-matrix $F=(f_{ij})_{i,j \in J_e}$ by 
$$
e_{ij} \seq \begin{cases}
\delta_{i,j} & \text{if $j \neq k$}, \\
-1 & \text{if $i=j=k$}, \\
\max(0, -b_{ik}) & \text{if $i\neq j = k$},
\end{cases} 
\qquad 
f_{ij}\seq \begin{cases}
\delta_{i,j} & \text{if $i \neq k$}, \\
-1 & \text{if $i=j=k$}, \\
\max(0, -b_{ik}) & \text{if $i = k \neq j$}.
\end{cases}
$$
In addition, define $\bfa^\prime = (a^\prime_{j})_{j \in J}$
 and $\bfa^\prime = (a^{\prime\prime}_{j})_{j \in J}$ by
$$
a^\prime_{j} \seq \begin{cases}
-1 & \text{if $i=k$}, \\
\max(0, b_{ik}) & \text{if $i \neq k$},
\end{cases}
\qquad 
a^{\prime\prime}_{j} \seq \begin{cases}
-1 & \text{if $i=k$}, \\
\max(0, -b_{ik}) & \text{if $i \neq k$},
\end{cases}
$$
Then the new datum
$\mu_{k} (\Sigma) = ((X_j^\prime)_{j \in J}, \Lambda^\prime, \tB^\prime)$ is given by 
$$
\Lambda^\prime \seq E^{T} \Lambda E, \quad \tB^\prime \seq E \tB F, 
\qquad   
X^\prime_{j} \seq 
\begin{cases}
X^{\bfa^\prime} + X^{\bfa^{\prime\prime}} & \text{if $j =k$}, \\
X_{j} & \text{if $j \neq k$}.
\end{cases}
$$
One can show that the datum $\mu_{k} (\Sigma)$ actually defines a quantum seed,
which is called \emph{the mutation of $\Sigma$ in direction $k$}.
This operation is involutive, i.e.,~we have $\mu_{k}(\mu_{k}(\Sigma)) = \Sigma$.

\begin{Def} 
Let $(\Lambda, \tB)$ be a compatible pair 
and $\Sigma = ((X_j)_{j \in J}, \Lambda, \tB)$ the associated quantum seed.  
\emph{The quantum cluster algebra} $\cA_{v}(\Lambda, \tB)$
is the $\Z[v^{\pm 1/2}]$-subalgebra of the skew field $\F(\cT(\Lambda))$
generated by all the quantum cluster variables 
in the quantum seeds obtained from $\Sigma$ by any sequence of mutations.
\end{Def}    

\begin{Thm}[The quantum Laurent phenomenon {\cite[Corollary 5.2]{BZ05}}]
The quantum cluster algebra $\cA_{v}(\Lambda, \tB)$ is contained in the quantum torus $\cT(\Lambda)$.
\end{Thm}


\end{document}